\documentclass[reqno,10pt]{amsart}
\usepackage{amsmath}
\usepackage{amssymb}
\usepackage{amsfonts,dsfont}
\usepackage{comment}
\usepackage{graphicx}
\usepackage[abs]{overpic}
\usepackage{caption}
\usepackage{float}
\usepackage{physics}
\usepackage{esint} 
\usepackage{braket} 
\usepackage{tikz}

\usepackage{xcolor}
\definecolor{blue_links}{RGB}{13,0,180} 

\usepackage{hyperref}
\hypersetup{
	colorlinks=true, 
	linktoc=all,     
	linkcolor=blue_links,  
	citecolor=blue_links,
	urlcolor=blue_links,
}

\usepackage{enumitem}		

\usepackage[backend=biber,style=alphabetic]{biblatex}		

\renewbibmacro{in:}{%
	\ifentrytype{article}{}{\printtext{\bibstring{in}\intitlepunct}}}
\bibliography{references}

\AtEveryBibitem{%
	\clearfield{issn} 
	\clearfield{doi} 
	
	\ifentrytype{online}{}{
		\clearfield{url}
	}
}

\usepackage[capitalise,noabbrev]{cleveref}	

 \usepackage[final]{showlabels}							

\usepackage{tikz}
\usetikzlibrary{decorations.pathreplacing,calc}
\def\centerarc[#1](#2)(#3:#4:#5)
{ \draw[#1] ($(#2)+({#5*cos(#3)},{#5*sin(#3)})$) arc (#3:#4:#5); }

\usepackage{mathtools}

\newtheorem{theorem}{Theorem}[section]
\newtheorem{lemma}[theorem]{Lemma}

\newtheorem{corollary}[theorem]{Corollary}
\newtheorem{definition}[theorem]{Definition}

\newtheorem{remark}[theorem]{Remark}

\newtheorem*{theorem*}{Theorem}

\newcommand{\R}{\mathbb{R}}
\newcommand{\C}{\mathbb{C}}
\newcommand{\CC}[1]{\C #1:#1}

\def\eps{\varepsilon}

\def\weaklystar{\stackrel{*}{\rightharpoonup}}

\def\XXint#1#2#3{{\setbox0=\hbox{$#1{#2#3}{\int}$}
\vcenter{\hbox{$#2#3$}}\kern-.5\wd0}}

\usepackage{color}

\numberwithin{equation}{section}

\DeclareMathOperator{\energyNonlin}{{E_\eps}}
\DeclareMathOperator{\energyLin}{{E}}
\DeclareMathOperator{\energyNonlinTot}{{\mathcal{E}_\eps}}
\DeclareMathOperator{\energyLinTot}{{ \mathcal{E}}}

\newcommand{\diff}{\mathrm{D}}
\newcommand{\lm}{\mathcal{L}}
\newcommand{\hm}{\mathcal{H}}

\newcommand{\sporth}{\mathrm{SO}}
\newcommand{\intpcs}{\mathcal{P}^{\mathrm{int}}}

\newcommand{\lp}{\mathrm{L}}
\newcommand{\wkp}{\mathrm{W}}
\newcommand{\cont}{\mathrm{C}}
\newcommand{\bv}{\mathrm{BV}}

\newcommand{\gsbvtwotwo}{\mathrm{GSBV}_{ 2 }^{ 2 }}
\newcommand{\gsbdtwo}{\mathrm{GSBD}^{ 2 }}
\newcommand{\sbv}{\mathrm{SBV}}
\newcommand{\gsbd}{\mathrm{GSBD}}
\newcommand{\gsbv}{\mathrm{GSBV}}
\newcommand{\sbd}{\mathrm{SBD}}

\DeclarePairedDelimiterX{\inner}[2]{\langle}{\rangle}{#1 \, , \, #2}
\DeclarePairedDelimiterX{\seminorm}[1]{[}{]}{#1}

\begin{document} 

\title[Linearization of a quasistatic evolution in fracture]{ Linearization of quasistatic fracture  evolution \\ in brittle materials}
\author[M. Friedrich]{Manuel Friedrich} 
\address[Manuel Friedrich]{Department of Mathematics, Friedrich-Alexander Universit\"at Erlangen-N\"urnberg. Cauerstr.~11,
    D-91058 Erlangen, Germany}
\email{manuel.friedrich@fau.de}
\author[P. Steinke]{Pascal Steinke}
\address[Pascal Steinke]{Institute for Applied Mathematics, University of Bonn, Endenicher Allee 60, 53115 Bonn,
Germany}
\email{steinke@iam.uni-bonn.de}
\author[K. Stinson] {Kerrek Stinson} 
\address[Kerrek Stinson]{Department of Mathematics, University of Utah, Salt Lake City, UT 84112 USA}
\email{kerrek.stinson@utah.edu}

\subjclass[2020]{49J45, 70G75,   74B10, 74B20, 74G65, 74R10.}
\keywords{ Brittle materials, variational fracture, free discontinuity problem, quasistatic fracture evolution, linearization,   $\Gamma$-convergence.}

\begin{abstract}   
We prove a linearization result for  quasistatic fracture evolution in nonlinear elasticity. As the stiffness of the material tends to infinity, we show that  rescaled displacement fields  and their associated crack sets converge to a solution of quasistatic crack growth in linear elasticity without any a priori assumptions on the geometry of the crack set. This result corresponds to the evolutionary counterpart of the static linearization result  \cite{friedrich_griffith_energies_as_small_strain_limit_of_nonlinear_models_for_nomsimple_brittle_materials},  where  a Griffith model for nonsimple brittle materials has been considered featuring an elastic energy which also  depends suitably on the second gradient of the deformations. The proof relies on a careful study of unilateral global  minimality, as determined by the nonlinear evolutionary problem, and its linearization together with a variant of the jump transfer lemma in $\gsbd$ \cite{friedrich_solombrino_quasistatic_crack_growth_in_2d_linearized_elasticity}.
\end{abstract}

\maketitle


\section{Introduction}

 A crucial question for  nonlinear models in materials science consists in  establishing  the range of validity of suitably linearized approximations. Indeed, large strain models are often challenging to treat, both analytically and numerically, due to their nonconvex behavior whereas linearized models are significantly easier and, in many situations, still allow to predict accurately observed phenomena.  The last decades have witnessed remarkable progress in providing rigorous results relating models with  different strain regimes  in terms of $\Gamma$-convergence \cite{DalMasoBook}. After the seminal work on elastostatics  \cite{DalMasoEtAl_2002}, this question was explored towards various directions, among others, for   incompressible materials models \cite{JesenkoSchmidt21Geometric, MaininiPercivale21Linearization},  atomistic settings \cite{braides_solci_vitali_derivation_of_linear_elastic_energies_from_pair_interaction_atomistic_systems, schmidt-lin}, residually stressed materials \cite{paroni},  problems without Dirichlet boundary conditions \cite{MaorMora21Reference}, homogenization \cite{jesschmidt-hom, neukamm}, brittle fracture \cite{Friedrich:15-2, friedrich_griffith_energies_as_small_strain_limit_of_nonlinear_models_for_nomsimple_brittle_materials, almi_davoli_friedrich_non_interpenetration_conditions_in_the_passage_from_nonlinear_to_linearized_griffith_fracture}, plasticity \cite{Zeppieri},  or multiwell energies allowing for phase transformations \cite{alicandro.dalmaso.lazzaroni.palombaro,  DavFri20, AlmRegSol23a, schmidt_linear_gamma_limits_of_multiwell_energies_in_nonlinear_elasticity}. Beyond the static framework, we refer to evolutionary results on elastodynamics \cite{abels}, viscoelasticity \cite{badal, manuel-visco, Oosterhout}, or plasticity \cite{mielke-stefanelli}.

The purpose of this paper is to prove a linearization result for the quasistatic evolution of brittle fracture.
We consider a variational model for fracture introduced by {\sc Francfort and Marigo} \cite{francfortMarigo98} based on  Griffith's idea  that 
crack formation and  propagation is the result  of the competition between an   elastic bulk energy and a surface energy proportional to the size of the crack.   Our goal is to show that, in dimension two, crack growth {in} nonlinear elasticity can be  rigorously approximated by crack growth in linear elasticity. 

Moving within the frame of   variational models for fracture,  we start by introducing the underlying energies. {For a reference domain $\Omega\subseteq \R^2$,} we consider nonlinear energies  of the form
\begin{equation}\label{eqn:introNonlinE}
\energyNonlinTot[y,{\Gamma}] : = \dfrac{1}{\eps^2}\int_\Omega \left( W\left(\nabla y\right)  + \eps^{2\left(1-\beta\right) }|\nabla^2 y|^2\right) \dd{x} + \kappa  \mathcal{H}^1\left({\Gamma}\right),
\end{equation}
where $y\in W^{2,2}_{\rm loc}(\Omega\setminus {\Gamma};\R^2)$ 
is the deformation, $W$ is a frame indifferent stored energy density, $\beta\in \left(0,1\right)$, $ \Gamma  \subseteq \overline{\Omega}$ is the crack set (on which $y$ is allowed to be discontinuous), $\kappa$ denotes the fracture toughness, and $\mathcal{H}^1$ is the Hausdorff measure penalizing the length of the crack. The prefactor  $1/\eps^2$ represents a high stiffness of the material depending on a small parameter $\eps>0$. For technical reasons discussed below, we consider an energy in the frame of  \emph{nonsimple materials}, referring  to the fact that the elastic energy depends additionally on the
second gradient of the deformation. However, as the contribution of the Hessian vanishes in the limit $\eps \to 0$, the effective energy is  described by  a linear Griffith fracture energy without second gradients. Precisely, using $\Gamma$-convergence, in \cite{friedrich_griffith_energies_as_small_strain_limit_of_nonlinear_models_for_nomsimple_brittle_materials} the first author showed that the above energies converge as $\eps \to 0$ to their linearized counterpart
\begin{equation}\label{eqn:introGriffith}
\energyLinTot[u, \Gamma ] : = \int_\Omega \dfrac{ 1 }{ 2 }\CC{e\left(u\right)}\dd{x} +  \kappa  \mathcal{H}^1\left( \Gamma \right),
\end{equation}
where $u\in W^{1,2}_{\rm loc}(\Omega\setminus  \Gamma; \R^2)$ is the displacement approximated by $u \approx (y-{\rm id})/\eps$ (${\rm id}$ is the identity mapping)  and $\mathbb{C} = \diff^2 W\left({\rm Id}\right)$ is a fourth-order elasticity  {tensor.} Whereas the models \eqref{eqn:introNonlinE}--\eqref{eqn:introGriffith} correspond to the so-called \emph{strong formulation}, the  result in \cite{friedrich_griffith_energies_as_small_strain_limit_of_nonlinear_models_for_nomsimple_brittle_materials} is actually formulated in the weak setting of the functions spaces $\gsbv$ (for \eqref{eqn:introNonlinE}) and  $\gsbd$ (for \eqref{eqn:introGriffith}).

In this work, our goal is to enhance the understanding of the relation between  the models \eqref{eqn:introNonlinE}--\eqref{eqn:introGriffith} by showing that solutions of quasistatic crack growth associated with the nonlinear energies converge to ones related to the linear energy, see our main result in  Theorem \ref{thm:main}. This provides a rigorous justification for the fact that crack propagation in  stiff materials (described in terms of $1/\eps^2$) is well approximated by using a linear theory. We also point out that the energetic rescaling in the energy (\ref{eqn:introNonlinE}) may be found by looking at the nonlinear fracture energy ($\eps =1$) on progressively larger domains $\Omega/\eps$, referred to as Ba\v{z}ant's law  \cite{NegriToader_2015}.

Linearization for variational fracture evolutions has already received some attention in the literature. {\sc  Negri and Zanini} \cite{NegriZanini_2014} showed that the nonlinear quasistatic fracture evolution converges to the linear counterpart when the path of the crack is restricted to a line segment. Subsequent work by {\sc Negri and Toader} \cite{NegriToader_2015} showed that this convergence still  holds  under a weaker assumption on the cracks (effectively, the crack is made of finitely many non-intersecting regular arcs). Our conclusions significantly improve these results: we show convergence in the passage from nonlinear-to-linearized models  with \emph{no assumptions} on the geometry of the crack.

 The study  of  {quasistatic crack evolutions}  was initiated in \cite{toader-dalmaso} for a  $2d$-model   with  restrictive assumptions on the crack topology.  A breakthrough existence  result in $\sbv$ is  due to {\sc Francfort and Larsen} \cite{francfort_larsen_existence_and_convergence_for_quasi_static_evolution_in_brittle_fracture}   in the simplified setting of anti-planar displacements (meaning scalar displacements with control on the full gradient). Subsequently, this result was  generalized to nonlinear elasticity \cite{dal_maso_francfort_toader_quasistatic_crack_growth_in_nonlinear_elasticity},  including the setting  of non-interpenetration \cite{lazzaron}.  
In all these works, the quasistatic fracture  evolution   is described in terms of \emph{i)} an irreversible crack path, \emph{ii)} unilateral global minimality of the deformation/displacement, and \emph{iii)} an energy-balance equation. 
Essential to the verification of   \emph{ii)} is the so-called \emph{jump transfer lemma}, allowing the authors to carry over unilateral minimality for approximate free discontinuity problems to their limits, where `unilateral' refers to the irreversible nature of the process. 

Surprisingly, despite its  paramount relevance for realistic engineering applications, the theory of  {linear fracture models} has been considerably  {less developed}. This is due to technical difficulties inherent to the presence of symmetrized gradients. Indeed, due to   the unknown and potentially   rough discontinuity sets and the  {lack of Korn's inequality},  there is no control on the  skew symmetric part $(\nabla u)^{\top}- \nabla u$  of the strain, leading to an analytically more intricate formulation  in  the larger space of  {special   functions of bounded deformation} ($\sbd$) \cite{Ambrosio-Coscia-DalMaso:1997}. Only recently, departing from {\sc Dal Maso}'s seminal paper \cite{dalMasoGSBD} on the generalized space $\gsbd$, an enormous amount of effort has been invested to develop the theory. Technical advances regarding  functional inequalities \cite{cagnettiChambolleScardia22,chambolleContiFrancfort16,friedrich2018-piecewiseKorn},  compactness theorems \cite{AlmiTasso_2023,chambolle_crismale_compactness_and_lsc_in_gsbd}, and lower semicontinuity   \cite{friedrich-perugini, kreisbeck}  have allowed a variety of authors to prove existence of Griffith energy minimizers in the strong formulation \cite{CONTI2019455,chambolleContiIurlano18,chambolleCrismale19} with higher regularity \cite{babadjianIurlanoLemenant22,LabourieLemenant22,FriedrichLabourieStinson23}, or existence of minimizers in anisotropic and heterogeneous settings \cite{chambolleCrismaleEquilibrium,FriedrichLabourieStinson24}. 

Most important for our work is  {the existence proof  of} a solution to the quasistatic evolution for cracks in dimension two \cite{friedrich_solombrino_quasistatic_crack_growth_in_2d_linearized_elasticity}. {In} the linear case, it was essential to develop another {jump transfer} lemma  suited for settings where only control on the symmetric gradient is available. Using a variant of this  {jump transfer} lemma is at the core of our approach dealing with the passage from nonlinear to linearized theory.  

We proceed by describing the ingredients of the proof in more detail. To see that solutions of the nonlinear problem associated with the energy {$\energyNonlinTot$} converge to a solution of linear quasistatic crack growth with energy {$\energyLinTot$}, we prove \emph{i)} irreversibility, \emph{ii)} minimality, and \emph{iii)} the energy balance for the limit displacement {and crack}. First, to identify a limiting displacement we need to ensure that  deformation gradients are sufficiently close to the identity. Given that the domain can fracture into multiple pieces, we are forced to identify the limit up to modification, which essentially transforms the deformation back to the identity by a piecwise  rigid motion related to a  Caccioppoli partition of the reference domain. This technique has been employed already in a variety of works, both in a nonlinear \cite{Friedrich:15-2, friedrich_griffith_energies_as_small_strain_limit_of_nonlinear_models_for_nomsimple_brittle_materials} and linearized setting \cite{chambolleCrismale-Hetero, friedrich_solombrino_quasistatic_crack_growth_in_2d_linearized_elasticity, stinson_wittig_elliptic_approximation_for_phase_separation_in_a_fractured_material}.  The proof of \emph{i)} then follows by design.

The proof of \emph{ii)} is significantly more involved.
 Beyond the {jump transfer} lemma, carrying minimality from nonlinear to linear problems requires a delicate look at the linearization procedure.
 Formally, given $(y_\eps,  \Gamma_\eps  )$ minimizing $\energyNonlinTot$, we can define a limit displacement $ u = \lim_{\eps\to 0}(y_\eps-{\rm id})/\eps$ and crack $\Gamma  = \lim_{\eps\to 0} \Gamma_\eps  $. To show that minimality is inherited by the limit, we want to show that $\energyLinTot[u,\Gamma]\leq \energyLinTot[u+\phi,   \Gamma   \cup   \Gamma_\phi  ]$ for all test pairs $(\phi,  \Gamma_\phi  )$. {For this, we} use a Taylor expansion to linearize the nonlinear minimality condition $\energyNonlinTot [y_\eps,  \Gamma_\eps ] \leq \energyNonlinTot[y_\eps+\phi_\eps,   \Gamma_\eps \cup \Gamma_{\phi_\eps}   ]$ for test pairs $(\phi_\eps, \Gamma_{\phi_\eps}  )$ approximating a fixed $(\phi,  \Gamma_\phi  )$. The challenge is that, within this expansion, we must carefully control the error on the right-hand side to ensure it disappears in the limit $\eps\to 0$. 
 The key ingredients are a quantitative description of the almost-minimality of the rescaled displacement $(y_\eps-{\rm id})/\eps$, see Lemma \ref{lemma:linearization_of_minimization_problem}, the jump transfer lemma of  \cite{friedrich_solombrino_quasistatic_crack_growth_in_2d_linearized_elasticity}, and a density result for $\gsbdtwo$ functions with prescribed boundary values, see \Cref{thm:density_with_boundary_values}.  
 
 The proof of the energy balance  \emph{iii)} follows the by now classical idea which relies on  global minimality and a Riemann sum argument. Yet, in our setting the expression representing the work of the external loads is more intricate if there are regions determined by the abovementioned Caccioppoli partition where the deformation is not close to the identity but to a different rigid motion. In this case, we will show that this piece is (almost) relaxed to an equilibrium configuration and thus  no  longer contributes  to the change of the total energy. Rigorously, this is done by using our analysis from \emph{ii)} to derive an approximate variational inequality that can pass to the limit in the time-integrated energy-balance equation, see \Cref{cor:approx_euler_lagrange} and \Cref{cor:approx_euler_lagrange_consequence}.

As a technical point on our result, we introduce a time step parameter $\tau>0$ (later encoded  as  $\Delta_\eps$ in equation (\ref{eqn:deltaEps}))  and a `stiffness' parameter $\eps>0$. For $\eps$ fixed, we use $\tau$ to construct a time-discretized approximate solution to the nonlinear quasistatic crack growth. Then, we pass to the limit jointly in $\tau\to 0$ and $\eps\to 0$ to derive a solution of linear quasistatic crack growth. This means that at no point we actually use a true solution in the nonlinear setting. (Note that existence of such is not guaranteed by \cite{dal_maso_francfort_toader_quasistatic_crack_growth_in_nonlinear_elasticity} due to the presence of the second-gradient term in \eqref{eqn:introNonlinE}.) Importantly, however, in our analysis, there is no constraint on the relation between $\tau$ and $\eps$. Thus, one can  expect that we can pass to the limit {as} $\tau\to 0$ first and then $\eps\to 0$. Although we believe that this would in general be possible, the  rigorous execution of this procedure is beyond the scope of this paper.

Our result comes with three reasonable caveats: First, to isolate a rotation about which we may linearize, we require   the nonlinear energy to be nonsimple and penalize the Hessian of the deformation. Yet, we emphasize that this penalization is small and vanishes  in the limit $\eps\to 0$  (see \Cref{cor:convergence of energies}). A possible alternative would be to modify  the geometric  rigidity estimate by {\sc Friesecke et al.} \cite{friesecke_james_mueller_geometric_rigidity}, but this is intricate in this setting as the domain regularity is determined by the a priori irregular crack set. Existing applications using \cite{friesecke_james_mueller_geometric_rigidity} for fracture require fixed crack geometries \cite{NegriToader_2015}, a curvature penalization of the crack \cite{friedrich_kreutz_zemas_derivation_of_effective_theories_for_thin_nonlinearly_elastic_rods_with_voids}, or  an involved modification procedure of deformations and crack sets on multiple scales \cite{Friedrich:15-2}. In particular, the latter modification is incompatible with the irreversibility condition   \emph{i)}.

 Second, our result is restricted to the plane. This is essentially because the piecewise Korn inequality \cite{friedrich2018-piecewiseKorn} has only been proven in the plane. We use this inequality in two places: the jump transfer lemma and a density result with Dirichlet conditions. With more work, we believe that it is possible to extend  the density result to higher dimensions as the role of the  piecewise Korn inequality  therein is to approximate $\gsbdtwo$ by $\lp^2$-functions. On the other hand, extending the {jump transfer} lemma to higher dimensions will require fundamentally new ideas.
 
Third,  our model does not account for non-interpenetration and, in particular, for simplicity we require  $W$ to be locally Lipschitz, preventing the blow-up of $W(M)$ as ${\rm det} \, M\to 0$.  The necessary techniques to implement non-interpenetration in the framework of crack growth in finite elasticity are already available \cite{lazzaron} (based on \cite{mielke-francfort}) and, in the static case, also the linearization has recently been studied \cite{almi_davoli_friedrich_non_interpenetration_conditions_in_the_passage_from_nonlinear_to_linearized_griffith_fracture}. The extension to the evolutionary setting will require nontrivial adaptations  and is a subject  of future research.   

The paper is organized as follows. In \Cref{sec: main results}, we introduce the model and formulate the main result. \Cref{sec:mathPrelim} is subject to some preliminaries, namely compactness and density results in $\gsbd$ and the jump transfer lemma from \cite{friedrich_solombrino_quasistatic_crack_growth_in_2d_linearized_elasticity}. In \Cref{sec:compactness} we derive a priori estimates for the nonlinear evolutions and obtain compactness for the rescaled displacement fields. Then,  \Cref{sec:approximateRelationsMinimality}  and \Cref{sec:approximateRelationsBalance} are devoted to the stability of minimality and an approximate energy balance as $\eps\to 0$. \Cref{sec:completeProof} contains the proof of the main result. Eventually, in the Appendix \ref{sec:appendix} and \ref{sec:appendix_proofs} we collect proofs that have been omitted in the paper.

\section{Setting and main results}\label{sec: main results}

 In this section, we  first provide  some basic notation. Then,   we introduce the nonlinear quasistatic time discretized evolution for crack growth using a time discrete scheme.  Next,  we define a notion of solution for quasistatic crack growth in linear elasticity.  Eventually, we state our main convergence   result relating the nonlinear to the linear setting.   


\subsection{Basic notation}

We introduce some basic notation:  
\begin{itemize}

\item We  let  $ \Omega \subseteq  \Omega' \subseteq\mathbb{ R }^{ 2 } $ be bounded Lipschitz domains {such that  also  $\Omega'\setminus \overline{\Omega}$ is  a Lipschitz  set.} 
	
\item $B_{ r } ( x )   \subseteq  \R^d  $, $d \ge 2$, denotes the  ball {centered at $x  \in \R^d$  with radius $r>0$}. {Similarly, for a set $A$, we define $B_\delta (A) :  =\{x \in \R^d: |x-y|  <\delta \text{ for some }y\in A\}.$}

\item We write ${\rm id}$ for the identity map on $ \mathbb{ R }^{ 2 } $ and ${\rm Id}$ for its derivative, viewed as an element of $ \mathbb{ R }^{ 2\times 2  } $. 

 \item We denote by  $\mathbb{ R }^{ 2 \times 2}_{ \mathrm{sym} }$ and $\mathbb{ R }_{ \mathrm{skew } }^{ 2 \times 2 }$ symmetric and skew-symmetric matrices, respectively. We set $\sporth(2) = \lbrace F \in \R^{2 \times 2} \colon F^\top F = {\rm Id} \rbrace$.  

\item We  call  an affine map $ A x +b $ {with} $ A \in \mathbb{ R }_{ \mathrm{skew } }^{ 2 \times 2 } $ and $ b \in \mathbb{ R }^{ 2 } $   an \emph{infinitesimal rigid motion}.

\item  $\chi_{ A }$ denotes the characteristic function of a set $ A   \subseteq  \R^2 $.

\item We denote  the two-dimensional Lebesgue measure and the one-dimensional Hausdorff measure by $\mathcal{L}^2$ and $\mathcal{H}^1$, respectively.

\item We write $ \partial^{ \ast  } $ for the essential boundary of a set taken always with respect to the set $ \Omega' $, see \cite[Def.~3.60]{ambrosio_fusco_pallara_functions_of_bv_and_free_discontinuity_problems}.

\item We say that a partition $ (P_{ j })_{ j\in \mathbb{ N } } $  of  $ {\Omega'} \subseteq \mathbb{ R }^{ 2 } $ is a \emph{Caccioppoli partition} of $ {\Omega'} $ if
$ \sum_{ j } \hm^{ 1 } \left( \partial^{ \ast } P_{ j } \right) < \infty $.

\item We assume the reader to be familiar with the function spaces $ \bv $ (\cite[Def.~3.1]{ambrosio_fusco_pallara_functions_of_bv_and_free_discontinuity_problems}), $ \sbv $ (\cite[Sct.~4.1]{ambrosio_fusco_pallara_functions_of_bv_and_free_discontinuity_problems}), $ \gsbv $ (\cite[Def.~4.26]{ambrosio_fusco_pallara_functions_of_bv_and_free_discontinuity_problems}), $ \gsbd $ (\cite[Def.~4.2]{dalMasoGSBD}),  and  $ \gsbdtwo $ (\cite{iurlano_density}). We further define (\cite[Sct.~2.1]{friedrich_griffith_energies_as_small_strain_limit_of_nonlinear_models_for_nomsimple_brittle_materials}) 
 $$ {\gsbvtwotwo(\Omega;\R^2) : = \{y\in \gsbv^2(\Omega;\R^2): \nabla y\in \gsbv^2(\Omega;\R^{2\times 2})\} }.$$

\item We use the notation $ \inner*{\cdot}{\cdot} $ for the  Euclidean  inner product  on $ \mathbb{ R }^{ 2 } $, $ \mathbb{ R }^{ 2 \times 2 } $, and  $\mathbb{ R }^{ 2 \times 2 \times 2  \times 2 }$, respectively.  When   an elastic tensor $ \mathbb{C} \in \R^{2 \times 2 \times 2 \times 2} $  is involved, {we follow  the convention to  use} the notation $ \colon $ for the inner product between two matrices in $ \mathbb{ R }^{ 2 \times 2 } $.

\item If we do not specify the domain over which a norm is taken, we always implicitly mean $ \Omega' $, i.e., $ \Vert \cdot \Vert_{\lp^{ 2 }} \coloneqq  \Vert \cdot \Vert_{\lp^{ 2 } ( \Omega' )}$. {Likewise, we often do not specify the range of a function within the norm.}

\item For $ a, b \in \mathbb{ R }^{ 2 } $, we define $ a \odot b \coloneqq ( a \otimes b + b \otimes a ) / 2  $ as the symmetric  outer product   of $ a $ and $ b $.

\item For a vector valued function $ f $, we write $ f^{ i } $ for the $ i$-th component of $ f $.

\item We use the symbol $\lesssim$ to say that the inequality $ \leq $ holds up to a constant which may depend on the potential, the boundary data, and the domain.   We follow the convention that generic constants  $C$ may change from line to line.

\end{itemize}

\subsection{Nonlinear quasistatic fracture evolution}
\label{sct:setup_and_relaxation}
We begin by constructing a time discretized evolution for a nonlinear model based on the energy (\ref{eqn:introNonlinE}),  see e.g.\  \cite{dal_maso_francfort_toader_quasistatic_crack_growth_in_nonlinear_elasticity} or \cite{francfort_larsen_existence_and_convergence_for_quasi_static_evolution_in_brittle_fracture}. Let $(\eps_i)_i$ be an arbitrary sequence with $\eps_i\to 0$ as $i\to \infty$. For convenience, with an abuse of notation, we {will} simply use $\eps$ in place of $\eps_i$ and  write  $\eps \to 0$. Taking our time interval to be $[0,1]$ without loss of generality, we partition the time interval using the collection of points
\begin{equation*}
	I_{ \varepsilon }
	:= ( t_{ i }^{ \varepsilon })_{i=0}^{n(\eps)} = 
	\left\{
	0=t_{ 0 }^{ \varepsilon } < t_{ 1 }^{ \varepsilon } < \cdots < t_{ n(\varepsilon ) }^{ \varepsilon } = 1 \right\}.
\end{equation*}
We impose that the partitions created by $I_\eps$ are nested, i.e., $I_{\eps} \subseteq I_{\eps'}$ when $\eps' <\eps$, and that they refine, in the sense that
\begin{equation}\label{eqn:deltaEps}
	\Delta_{ \varepsilon }
	\coloneqq
	\max_{ 1 \leq i \leq n(\varepsilon ) }
	t_{ i }^{ \varepsilon } - t_{ i - 1 }^{ \varepsilon }
\end{equation}
vanishes as $ \varepsilon $ tends to zero. Note that the union of the partitions $ I_{ \infty } \coloneqq \bigcup_{ \varepsilon } I_{ \varepsilon } $ is a countable dense subset of $ [ 0 , 1 ] $. We also introduce the shorthand notations
\begin{equation*}
	I_{ \varepsilon }^{ t }\coloneqq \left\{ \tau \in I_{ \varepsilon } \colon \ \tau \leq t \right\}
	\text{ and }
	I_{ \infty }^{ t } \coloneqq \left\{ \tau \in I_{ \infty } \colon \ \tau \leq t \right\}.
\end{equation*} 
 We let $ \Omega  \subseteq\mathbb{ R }^{ 2 } $ be a bounded Lipschitz domain. As is typically done for fracture problems, we impose displacement boundary conditions by considering a larger set $\Omega '$ containing $\Omega$: we fix  a boundary datum 
\begin{equation}\label{eqn:bdryData}
	h \in \wkp^{ 1 , 1}\big( (0, 1 ) {;} \wkp^{ 2 , 2 } ( \Omega' ; \mathbb{ R }^{ 2 }) \big)
	\cap
	\lp^{ \infty } \big( (0, 1 ){;} \wkp^{ 2 , \infty } (\Omega'; \mathbb{ R }^{ 2 } ) \big),
\end{equation}
and define 
\begin{equation}
	\label{def:sneps}
	\mathcal{S }_{ n }^{ \varepsilon }
	\coloneqq
	\big\{
	y \in \gsbvtwotwo(\Omega'; \mathbb{ R }^{ 2 } ) \,\colon\,
	y = \mathrm{id} + \varepsilon h_{ n }^{ \varepsilon }  \text{ on } \Omega'\setminus \overline{\Omega }
	\big\},
\end{equation}
where  $ h_{ n }^{ \varepsilon } \coloneqq h(t_{ n }^{ \varepsilon } ) $.   We assume that also $\Omega'$ and  $\Omega'\setminus \overline{\Omega}$ are Lipschitz sets. 

Following  \cite[Sct.~2.1]{friedrich_griffith_energies_as_small_strain_limit_of_nonlinear_models_for_nomsimple_brittle_materials}, we define the {energy} 
\begin{equation}\label{eqn:existyMini}
	\begin{dcases*}
		\dfrac{1}{\varepsilon^{ 2 } }
		\int_{ \Omega' } W ( \nabla y )\dd{ x }
		+
		\dfrac{ 1 }{ \varepsilon^{ 2 \beta } }
		\int_{ \Omega' }
		\abs{ \nabla^{ 2 } y }^{ 2 }
		\dd{ x }
		+
		\kappa  \hm^{ 1} ( J_{ y } ),
		& if $ J_{ \nabla y } \subseteq J_{ y } $,
		\\
		\infty,
		& else
	\end{dcases*}
\end{equation}
for {$y\in \gsbvtwotwo (\Omega' ;  \mathbb{ R }^{ 2 })$ and} some $ \beta \in (2/3,1) $, $\kappa >0$. This corresponds to the weak formulation of the energy introduced in equation (\ref{eqn:introNonlinE}). Here,   $ W \colon \mathbb{ R }^{ 2 \times 2 } \to [0, \infty) $ satisfies  
\begin{enumerate}[labelindent=0pt,labelwidth=\widthof{\ref{eq:W_zero_on_sod}},label=(W\arabic*),ref=(W\arabic*),itemindent=0em,leftmargin=!]
	\item 
	\label{eq:W_growth_assumptions}$ W $ is locally Lipschitz continuous with $$ \abs{ W ( A ) - W ( B ) } \leq C ( 1 + \max\{ \abs{A} , \abs{ B} \} ) {|A-B|}$$ for some $ C > 0 $, and {there is $r>0$ such that $W$ is $ \cont^{3 } $ in the neighborhood $B_{2r}(\sporth(2))$ of $ \sporth ( 2 ) $.}
	\item  \label{eq:frami} Frame indifference: $ W (RF)=W(F ) $ for all $ F \in \mathbb{ R }^{ 2 \times 2 } , R \in \sporth(2) $.
	\item \label{eq:W_zero_on_sod}$ W ( F ) \geq c \mathrm{dist}^{ 2 } ( F , \sporth( 2 ) ) $ for some $c>0$ for all $ F \in \mathbb{ R }^{ 2 \times 2 } $, $ W(F) = 0 $ if and only if $ F \in \sporth( 2 ) $. 
\end{enumerate}
Note that the local Lipschitz continuity in \ref{eq:W_growth_assumptions} is a consequence of the common assumption $ \abs{ \diff W ( A ) } \lesssim ( 1 + \abs{ A } ) $.  We emphasize that \ref{eq:W_growth_assumptions} excludes $W(A) \to + \infty $ as ${\rm det} \, A\to 0$.  

Considering a continuous  function $y$ on $\R^2$ that is  $C^2$ in $\{x_2> 0\}$ and $\{x_2< 0\}$ and satisfies $J_y = \emptyset$ and $J_{\nabla y} = \lbrace x_2 = 0 \rbrace \cap \Omega$,   one can approximate this by ${y_\eta := y + \eta \, {\rm id}}  \chi_{\{x_2>0\}}  $ for small ${\eta>0}$ to see that, as $\eta \to 0$, the condition $ J_{ \nabla y_\eta } \subseteq J_{ y_\eta } $ is not stable under convergence in measure. Thus, for existence of minimizers of the energy (\ref{eqn:existyMini}), we must pass to its relaxation
\begin{equation}\label{relaxed energy}
	\energyNonlin ( y )
	\coloneqq
	\dfrac{1}{\varepsilon^{ 2 } }
	\int_{ \Omega' } W ( \nabla y )\dd{ x }
	+
	\dfrac{ 1 }{ \varepsilon^{ 2 \beta } }
	\int_{ \Omega' }
	\abs{ \nabla^{ 2 } y }^{ 2 }
	\dd{ x }
	+
	\kappa  \hm^{ 1} ( J_{ y } \cup J_{ \nabla y } ),
\end{equation}
see \cite[Prop.~2.1]{friedrich_griffith_energies_as_small_strain_limit_of_nonlinear_models_for_nomsimple_brittle_materials}.  Starting  from  initial conditions $ y_{ 0 }^{ \varepsilon } \in \mathcal{S }_{ 0 }^{ \varepsilon } $ with bounded energy, i.e.,
\begin{equation}\label{eqn:wellPrepInitial}
 \sup_{ \varepsilon > 0 } \energyNonlin ( y_{ 0 }^{ \varepsilon } ) < \infty,
 \end{equation} 
assuming that   $y_{k }^{ \varepsilon } \in \mathcal{S }_{ k }^{ \varepsilon } $  for $k=0,\ldots,n-1$ have already been found, we iteratively define $ y_{ n }^{ \varepsilon } $ as a minimizer of the energy 
\begin{equation}
	\label{eq:discrete_min_problem}
	\dfrac{1}{\varepsilon^{ 2 } }
	\int_{ \Omega' } W ( \nabla y )\dd{ x }
	+
	\dfrac{ 1 }{ \varepsilon^{ 2 \beta } }
	\int_{ \Omega' }
	\abs{ \nabla^{ 2 } y }^{ 2 }
	\dd{ x }
	+
	\kappa \hm^{ 1} \left( \left( J_{ y } \cup J_{ \nabla y } \right) \setminus \bigcup_{ k = 0 }^{ n-1 } \left( J_{ y_{ k }^{ \varepsilon } } \cup J_{ \nabla y_{ k }^{ \varepsilon } } \right) \right)
\end{equation}
over $y \in \mathcal{S }_{ n }^{ \varepsilon }$. We emphasize that the integrals in the above minimization could be over $\Omega$ or $\Omega'$ due to the boundary conditions, but the jump sets are always taken relative to $\Omega'.$
The existence of a minimizer for the problem \eqref{eq:discrete_min_problem} follows from an adaptation of \cite[Thm.~2.2]{friedrich_griffith_energies_as_small_strain_limit_of_nonlinear_models_for_nomsimple_brittle_materials}, which gives existence of minimizers on an open set. Precisely, the only difference is that we  remove  the energy of the crack from the previous time steps. One can account for this difference by approximating $ \bigcup_{ k=0 }^{ n - 1 }  (J_{ y^\eps_{ k } } \cup J_{ \nabla y^\eps_{ k } }) $ from inside with compact sets $K$ and applying the compactness  result  \cite[Thm.~3.2]{friedrich_griffith_energies_as_small_strain_limit_of_nonlinear_models_for_nomsimple_brittle_materials} on $\Omega'\setminus K$ and then diagonalizing as $K\nearrow \bigcup_{ k=0 }^{ n - 1 } (J_{ y^\eps_{ k } } \cup J_{ \nabla y^\eps_{ k } }) $. The technique used here is similar to the proof of \cite[Thm.~2.8]{dal_maso_francfort_toader_quasistatic_crack_growth_in_nonlinear_elasticity}.


Lastly, we define the \emph{discrete quasistatic  fracture  evolution} $ y_{ \varepsilon } \colon [0,1] \times  \Omega'  \to \R^2 $ with parameter $ \varepsilon > 0 $ as the piecewise constant interpolation of the   minimizers, namely   
\begin{equation}\label{eqn:minMovDeform}
 y_{ \varepsilon } ( t ) =   y_{ n }^{ \varepsilon }   \text{ for } t \in [t_{ n }^{ \varepsilon }, t_{ n + 1 }^{ \varepsilon } ) \quad \text{  for each $0 \le n \le n(\eps)-1$, }  \quad y_{ \varepsilon } (1 ) = y_{ n(\eps) }^{ \varepsilon }.
 \end{equation}
Note that, if $ h_{ \varepsilon } $ denotes the piecewise constant interpolation of $ h $ with respect to $ I_{ \varepsilon } $, which means
\begin{equation*}
	h_{ \varepsilon } ( t ) = h ( t_{ n }^{ \varepsilon } ) \text{ for } t \in [t_{ n }^{ \varepsilon } , t_{ n + 1 }^{ \varepsilon } ),
\end{equation*}
then $ y_{ \varepsilon } $ satisfies the boundary condition $
	y_{ \varepsilon } ( t ) = h_{ \varepsilon } ( t ).$ 
Lastly, we introduce the shorthand notation
\begin{equation}\label{eqn:GammaEps}
	\Gamma_{ \varepsilon } ( t ) 
	\coloneqq 
	\bigcup_{ \substack{\tau \in I_{ \varepsilon }^t  } }
	J_{ y_{ \varepsilon } ( \tau ) } \cup J_{ \nabla y_{ \varepsilon } ( \tau ) }
\end{equation}
{for the total  crack   at time $t.$}
Note then that, as a consequence of the minimization problem (\ref{eq:discrete_min_problem}), for all $t \in [0,1]$, $ y_{ \varepsilon }  (t)  $ is a minimizer of the functional
\begin{equation}
	\label{eq:min_problem_wrt_own_jumpset}
	\dfrac{ 1 }{ \varepsilon^{ 2 } }
	\int_{ \Omega' }
	W ( \nabla z )
	\dd{ x }
	+
	\dfrac{ 1 }{ \varepsilon^{ 2 \beta } }
	\int_{ \Omega' }
	\abs{ \nabla^{ 2 } z }^{ 2 }
	\dd{ x }
	+
	\kappa \hm^{ 1 } \big(
	\left(
	J_{ z } \cup J_{ \nabla z }
	\right)
	\setminus
	\Gamma_\eps  ( t ) 
	\big)
\end{equation}
among all $ z \in \gsbvtwotwo  ( \Omega' ; \mathbb{ R }^{ 2 }  ) $ with $ z = h_{ \varepsilon } ( t ) $ on $ \Omega' \setminus \overline{ \Omega } $.

\subsection{Linear quasistatic fracture evolution}

We introduce the notation
\begin{equation}\label{eq: QQQ}
	\mathbb{ C } \coloneqq \diff^{ 2 } W ( \mathrm{Id} )
	\quad \text{ and } \quad
	Q ( A ) \coloneqq  \mathbb{ C } A  \colon A,
\end{equation}
and  note that $\mathbb{C}$ is  a symmetric positive-definite fourth-order tensor by  \ref{eq:frami} and \ref{eq:W_zero_on_sod}. We define  the linearized energy  as
\begin{equation}\label{linen}
	\energyLin  ( u )
	\coloneqq
	\int_{ \Omega' }
	\dfrac{ 1 }{2}
	Q ( e ( u ) ) 
	\dd{ x }
	+
	\kappa
	\hm^{ 1 } ( J_{ u } ).
\end{equation}
As shown in  \cite{friedrich_griffith_energies_as_small_strain_limit_of_nonlinear_models_for_nomsimple_brittle_materials},  $\energyLin$ is the $\Gamma$-limit of $\energyNonlin$ as $\eps \to 0$ for a suitable notion of convergence of the deformations.  Based on this energy, we now introduce the notion of a  quasistatic fracture evolution  in the setting of linear elasticity,  see also \cite[Thm.~3.1]{friedrich_solombrino_quasistatic_crack_growth_in_2d_linearized_elasticity}. As before, we take the time interval to be $[0,1]$.

\begin{definition}\label{def:linearQuasistatic}
Let $h$ be as in \eqref{eqn:bdryData}  and let $u_0 \in  \gsbd^2(\Omega') $ with $u_0 = h(0)$ on $\Omega' \setminus \overline{\Omega}$.  We say that   displacements  $t \mapsto u(t)\in  \gsbdtwo(\Omega')$ and  cracks  $t \mapsto\Gamma(t)\subseteq \Omega'$ are a  (linear) quasistatic fracture evolution with the boundary condition $h$ and initial condition $ u_{ 0 } $ if the following conditions are satisfied:
\begin{enumerate}[labelindent=0pt,labelwidth=\widthof{\ref{item:energy_balance_def}},label=(\roman*),ref=(\roman*),itemindent=0em,leftmargin=!]
\item \label{item:initial_condition_quasistatic_fracture_evolution}\emph{(Initial condition).} The displacement $u(0)= u_{ 0 }$ minimizes \eqref{linen} among all $v \in \gsbdtwo(\Omega')$ with $v = h(0)$ on $\Omega' \setminus \overline{\Omega}$, and $\Gamma(0) = J_{u(0)}$.  
\item \label{item:displacement_and_boundary_conditions}\emph{(Displacement and boundary conditions).} 
The displacement satisfies $e(u) \in \lp^\infty({(0,1)};\lp^2(\Omega';  \mathbb{R}^{2\times 2}_{\rm sym})) $, and at each time  it holds that  $J_{u(t)}\subseteq \Gamma(t)$ up to an $\mathcal{H}^1$-null set and $u(t) = h(t)$ in $\Omega'\setminus \overline{\Omega}$.

\item \label{item:irreversibilty }\emph{(Irreversibility).}  It holds that $\Gamma(t)\subseteq \Gamma(t')$ if $t \leq t'$.

\item \label{item:minimality}\emph{(Minimality).} The displacement is a global minimizer of the elastic energy given the existing crack, in the sense that
\begin{equation}\nonumber
\int_{\Omega'}\dfrac{1}{2}\CC{e(u(t))}\dd{x} \leq \int_{\Omega'} \dfrac{1}{2}\CC{e(v)}\dd{x} +   \kappa  \mathcal{H}^1(J_v\setminus \Gamma(t))
\end{equation}
for all $v\in \gsbdtwo(\Omega')$ with $v = h(t)$ in $\Omega'\setminus \overline{\Omega}.$
\item \label{item:energy_balance_def}\emph{(Energy balance).} The change in energy is equal to the work done on the system as captured by 
\begin{align}
\notag
\energyLinTot(t)   & = \energyLinTot(0)  + \int_0^t\int_{\Omega'}\C e(u(s)):e( \partial_t h (s))\dd{x}\dd{s},
\shortintertext{where by }
\label{totalenergy}
	\energyLinTot ( t )
	& \coloneqq
	\int_{ \Omega' }
	\dfrac{ 1 }{ 2 }
	Q \big( e ( u ( t ) ) \big) \dd{ x }
	+
	\kappa
	\hm^{ 1 } \left( \Gamma(t) 
	\right)
\end{align}
we denote the  total energy at time $t \in [0,1]$.
\end{enumerate}
\end{definition}

We remark that existence of linear quasistatic fracture evolutions for the Griffith energy  has been shown in \cite{friedrich_solombrino_quasistatic_crack_growth_in_2d_linearized_elasticity} by means of time-discrete approximations (analogous to \eqref{eq:discrete_min_problem}).  

\subsection{Main result}
To describe the convergence of the deformations $y_\eps$ to a displacement $u$,  we need to introduce a proper type of convergence.  Typically,  in linearization results,    one considers the asymptotic behavior of the rescaled displacements $ \bar{u}_\eps  \coloneqq (y_\eps -{\rm id})/\eps.$ However, as the material may fracture, it is not possible to linearize around the single rigid motion given by ${\rm id}$    as pieces may be broken {off}. On these pieces, compactness cannot be expected and we essentially have no information except for the fact that the  elastic energy should vanish asymptotically due to {the} minimality of $y_\eps$ in the sense of \eqref{eq:discrete_min_problem}.

Let us now describe the convergence towards a limit $u(t) \in \gsbdtwo(\Omega')$ {for} $t \in [0,1]$. {For a crack $\Gamma(t) \subseteq  \overline{\Omega} \cap \Omega' $ with $\mathcal{H}^1(\Gamma(t))<\infty$, by} $B(t) \subseteq \Omega$ we denote the largest set of finite perimeter (with respect to set inclusion) which satisfies $  \partial^* B(t) \subseteq \Gamma(t)  $ up to an $\mathcal{H}^1$-negligible set (see \Cref{lemma:characterization_of_bad_set} for existence and uniqueness, and recall that $\partial^*$ is with respect to $\Omega'$). This set represents the `broken off pieces,' and by $G(t) := \Omega' \setminus B(t)$ instead we denote the `good set', {which}   in particular {has} $\Omega' \setminus \overline{\Omega} \subseteq G(t)$. Intuitively, if a point $ x \in \Omega $ is connected to $ \Omega' \setminus \overline{ \Omega } $ in $ \Omega' \setminus \Gamma ( t ) $, then $ x \in G ( t ) $, see  \Cref{lemma:characterization_of_bad_set} for details. 

\tikzset{every picture/.style={line width=0.75pt}} 

\begin{figure}
	\label{fig:bad_set}

\tikzset{every picture/.style={line width=0.75pt}} 

\begin{tikzpicture}[x=0.75pt,y=0.75pt,yscale=-1,xscale=1]
	
	\draw  [fill={rgb, 255:red, 74; green, 144; blue, 226 }  ,fill opacity=0.6 ] (147,49.75) .. controls (148.5,95.25) and (147.44,98.11) .. (148,139.75) .. controls (171.89,130.56) and (161.44,135) .. (184,126.25) .. controls (185.5,159.25) and (187.5,157.75) .. (190,199.25) .. controls (131.5,204.75) and (86.56,201.89) .. (76.11,191.89) .. controls (64.11,184.56) and (56.9,111.82) .. (74.78,79.44) .. controls (92.65,47.07) and (142,52.25) .. (147,49.75) -- cycle ;
	\draw  [fill={rgb, 255:red, 65; green, 117; blue, 5 }  ,fill opacity=0.6 ][line width=0.75]  (269.22,41.44) .. controls (294.56,61.67) and (267.44,187.44) .. (190,199.25) .. controls (187.67,167.22) and (187.22,166.33) .. (184,126.25) .. controls (164.56,133.89) and (169.22,131.67) .. (148,139.75) .. controls (147.83,119.26) and (147.81,101.74) .. (147.71,86.85) .. controls (147.66,79.41) and (147.59,72.62) .. (147.48,66.45) .. controls (147.43,63.37) and (147.08,52.53) .. (147,49.75) .. controls (180.56,48.11) and (240.56,28.33) .. (269.22,41.44) -- cycle ;
	\draw    (195.89,42.33) .. controls (240.78,49.89) and (205.89,169.44) .. (243.22,166.78) ;
	\draw [color={rgb, 255:red, 208; green, 2; blue, 27 }  ,draw opacity=1 ][line width=1.5]    (147,49.75) -- (148,139.75) ;
	\draw [color={rgb, 255:red, 208; green, 2; blue, 27 }  ,draw opacity=1 ][line width=1.5]    (148,139.75) -- (184,126.25) ;
	\draw [color={rgb, 255:red, 208; green, 2; blue, 27 }  ,draw opacity=1 ][line width=1.5]    (184,126.25) -- (190,199.25) ;
	\draw  [fill={rgb, 255:red, 65; green, 117; blue, 5 }  ,fill opacity=0.6 ] (395,50.17) .. controls (415,40.17) and (440.2,56.2) .. (505.8,58.2) .. controls (500,79.83) and (498.67,85.5) .. (489.4,121.8) .. controls (490.65,123.68) and (539,60.6) .. (539.33,62.17) .. controls (572.18,50.21) and (652.14,48.43) .. (637.67,86.5) .. controls (622.33,121.83) and (637,136.83) .. (614,183.5) .. controls (592.33,217.17) and (583.59,225.49) .. (570.67,224.17) .. controls (544.82,221.52) and (496.67,230.5) .. (486.67,215.5) .. controls (476.67,200.5) and (448.67,199.83) .. (404.33,185.5) .. controls (379.53,186) and (366.33,161.5) .. (351.33,147.83) .. controls (333,129.83) and (332,97.5) .. (351,74.83) .. controls (365,54.83) and (378.33,57.17) .. (395,50.17) -- cycle ;
	\draw    (403,47.83) .. controls (390,112.5) and (451.53,137.33) .. (404.33,185.5) ;
	\draw    (590.33,53.83) .. controls (588,111.83) and (493.67,194.83) .. (557.33,223.5) ;
	\draw    (505.8,58.2) -- (539.33,62.17) ;
	\draw [color={rgb, 255:red, 208; green, 2; blue, 27 }  ,draw opacity=1 ][line width=1.5]    (489.4,121.8) -- (466.67,102.5) ;
	\draw [color={rgb, 255:red, 208; green, 2; blue, 27 }  ,draw opacity=1 ][line width=1.5]    (466.67,102.5) -- (472.33,204.83) ;
	\draw  [draw opacity=0][fill={rgb, 255:red, 74; green, 144; blue, 226 }  ,fill opacity=0.6 ] (539.33,62.17) -- (490.11,120.78) -- (505.8,58.2) -- cycle ;
	\draw [color={rgb, 255:red, 208; green, 2; blue, 27 }  ,draw opacity=1 ][line width=1.5]    (489.4,121.8) -- (539.33,62.17) ;
	\draw [color={rgb, 255:red, 208; green, 2; blue, 27 }  ,draw opacity=1 ][line width=1.5]    (505.8,58.2) -- (489.4,121.8) ;
	
	\draw (120.2,56.8) node [anchor=north west][inner sep=0.75pt]  [font=\footnotesize,color={rgb, 255:red, 208; green, 2; blue, 27 }  ,opacity=1 ]  {$\Gamma(t) $};
	\draw (92.4,206.6) node [anchor=north west][inner sep=0.75pt]  [font=\footnotesize,color={rgb, 255:red, 74; green, 144; blue, 226 }  ,opacity=1 ]  {$B( t)$};
	\draw (174.8,93) node [anchor=north west][inner sep=0.75pt]  [font=\footnotesize]  {$\Omega $};
	\draw (227.6,60.4) node [anchor=north west][inner sep=0.75pt]  [font=\footnotesize]  {$\Omega '\setminus \overline{\Omega }$};
	\draw (222,206) node [anchor=north west][inner sep=0.75pt]  [font=\footnotesize,color={rgb, 255:red, 65; green, 117; blue, 5 }  ,opacity=1 ]  {$G( t)$};
	\draw (438,65.83) node [anchor=north west][inner sep=0.75pt]  [font=\footnotesize]  {$\Omega $};
	\draw (473.71,160.4) node [anchor=north west][inner sep=0.75pt]  [font=\footnotesize,color={rgb, 255:red, 208; green, 2; blue, 27 }  ,opacity=1 ]  {$\Gamma(t) $};
	\draw (360.57,116.4) node [anchor=north west][inner sep=0.75pt]  [font=\footnotesize]  {$\Omega '\setminus \overline{\Omega }$};
	\draw (560.86,158.83) node [anchor=north west][inner sep=0.75pt]  [font=\footnotesize]  {$\Omega '\setminus \overline{\Omega }$};
	\draw (429.14,205.11) node [anchor=north west][inner sep=0.75pt]  [font=\footnotesize,color={rgb, 255:red, 65; green, 117; blue, 5 }  ,opacity=1 ]  {$G( t)$};
	\draw (511.14,36.54) node [anchor=north west][inner sep=0.75pt]  [font=\footnotesize,color={rgb, 255:red, 74; green, 144; blue, 226 }  ,opacity=1 ]  {$B( t)$};

\end{tikzpicture}

\caption{Two examples of the `bad set' $ B( t) $ seen in blue.}
\end{figure}
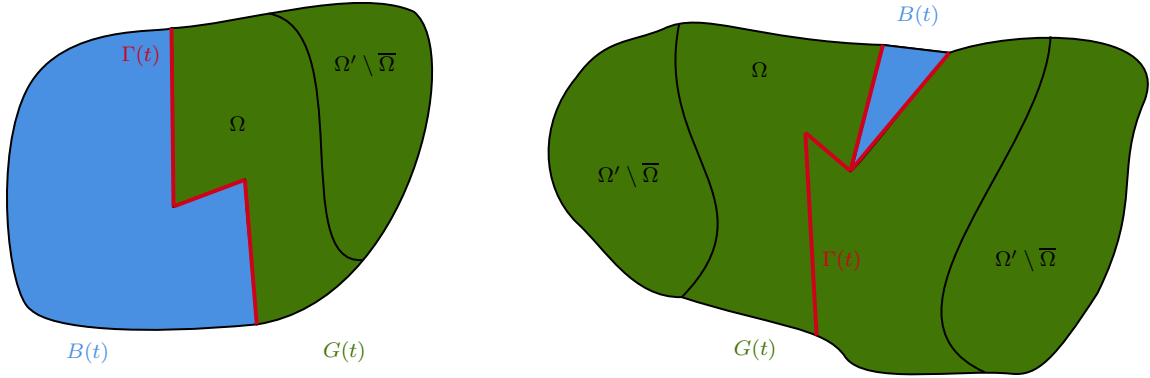

    Now,   we say that  $y_\eps$ converges to $u$, and write $y_\eps \rightsquigarrow u $, if  for {every} $t \in [0,1]$
\begin{align}\label{baruconv0}
\bar{u}_\eps(t)\coloneqq \dfrac{y_\eps(t) -{\rm id}}{\eps}\to u(t) \text{ in measure on $G(t)$},  \quad \quad    e(\bar{u}_\eps(t)) \to e(u(t)) \text{ in measure on } G ( t ) 
\end{align}
and 
\begin{align}\label{baruconv00}
\dfrac{1}{\varepsilon^{ 2 } }
	\int_{ B(t) } {\rm dist}^2\big( \nabla y_\eps(t), \sporth(2)\big)\dd{ x } \to 0. 
	\end{align}
We are now ready to  state our main result.

\begin{theorem}[Linearization of  quasistatic  fracture  evolution]\label{thm:main}
{Let  $\Omega\subseteq \Omega'\subseteq \R^2$ be  bounded   Lipschitz domains  such that  $\Omega' \setminus \overline{\Omega}$ is a Lipschitz  set.}   Let $u_0 \in \gsbd^2(\Omega')$ with $u_0 = h(0)$ on $\Omega' \setminus \overline{\Omega}$ such that $u_{ 0 }$ minimizes \eqref{linen} among all $v \in \gsbdtwo(\Omega')$ with $v = h(0)$ on $\Omega' \setminus \overline{\Omega}$. Let $y_\eps$ be approximate solutions to nonlinear quasistatic fracture as defined in \eqref{eqn:minMovDeform} with well prepared initial data $ (y_{ 0 }^{ \varepsilon})_{ \varepsilon } $  in the sense that   $(y^\eps_0 - {\rm id})/\eps \to u_0  $ in measure on $\Omega'$  and $ \limsup_{ \varepsilon \to 0 } \energyNonlin ( y_{ 0 }^{ \varepsilon } ) \leq  \energyLin   ( u_{ 0 } )  $.

 Then, there  exists a linear quasistatic fracture  evolution  as in Definition \ref{def:linearQuasistatic},  denoted by $(u(t),\Gamma(t))_{t \in [0,1]}$, {such that}
$$
\Gamma(t) = 
	\bigcup_{ \tau \in I_{ \infty }^{ t } }
	J_{ u ( \tau )} \quad \text{for all $t \in [0,1]$}$$
  {and,} up to a subsequence in $\eps$ (not relabeled), $y_\eps \rightsquigarrow u$. 

 Moreover, denoting the total energy  along  the evolution $y_\eps$ by
\begin{align}\label{totalepsenergy}
\energyNonlinTot(t) := 	\dfrac{1}{\varepsilon^{ 2 } }
	\int_{ \Omega' } W ( \nabla y_\eps(t) )\dd{ x }
	+
	\dfrac{ 1 }{ \varepsilon^{ 2 \beta } }
	\int_{ \Omega' }
	\abs{ \nabla^{ 2 } y_\eps(t) }^{ 2 }
	\dd{ x }
	+
	\kappa \hm^{ 1}  \left(  \Gamma_{ \varepsilon } ( t ) \right), 
	\end{align}
	we find
\begin{align}\label{separate0}
\energyNonlinTot(t) \to \energyLinTot(t) \quad \text{for all $t\in[0,1]$}.  
\end{align}
	
\end{theorem}

 \begin{remark}\label{rem: well}
 	We   make  the following remarks regarding the main result.
\begin{enumerate}[labelindent=0pt,labelwidth=\widthof{\ref{item:improved_convergence}},label=(\roman*),ref=(\roman*),itemindent=0em,leftmargin=!]
	\item The existence of well-prepared initial conditions is a consequence of the $ \Gamma $-convergence of $ \energyNonlin $ to  $\energyLin$,  see \cite[Thm.~2.7]{friedrich_griffith_energies_as_small_strain_limit_of_nonlinear_models_for_nomsimple_brittle_materials}.
\item \label{item:improved_convergence} More precisely, convergence (\ref{separate0}) can be refined to separate convergence of elastic and crack energies, see \eqref{separate} for details. 
\item If we accept that we modify the sequence $ (y_{ \varepsilon })_\eps $ to {account for the broken off pieces}, then we can actually show a much stronger convergence than $ y_{ \varepsilon } \rightsquigarrow u $. This modification procedure is however rather involved and is done in equations (\ref{def:y_rot}), (\ref{eqn:tildeuEps}), (\ref{def:u_eps}), and (\ref{eq:veps_def}). The resulting improved convergence is then seen in Item \ref{item:ueps_converges_in_measure_to_u} and \Cref{rmk:L2StrongEverywhere}.

\end{enumerate}

\end{remark}

The proof of this result will be our focus for the rest of the paper and it is split into manageable pieces.
In Section \ref{sec:mathPrelim}, we introduce a variety of tools that will be essential to our analysis.
 In Section \ref{sec:compactness}, we address compactness of the sequence $(y_\eps)_\eps$ and  its  rescaled  displacements  $(u_\eps)_\eps$ to find a candidate displacement $u$  {that might be}  a linear quasistatic fracture evolution in the sense  of Definition~\ref{def:linearQuasistatic}. In Section \ref{sec:approximateRelationsMinimality}, we prove minimality for {the} limit displacement $u$ at times $ t \in I_{ \infty } $.  Likewise, in Section~\ref{sec:approximateRelationsBalance}, we show that this displacement satisfies two approximate energy-balance relations.
We complete the proof of Theorem \ref{thm:main} in Section \ref{sec:completeProof}. This is done by extending minimality for the limit displacement to all times and showing that the approximate energy-balance relations give rise to the desired energy-balance equation.

\section{ Preliminaries:  Compactness, density, and jump transfer}\label{sec:mathPrelim}

 In this section, we  record  some auxiliary results needed  to prove our main theorem.   

First, we state a compactness result, which is \cite[Thm.~6.1]{friedrich_solombrino_quasistatic_crack_growth_in_2d_linearized_elasticity} but with extra information extrapolated from the proof.
\begin{theorem}
	\label{thm:gsbd2_compactness}
	Let $ \Omega \subseteq  \Omega' \subseteq\mathbb{ R }^{ 2 } $ be  bounded  Lipschitz domains  and $(h_n)_n, h \subseteq \wkp^{1,2}(\Omega';\R^2)$  such that $h_n\to h$ in $\wkp^{1,2}(\Omega';\R^2)$. Let $ (u_{ n } )_{ n \in \mathbb{ N } }   \subseteq  \gsbdtwo ( \Omega ') $ be a sequence with 
	\begin{equation*}
		\norm{ e( u_{ n } ) }_{ \lp^{ 2 } (  \Omega'  ) }
		+
		\hm^{ 1 } ( J_{ u_{ n } } )
		\leq 
		C < 
		\infty
	\end{equation*}
	and $ u_{ n } = h_n $ on $ \Omega' \setminus \overline{\Omega } $. Then, there is a subsequence  (not relabeled),  some $ u \in \gsbdtwo ( \Omega' ) $ with $ u = h $ on $ \Omega' \setminus \overline{\Omega} $, {a sequence of sets $\mathcal{S}_n : = ( S_{ j }^{ n } )_{ j \ge 0 } $ contained in $\Omega$ such that $\mathcal{S}_n \cup(\Omega' \setminus \bigcup_{j\ge 0} S_{j}^n)$ is a Caccioppoli partition of $\Omega'$,} and a sequence of infinitesimal rigid motions $ ( a_{j }^{ n } )_{ j \ge 1 } $ such that
	\begin{enumerate}[labelindent=0pt,labelwidth=\widthof{\ref{item:S0n_vanishes}},label=(\roman*),ref=(\roman*),itemindent=0em,leftmargin=!]
		\item \label{item:rescaled_un }
		$ u_{ n } - \chi_{ S_{ 0 }^{ n } } ( u_{ n } - h_n ) -\sum_{ j=1 }^{ \infty } a_{ j }^{ n } \chi_{ S_{ j }^{ n } } \to u $ pointwise almost everywhere in $\Omega'$,
		\item 
		\label{item:weak_convergence_un}$ e ( u_{ n } ) \rightharpoonup e(u) $ weakly in $ \lp^{ 2 } (  \Omega'  ; \mathbb{ R }^{ 2 \times 2 }_{\rm sym} ) $,
		\item  \label{item:lscsurface} $ \hm^{ 1 } \left( J_{ u } \cap U \right) \leq \liminf_{ n \to \infty } \hm^{ 1} ( J_{ u_{ n } } \cap U ) $ for all open Lipschitz subsets $ U \subseteq \Omega' $,
		\item \label{item:caccioppoli_partitions_small_jump}$ \hm^{ 1 } \left( \bigcup_{ j=0 }^{ \infty } \partial^{ \ast } S_{ j }^{ n }  \setminus J_{ u_{ n } }\right) \to 0 $, and
		\item \label{item:S0n_vanishes}$ \lm^{ 2 } ( S_{ 0 }^{ n } ) \to 0 $.
	\end{enumerate}
\end{theorem}
\begin{proof}
	This is an immediate consequence of the compactness result \cite[Thm.~6.1]{friedrich_solombrino_quasistatic_crack_growth_in_2d_linearized_elasticity} combined with the proper bookkeeping. More precisely, the form of the functions in \ref{item:rescaled_un } can be found in  \cite[Eq.~(65)]{friedrich_solombrino_quasistatic_crack_growth_in_2d_linearized_elasticity}, where the component $S_0^n$ there is denoted by $E^l_k$, see  \cite[Eq.~(66)]{friedrich_solombrino_quasistatic_crack_growth_in_2d_linearized_elasticity}. The rest of the properties then  follows  {immediately} from the proof.
\end{proof}

We denote by $\mathcal{W}(\Omega';\R^2)$ the collection of functions  $u \in \gsbdtwo(\Omega')$ such that  $J_u$ is closed and included in a
finite union of closed connected pieces of $C^1$-curves, and $u \in C^\infty(  \Omega'  \setminus J_u;\R^2) \cap {\wkp}^{2,\infty}(\Omega' \setminus J_u;\R^2)$.  Note that functions $u \in  \mathcal{W}(\Omega';\R^2)$ lie in $\gsbvtwotwo ( \Omega' ; \mathbb{ R }^{ 2 } ) $ and automatically satisfy  $\hm^{ 1 } ( J_{ \nabla u } \setminus J_{u }  ) = 0 $.  We present the following density  result  whose proof will be given in \Cref{sec:appendix_proofs}. 

\begin{theorem}
	\label{thm:density_with_boundary_values}
	Given $ h \in \wkp^{ 2 , \infty } ( \Omega'  ; \mathbb{R}^2  ) $ and $ v \in \gsbdtwo ( \Omega' ) $ with $ v = h $ on $ \Omega' \setminus \overline{\Omega } $, there exists $ (w_{ \delta})_{ \delta } \subseteq \mathcal{W} ( \Omega';\R^2 ) $ such that $ w_{ \delta } = h $ on $ \Omega' \setminus \overline{\Omega } $ and,  as $\delta \to 0$,  
	\begin{enumerate}[labelindent=0pt,labelwidth=\widthof{\ref{item:density_jump_set_diff_vanishes}},label=(D\arabic*),ref=(D\arabic*),itemindent=0em,leftmargin=!]
		\item \label{item:density_measure_conv}
		$ w_{ \delta} \to v $ in measure on $\Omega'$, 
		\item \label{item:density_l2_convergence_sym_grad}
		$ e( w_{ \delta } ) \to e( v )  $ in $\lp^2(\Omega'; \mathbb{ R }^{ 2 \times 2}_{ \mathrm{sym} } )$ and  
		\item \label{item:density_jump_set_diff_vanishes} $ \hm^{ 1} \big( J_{ w_{ \delta } } \triangle J_{ v } \big) \to 0 $. 
	\end{enumerate}
	\end{theorem}

By keeping track of the bound on the second derivatives, we can also deduce the following refined \emph{jump transfer lemma}.

\begin{theorem}[Refined Jump Transfer in $ \gsbdtwo $]
	\label{thm:jump_transfer_refined}
	Let  {$ \Omega   \subseteq   \Omega' $ }be bounded Lipschitz domains in $ \mathbb{ R }^{ 2 } $ such that $ \Omega' \setminus \overline{ \Omega } $ has Lipschitz boundary. Let $ \ell \in \mathbb{ N } $ and let $  (h_{ n }^{ l } )_{ n }  \subseteq  \wkp^{ 2, \infty } ( \Omega' ; \mathbb{ R }^{ 2 } ) $ be bounded sequences for $ l = 1 , \ldots, \ell$. Let $ ( u_{ n }^{ l })_{n } $ be sequences in $ \gsbdtwo ( \Omega' ) $ and $ u^{ l } \in \gsbdtwo ( \Omega' ) $ such that
	\begin{enumerate}[labelindent=0pt,labelwidth=\widthof{\ref{item:condition_measure_convergence}},label=(\arabic*),ref=(\arabic*),itemindent=0em,leftmargin=!]
		\item $ \norm{ e ( u_{ n }^{l } ) }_{ \lp^{ 2 } }
		+ \hm^{ 1 } \left( J_{ u_{ n }^{ l } } \right) \leq M $ for all $ n \in \mathbb{ N } $,
		\item \label{item:condition_measure_convergence} $ u_{ n }^{ l } \to u^{ l } $ in measure in $ \Omega' $, $ u_{ n }^{ l } = h_{ n }^{ l } $ on $ \Omega' \setminus \overline{ \Omega } $,
	\end{enumerate}
	for all $ l = 1 , \ldots , \ell $. Then there exists a  subsequence in $ n $   (not   relabeled) with the following property: For each $ \phi \in \mathcal{ W } ( \Omega' ; \mathbb{ R }^{ 2 } ) $, there is a sequence $  (\phi_{ n }  )_{n }  \subseteq  \gsbvtwotwo ( \Omega' ; \mathbb{ R }^{ 2 } ) $ with $ \phi_{ n } = \phi $ on $ \Omega' \setminus \overline{ \Omega } $ such that, as $ n \to \infty $,
	\begin{enumerate}[labelindent=0pt,labelwidth=\widthof{\ref{item:jump_transfer_l_inf_estimates}},label=(\roman*),ref=(\roman*),itemindent=0em,leftmargin=!]
		\item $ \phi_{ n } \to \phi $ in measure in  $\Omega'$, 
		\item $ e( \phi_{ n } ) \to e ( \phi ) $ in $ \lp^{ 2 }  ( \Omega' ; \mathbb{ R }^{ 2 \times 2}_{ \mathrm{sym} }   ) $,
		\item \label{eqn:JTjumpGrad}$ \hm^{ 1 } \left( 
		\left(
		\left( J_{ \phi_{ n } } \cup J_{ \nabla \phi_{ n } } \right) \setminus
		\bigcup_{ l = 1 }^{ \ell } 
		J_{ u_{ n }^{ l } }
		\right)
		\setminus
		\left(
		J_{ \phi } \setminus
		\bigcup_{ l = 1 }^{ \ell }
		J_{ u^{ l } } 
		\right)
		\right)
		\to 0 $,
		\item 
		\label{item:jump_transfer_l_inf_estimates}
		$ \norm{ \nabla \phi_{ n } }_{ \lp^{ \infty } } \leq C \norm{ \nabla \phi }_{ \lp^{ \infty } } $ and 
		$ \norm{ \nabla^{ 2 } \phi_{ n } }_{ \lp^{ \infty  } }
		\leq C  \norm{ \nabla^{ 2 } \phi }_{ \lp^{ \infty } } $.
	\end{enumerate}
\end{theorem}

For the proof, we can closely follow \cite[Thm.~5.1]{friedrich_solombrino_quasistatic_crack_growth_in_2d_linearized_elasticity}, where the statement {has been proven} without item \ref{item:jump_transfer_l_inf_estimates} and without the jump of the gradient {in \ref{eqn:JTjumpGrad}}. To {obtain these additions}, we adapt the extension lemma in \cite[Lem.~5.2]{friedrich_solombrino_quasistatic_crack_growth_in_2d_linearized_elasticity} in the following way. 

\begin{lemma}
	\label{lemma:reflection}
	{For $r_1,r_2>0$, let $ R^+$, $ R^{ - } ,$ and $R$ be given by a common rotation and translation of 
	\begin{equation} \label{eqn:Reccies}
	(0, r_{ 1 } ) \times ( 0, r_{ 2 } ) , \quad( 0, r_{ 1 } ) \times ( - r_{ 2 } /2 , 0 ), \quad \text{and} \quad ( 0, r_{ 1 } )\times (-r_2/2,r_2),
	\end{equation} respectively.  Let $ \phi \in \gsbvtwotwo ( R^+;  \R^2  ) $ with $ \hm^{ 1 } \left( J_{ \nabla \phi } \setminus J_{ \phi } \right)=0$. Then there is an extension of $\phi$ onto $R^-$ given by $ \hat \phi \in \gsbvtwotwo ( { R }; \mathbb{ R }^{ 2 } ) $ satisfying the} properties
	\begin{enumerate}[labelindent=0pt,labelwidth=\widthof{\ref{item:reflection_jump}},label=(\roman*),ref=(\roman*),itemindent=0em,leftmargin=!]
		\item $ \hm^{ 1 } ( J_{ \hat{ \phi } } ) \leq C \hm^{ 1 } ( J_{ \phi } ) $,
		\item  $ \Vert \nabla \hat{ \phi }  \Vert_{ \lp^{ \infty }  ({R})   }
		\leq C  \Vert\nabla \phi   \Vert_{ \lp^{ \infty }    ({R^+})  } $,  
		\item 		$  \Vert \nabla^{ 2 } \hat{ \phi }   \Vert_{ \lp^{ \infty }  ({R})   }
		\leq C  \Vert  \nabla^2  \phi   \Vert_{ \lp^{ \infty }    ({R^+})  } $, and 
		\item \label{item:reflection_jump}$ \hm^{ 1 } \left( J_{ \nabla \hat{ \phi } } \setminus J_{ \hat{ \phi } } \right)= 0  $
	\end{enumerate} 
	for some universal constant $ C \ge 1$ independent of $ R $ and $ \phi $.
\end{lemma}

\begin{proof}[Proof of \Cref{lemma:reflection}]
	{We may assume $R^+$, $R^-$, and $R$ are as in equation (\ref{eqn:Reccies}), respectively. For $ x \in R^{- } $ we define $\hat \phi$ by  }
	\begin{equation*}
		\hat{ \phi } ( x) = 3 \phi ( x_{ 1 }, -x_{ 2 } ) -2 \phi ( x_{ 1 }, -2 x_{ 2 } ). 
	\end{equation*}
	{To see that} no jump is created by $ \hat{ \phi } $ and $ \nabla \hat{ \phi } $ along $ \left\{ x_{ 2 } = 0 \right\} $, we show that the traces of $ \hat{ \phi } $ and $ \nabla \hat{ \phi } $ from inside $R^-$ agree with the traces of $ { \phi } $ and $ \nabla { \phi } $ from inside $R^+$. It is immediate that   $\phi(x_1,0) = \hat \phi(x_1,0)$ for $x_1 \in (0,r_1)$.  Additionally, the absolutely continuous derivatives are given by
	\begin{align*}
		\partial_{ x_{ 1 } } \hat{ \phi } ( x )
		& =
		3 \partial_{ x_{ 1 } } \phi ( x_{ 1 } , - x_{2 } )
		-
		2 \partial_{ x_{ 1 } } \phi ( x_{ 1 } , - 2x_{ 2 } ),
		\\
		\partial_{ x_{ 2 } } \hat{ \phi } ( x ) 
		& =
		-3 \partial_{ x_{ 2 } } \phi ( x_{ 1 } , -x_{ 2 } )
		+4 \partial_{ x_{ 2 } } \phi ( x_{ 1 } , - 2 x_{ 2 } ),
	\end{align*}
	so we have continuity {in the variable $x_2$} at  $\lbrace x_2 = 0\rbrace$, proving  $\nabla \phi(x_1,0) = \nabla \hat \phi(x_1,0)$ for $x_1 \in (0,r_1)$.  
	{Given that $\hat \phi$ is defined by reflection and creates no additional jump on $\{x_2 = 0\}$ the conclusion of the lemma follows.}
\end{proof}
\begin{proof}[Proof of \Cref{thm:jump_transfer_refined}]
	Inspecting the proof of the jump transfer lemma in $ \gsbdtwo $  (\cite[Thm.~5.1]{friedrich_solombrino_quasistatic_crack_growth_in_2d_linearized_elasticity}),  we see that $ \phi_{ n } $ is defined using a reflection argument \cite[Lem.~5.2]{friedrich_solombrino_quasistatic_crack_growth_in_2d_linearized_elasticity} on a finite, but increasing number of squares. {Using instead the reflection from \Cref{lemma:reflection},} we see that the additional properties are {satisfied by the construction.}
\end{proof}

\section{Compactness of the nonlinear evolution}\label{sec:compactness}

We first obtain a priori estimates on the sequence $(y_\eps)_\eps$, essentially coming from  an  approximate energy-balance {relation}. With these bounds, we apply techniques developed in \cite{friedrich_griffith_energies_as_small_strain_limit_of_nonlinear_models_for_nomsimple_brittle_materials} to identify a limiting displacement at countable times $ t \in I_{ \infty } $.

\subsection{A priori estimates}
As in  \cite{dal_maso_francfort_toader_quasistatic_crack_growth_in_nonlinear_elasticity, francfort_larsen_existence_and_convergence_for_quasi_static_evolution_in_brittle_fracture},   we establish some a priori estimates on the total energy $\energyNonlinTot$  defined in \eqref{totalepsenergy}.

\begin{lemma}[{A priori estimate}]\label{lem:aprioriEst}
Letting $y_\eps$ {and $\Gamma_\eps$ be defined as in \eqref{eqn:minMovDeform} and \eqref{eqn:GammaEps}, respectively,} the following estimates are satisfied:
\begin{equation}
	\label{eq:a_priori_estimate}
	\sup_{ \varepsilon > 0 , t \in [0,1] }
	\energyNonlinTot (t)
	< \infty,
\end{equation}
{and,  for all $t_n^\eps \in I_\eps$,  
\begin{equation}
	 \energyNonlinTot (t_{ n }^{ \varepsilon })
	-
	\sigma_{ \varepsilon }
	\leq{}
	\energyNonlin (y_{ 0 }^{ \varepsilon } )
	+ 
	\dfrac{ 1 }{ \varepsilon }
	\int_{ 0 }^{ t_{ n }^{ \varepsilon } }
	\int_{{ \{ \nabla y_{ \varepsilon } \in B_{ r } ( \sporth( 2 ) )\} }}
	\langle \diff W ( \nabla y_{ \varepsilon } ) , \nabla  \partial_t h \rangle
	\dd{ x }
	\dd{ t }
	\label{eq:a_priori_inequality_error_reduced}
\end{equation}
for $\sigma_\eps\to 0$  as $\eps \to 0 $  (independently of the chosen time $t_{ n }^{ \varepsilon }$),  	where $r$ is as in \ref{eq:W_growth_assumptions}.  }
\end{lemma}

\begin{proof}
Let $ t \in I_{ \infty } $ and $ \varepsilon $ be sufficiently small such that $ t \in I_{ \varepsilon } $. 
The function $ \mathrm{id} + \varepsilon h_{ n }^{ \varepsilon }  $ is an admissible competitor for the minimization problem (\ref{eq:discrete_min_problem})  without any jumps and thus
\begin{align}
	& \nonumber \dfrac{1}{ \varepsilon^{ 2 } }
	\int_{\Omega'}
	W ( \nabla y_{ n }^{ \varepsilon } )
	\dd{ x }
	+
	\dfrac{1 }{ \varepsilon^{ 2 \beta } }
	\int_{\Omega' }
	\abs{ \nabla^{ 2 } y_{ n }^{ \varepsilon } }^{ 2 }
	\dd{ x }
	+
	\kappa
	\hm^{ 1}  \left(
	\left( 
	J_{ y_{ n }^{ \varepsilon } } \cup J_{ \nabla y_{ n }^{ \varepsilon } } 
	\right)
	\setminus
	\bigcup_{ k=0 }^{ n - 1 }
	\left(
	J_{ y_{ k }^{ \varepsilon } } \cup J_{ \nabla y_{ k }^{ \varepsilon } }
	\right)
	\right)
	\\
	\label{eqn:compareBulkBC} \leq{} &
	\dfrac{ 1 }{ \varepsilon^{ 2 } }
	\int_{\Omega' } W ( \mathrm{Id}+ \varepsilon \nabla h_{ n }^{ \varepsilon }  )
	\dd{ x }
	+
	\dfrac{ 1 }{ \varepsilon^{ 2 \beta } }
	\int_{\Omega' }
	\abs{ \varepsilon \nabla^{ 2 } h_{ n }^{ \varepsilon } }^{ 2 }
	\dd{ x }.
\end{align}
By assumption \eqref{eqn:bdryData} and \ref{eq:W_zero_on_sod}, we have $ h \in \lp^{ \infty }  ( (0,1); \wkp^{ 2, \infty }  ( \Omega'; \mathbb{ R }^{ 2 }  )  ) $ and $ W ( \mathrm{Id} ) = 0 $, $  \diff  W ( \mathrm{Id} ) = 0 $. Therefore,   by Taylor's formula, using that $ W $ is $ \cont^{ 3 } $ near the identity, we get the inequality 
\begin{equation*}
	W \left( \mathrm{Id} + \varepsilon \nabla h _{ n }^{ \varepsilon } \right)
	\lesssim
	\varepsilon^{ 2 }
	\abs{ \nabla h_{ n }^{ \varepsilon }  }^{ 2 }.
\end{equation*}
Thus, we may estimate the right-hand side of \eqref{eqn:compareBulkBC}, up to a constant, by
\begin{equation*}
	\int_{\Omega' }
	\abs{ \nabla h_{ n }^{ \varepsilon } }^{ 2 }
	\dd{ x }
	+
	\varepsilon^{ 2 - 2 \beta }
	\int_{\Omega'}
	\abs{ \nabla^{ 2 } h_{ n }^{ \varepsilon } }^{ 2 }
	\dd{ x },
\end{equation*}
which stays uniformly bounded. Consequently, we have established that, uniformly for  $t \in [0,1]$, 
\begin{equation}
	\label{eq:a_priori_estimate_one}
	\dfrac{ 1 }{ \varepsilon^{ 2 } }
	\int_{\Omega'}
	W ( \nabla y_{ \varepsilon }  (t)  )
	\dd{ x }
	+
	\dfrac{ 1 }{ \varepsilon^{ 2 \beta } }
	\int_{\Omega'} 
	\abs{ \nabla^{ 2 } y_{ \varepsilon }   (t)  }^{ 2 }
	\dd{ x }
	\leq C < \infty.
\end{equation}
To obtain a uniform bound on the measure of the cumulative crack, we compare the deformation at time $t_n^\eps$ to the deformation at time $t_{n-1}^\eps$. Precisely, {noting} that 
$y_{ n - 1 }^{ \varepsilon }+ \varepsilon ( h_{ n }^{ \varepsilon }  - h_{ n - 1 }^{ \varepsilon } )$ 
is an admissible competitor for the minimization problem (\ref{eq:discrete_min_problem}), {we} get
\begin{align}
	\notag
	& \dfrac{1}{ \varepsilon^{ 2 } }
	\int_{\Omega'}
	W ( \nabla y_{ n }^{ \varepsilon } )
	\dd{ x }
	+
	\dfrac{1 }{ \varepsilon^{ 2 \beta } }
	\int_{\Omega'}
	\abs{ \nabla^{ 2 } y_{ n }^{ \varepsilon } }^{ 2 }
	\dd{ x }
	+
	\kappa
	\hm^{ 1 } \left(
	\left( 
	J_{ y_{ n }^{ \varepsilon } } \cup J_{ \nabla y_{ n }^{ \varepsilon } } 
	\right)
	\setminus
	\bigcup_{ k=0 }^{ n - 1 }
	\left(
	J_{ y_{ k }^{ \varepsilon } } \cup J_{ \nabla y_{ k }^{ \varepsilon } }
	\right)
	\right)
	\\
	\label{eq:a_priori_estimate_jumpset_first_step}
	\leq {} &
	\dfrac{ 1 }{ \varepsilon^{ 2 } }
	\int_{\Omega'}
	W ( \nabla y_{ n - 1 }^{ \varepsilon } + \varepsilon \nabla ( h_{ n }^{ \varepsilon }  - h_{ n - 1 }^{ \varepsilon } ) )
	\dd{ x }
	+
	\dfrac{ 1 }{ \varepsilon^{ 2 \beta } }
	\int_{\Omega'}
	\abs{ \nabla^{ 2 } y_{ n - 1 }^{ \varepsilon } + \varepsilon \nabla^{ 2 } ( h_{ n }^{ \varepsilon }  - h_{ n - 1 }^{ \varepsilon } ) }^{ 2 }
	\dd{ x }.
\end{align}
Take $r > 0$ as in \ref{eq:W_growth_assumptions} so that $ W $ is $ \cont^{ 3 } $ in $B_{ 2r } ( \sporth ( 2 ) )$  and define the preimage
\begin{equation*} 
	\mathcal{N}_{ n-1 }^{ \varepsilon } \coloneqq \left\{ \nabla y_{ n - 1 }^{ \varepsilon } \in B_{ r } ( \sporth ( 2 ) ) \right\} .
\end{equation*} 
We want to apply Taylor's formula to the {first summand on the right-hand side of \eqref{eq:a_priori_estimate_jumpset_first_step}} and use the local Lipschitz continuity of $W$ to argue that the integral asymptotically vanishes as $\eps \to 0$  {outside of}  $  \mathcal{ N }^\eps_{n-1}  $, after summing up all times.

First we take care of the integral over $ \Omega' \setminus \mathcal{ N }_{ n - 1 }^{ \varepsilon } $.
By the local Lipschitz estimate for $ W $ in \ref{eq:W_growth_assumptions}, the uniform {control} on the boundary data \eqref{eqn:bdryData}, the energy estimate (\ref{eq:a_priori_estimate_one}),  \ref{eq:W_zero_on_sod}, and H\"older's inequality  we bound
\begin{align}
	\notag
	&
	\dfrac{ 1 }{ \varepsilon^{ 2 } }
	\int_{ \Omega' \setminus\mathcal{ N }_{ n - 1 }^{ \varepsilon } }
	W \left( \nabla y_{ n - 1 }^{ \varepsilon } + \varepsilon \nabla \left( h_{n }^{ \varepsilon } - h_{ n - 1 }^{ \varepsilon } \right) \right)
	\dd{ x }
	\\
	\notag
	\leq{} &
	\dfrac{ 1 }{ \varepsilon^{ 2 } }
	\int_{ \Omega' \setminus \mathcal{ N }_{n - 1 }^{ \varepsilon } }
	W \left( \nabla y_{n - 1 }^{ \varepsilon } \right)
	+
	C
	\left( 1 + \abs{ \nabla y_{ n - 1 }^{ \varepsilon } } + \varepsilon \norm{ \nabla \left( h_{n }^{ \varepsilon } - h_{ n - 1 }^{ \varepsilon } \right) }_{ \lp^{ \infty } } 
	\right)
	\varepsilon \abs{ \nabla \left( h_{n }^{ \varepsilon } - h_{ n - 1 }^{ \varepsilon } \right) }
	\dd{ x }
	\\
	\notag
	\leq{} & 
	\dfrac{ 1 }{ \varepsilon^{ 2 } }
	\int_{ \Omega' \setminus \mathcal{ N }_{n - 1 }^{ \varepsilon } }
	W \left( \nabla y_{n - 1 }^{ \varepsilon } \right)
	\dd{ x }
	+ 
	C
	\int_{ t_{n - 1 }^{ \varepsilon } }^{ t_{ n }^{ \varepsilon } }
	\int_{ \Omega' \setminus \mathcal{ N }_{ n - 1 }^{ \varepsilon } }
	\dfrac{ 1 }{ \varepsilon }
	\left( 1 + \abs{ \nabla y_{ n - 1 }^{ \varepsilon } } \right)
	\abs{ \nabla \partial_{ t } h }
	\dd{ x }
	\dd{ t }
	\\
	\notag
	\leq{} &
	\dfrac{ 1 }{ \varepsilon^{ 2 } } 
	\int_{ \Omega' \setminus \mathcal{ N }_{n - 1 }^{ \varepsilon } }
	W \left( \nabla y_{n - 1 }^{ \varepsilon } \right)
	\dd{ x }
	+
	C
	\int_{ t_{ n - 1 }^{ \varepsilon } }^{ t_{ n }^{ \varepsilon } }
	\left( 
	\dfrac{ 1 }{ \varepsilon^{ 2 } }
	\int_{ \Omega' \setminus \mathcal{ N }_{ n - 1 }^{ \varepsilon } }
	\left( 1 + \abs{ \nabla y_{n - 1 }^{ \varepsilon } } \right)^{ 2 }
	\dd{ x }
	\right)^{ 1/2 }
	\norm{ \nabla \partial_{ t } h }_{ \lp^{2 } \left( \Omega' \setminus \mathcal{ N }_{n - 1 }^{ \varepsilon } \right) }
	\dd{ t }
	\\
	\notag
	\leq{} &
	\dfrac{ 1 }{ \varepsilon^{ 2 } } 
	\int_{ \Omega' \setminus \mathcal{ N }_{n - 1 }^{ \varepsilon } }
	W \left( \nabla y_{n - 1 }^{ \varepsilon } \right)
	\dd{ x }
	+
	C
	\int_{ t_{ n - 1 }^{ \varepsilon } }^{ t_{ n }^{ \varepsilon } }
	\left( 
			\dfrac{ 1 }{ \varepsilon^{ 2 } }
	\int_{\Omega' } W \left( \nabla y_{ n -1 }^{ \varepsilon } \right) \dd{ x } 
		\right)^{ 1 / 2 }
	\norm{ \nabla \partial_{ t } h }_{ \lp^{ 2 } \left( \Omega' \setminus \mathcal{ N }_{ n - 1 }^{ \varepsilon } \right) }
	\dd{ t }
	\\
	\label{eq:a_priori_estimate_away_from_nbhd}
	\leq{} &
	\dfrac{ 1 }{ \varepsilon^{ 2 } } 
	\int_{ \Omega' \setminus \mathcal{ N }_{n - 1 }^{ \varepsilon } }
	W \left( \nabla y_{n - 1 }^{ \varepsilon } \right)
	\dd{ x }
	+
	C \int_{ t_{ n - 1 }^{ \varepsilon } }^{ t_{ n }^{ \varepsilon } }
	\norm{ \nabla \partial_{ t } h }_{ \lp^{ 2 } \left( \Omega' \setminus \mathcal{ N }_{ n - 1 }^{ \varepsilon } \right) }
	\dd{ t }.
\end{align}
In the  neighborhood  $ \mathcal{ N }_{ n - 1 }^{ \varepsilon } $, we apply {Taylor's} formula, the differentiability in time of $ h $, and Minkowski's integral inequality to obtain 
\begin{align}
	\notag
	&
	\dfrac{ 1 }{ \varepsilon^{ 2 } }
	\int_{ \mathcal{ N }_{ n - 1 }^{ \varepsilon } }
	W \left( \nabla y_{ n - 1 }^{ \varepsilon } + \varepsilon \nabla ( h_{ n }^{ \varepsilon } - h_{ n - 1 }^{ \varepsilon } ) \right)
	\dd{ x }
	\\
	\notag
	\leq{} &
	\dfrac{ 1 }{ \varepsilon^{ 2 } }
	\int_{ \mathcal{ N }_{ n - 1 }^{ \varepsilon } }
	W \left( \nabla y_{ n - 1 }^{ \varepsilon } \right)
	\dd{ x }
	+
	\dfrac{ 1 }{ \varepsilon }
	\int_{ t_{ n - 1 }^{ \varepsilon } }^{ t_{ n }^{ \varepsilon } }
	\int_{ \mathcal{ N }_{ n - 1 }^{ \varepsilon } }
	{\langle \diff W \left( \nabla y_{ n - 1 }^{ \varepsilon } \right),
	\nabla \partial_{ t } h  \rangle}
	\dd{ x }
	\dd{ t }
	+
	C \int_{ \mathcal{ N }_{ n - 1 }^{ \varepsilon } }
	\abs{ \nabla \left( h_{ n }^{ \varepsilon } - h_{ n -1 }^{ \varepsilon } \right) }^{ 2 }
	\dd{ x }
	\\
	\label{eq:a_priori_estimate_nbhd}
	\leq{} &
	\dfrac{ 1 }{ \varepsilon^{ 2 } }
	\int_{ \mathcal{ N }_{ n - 1 }^{ \varepsilon } }
	W \left( \nabla y_{ n - 1 }^{ \varepsilon } \right)
	\dd{ x }
	+
	\dfrac{ 1 }{ \varepsilon }
	\int_{ t_{ n - 1 }^{ \varepsilon } }^{ t_{ n }^{ \varepsilon } }
	\int_{ \mathcal{ N }_{ n - 1 }^{ \varepsilon } }
	{\langle \diff W \left( \nabla y_{ n - 1 }^{ \varepsilon } \right),
	\nabla \partial_{ t } h \rangle}
	\dd{ x }
	\dd{ t }
	+
	C
	\left(
	\int_{ t_{ n - 1 }^{ \varepsilon } }^{ t_{ n }^{ \varepsilon } }
	\norm{ \nabla \partial_{ t } h }_{ \lp^{ 2  } }
	\dd{ t }
	\right)^{ 2 }.
\end{align}
 By expanding the square in the integral involving the second derivatives, we therefore deduce by using the estimates (\ref{eq:a_priori_estimate_away_from_nbhd}) and (\ref{eq:a_priori_estimate_nbhd}), and again the differentiability in time of $ h $, that the right-hand side of  inequality  (\ref{eq:a_priori_estimate_jumpset_first_step}) is controlled by
\begin{align}\label{eq:non_it_line_applied_time_1}
	& \dfrac{1}{\varepsilon^{2 }}
	\int_{\Omega'}
	W \left( \nabla y_{ n - 1 }^{ \varepsilon } \right)
	\dd{ x }
	+
	\dfrac{ 1 }{ \varepsilon^{ 2 \beta } }
	\int_{\Omega'}
	\abs{ \nabla^{ 2 } y_{ n - 1 }^{ \varepsilon } }^{ 2 }
	\dd{ x }
	 +{}\dfrac{ 1 }{ \varepsilon}
	 \int_{ t_{ n - 1 }^{ \varepsilon } }^{ t_{ n }^{ \varepsilon } }
	\int_{ \mathcal{ N }_{ n - 1 }^{ \varepsilon } }
	{\langle \diff W \left( \nabla y_{ \varepsilon } \right), \nabla \partial_t  h\rangle }
	\dd{ x }
	\dd{ t }  \notag
	\\
	\notag
	&+
	C \int_{ t_{ n - 1 }^{ \varepsilon } }^{ t_{ n }^{ \varepsilon } }
	\norm{ \nabla \partial_{ t } h }_{ \lp^{ 2 } \left( \Omega' \setminus \mathcal{ N }_{ n - 1 }^{ \varepsilon } \right) }
	\dd{ t }
	+
	\dfrac{ 2 }{ \varepsilon^{ 2 \beta - 1 } }
	\int_{ t_{ n - 1 }^{ \varepsilon } }^{ t_{ n }^{ \varepsilon } }
	\int_{\Omega' }
	\inner*{ \nabla^{ 2 } y_{ \varepsilon } }{ \nabla^{ 2 } \partial_t h  } 
	\dd{ x }
	\dd{ t }	
	 +{}
	C \left(
	\int_{ t_{ n - 1 }^{ \varepsilon } }^{ t_{ n }^{ \varepsilon } }
	\norm{ \nabla \partial_t h }_{ \lp^{ 2 } }
	\dd{ t }
	\right)^{2 }
	\\
	&+
	\dfrac{ 1 }{ \varepsilon^{ 2 \beta - 2 } }
	\left(
	\int_{ t_{ n - 1 }^{ \varepsilon } }^{ t_{ n }^{ \varepsilon } }
	\norm{ \nabla^{ 2 } \partial_t h  }_{ \lp^{ 2 } }
	\dd{ t }
	\right)^{ 2 }
\end{align}
where we also employed the notation in \eqref{eqn:minMovDeform}.
Noting that the first two terms in (\ref{eq:non_it_line_applied_time_1}) are the same as the first two terms on the left-hand side of \eqref{eq:a_priori_estimate_jumpset_first_step} with $n-1$ in place of $n$, we will iterate the above reasoning to derive the a priori estimate.
Starting from the energy 
\begin{align}
	\notag
	& \dfrac{1}{ \varepsilon^{ 2 } }
	\int_{\Omega'}
	W ( \nabla y_{ n }^{ \varepsilon } )
	\dd{ x }
	+
	\dfrac{1 }{ \varepsilon^{ 2 \beta } }
	\int_{\Omega'}
	\abs{ \nabla^{ 2 } y_{ n }^{ \varepsilon } }^{ 2 }
	\dd{ x }
	+
	\kappa
	\hm^{ 1} \left(
	\bigcup_{ k=0 }^{ n  }
	\left(
	J_{ y_{ k }^{ \varepsilon } } \cup J_{ \nabla y_{ k }^{ \varepsilon } }
	\right)
	\right),
	\end{align}
we iteratively apply inequality (\ref{eq:a_priori_estimate_jumpset_first_step}), together with the fact that at every time step the right-hand side of  inequality  (\ref{eq:a_priori_estimate_jumpset_first_step}) is controlled by  the term  (\ref{eq:non_it_line_applied_time_1}), to deduce
\begin{align}
	\notag
	& \dfrac{1}{ \varepsilon^{ 2 } }
	\int_{\Omega'}
	W ( \nabla y_{ n }^{ \varepsilon } )
	\dd{ x }
	+
	\dfrac{1 }{ \varepsilon^{ 2 \beta } }
	\int_{\Omega'}
	\abs{ \nabla^{ 2 } y_{ n }^{ \varepsilon } }^{ 2 }
	\dd{ x }
	+
	\kappa \hm^{ 1} \left(
	\bigcup_{ k=0 }^{ n }
	\left( J_{ y_{ k }^{ \varepsilon } } \cup J_{ \nabla y_{ k }^{ \varepsilon } } \right)
	\right)
	\\
	\notag
	\leq{} &
	\energyNonlin (y_{ 0 }^{ \varepsilon } )
	+ 
	\dfrac{ 1 }{ \varepsilon }
	\int_{ 0 }^{ t_{ n }^{ \varepsilon } }
	\int_{ \lbrace \nabla y_{ \varepsilon } \in B_{ r } ( \sporth( 2 ) ) \rbrace }
	{\langle \diff W ( \nabla y_{ \varepsilon } ), \nabla  \partial_t h \rangle }
	\dd{ x }
	\dd{ t }
	\\
	\nonumber
	& + C
	\int_{ 0 }^{ 1 }
	\left(\norm{ \nabla \partial_{ t } h }_{ \lp^{ 2 } \left( \{ \nabla y_{ \varepsilon } \notin B_{ r } ( \sporth( 2 ) ) \}
	\right)}
	+
	\omega ( \Delta_{ \varepsilon } )
	\left(
	\norm{ \nabla \partial_t h  }_{ \lp^{ 2 } }
	+
	\dfrac{ 1 }{ \varepsilon^{ 2 \beta - 2 } }
	\norm{ \nabla^{ 2 } \partial_t h }_{ \lp^{ 2 } }
	\right)\right)
	\dd{ t },  
	\\
	\label{eq:a_priori_estimate_iterated}
	& +
	\int_{ 0 }^{ 1 }\int_{ \Omega' }
	\dfrac{ 2 }{ \varepsilon^{ 2 \beta - 1 } }
	\abs{ \inner*{ \nabla^{ 2 } y_{ \varepsilon } }{ \nabla^{ 2 } \partial_t h } }
	\dd{ x }
	\dd{ t },
\end{align}
where $	\energyNonlin $ is defined in \eqref{relaxed energy} and $ \omega ( \Delta_{ \varepsilon }) $ vanishes as $ \varepsilon $ tends to zero {since integrable functions are absolutely continuous.} 
By the energy estimate (\ref{eq:a_priori_estimate_one}) {and \ref{eq:W_zero_on_sod}}, the measure of $ \Omega' \setminus  \mathcal{N }_{ n-1 }^{ \varepsilon } $ vanishes uniformly in $ n $ since
\begin{equation}
	\label{eq:measure_of_non_nbhd_vanishes}
	 \mathcal{L}^2 	\big( \Omega' \setminus \mathcal{ N }_{ n - 1 }^{ \varepsilon } \big)
	\leq \dfrac{ 1 }{ r^{ 2 } }
	\int_{ \Omega' \setminus \mathcal{ N }_{n - 1 }^{ \varepsilon } }
	\mathrm{dist}^{ 2 } \left( \nabla y_{ n - 1 }^{ \varepsilon } , \sporth( 2 ) \right)
	\dd{ x }
	\lesssim 
	\int_{\Omega'} W ( \nabla y_{ n - 1 }^{ \varepsilon } )
	\dd{ x }
	\lesssim 
	\varepsilon^{ 2 } . 
\end{equation}
Applying {Hölder's inequality} to the last summand, using the uniform bound on the data \eqref{eqn:bdryData} and energy (\ref{eq:a_priori_estimate_one}), $\beta <1$, and that the measure of $ \{ \nabla y_{ \varepsilon } \notin B_{ r } ( \sporth( 2 ) ) \} $ vanishes uniformly {in time} by inequality (\ref{eq:measure_of_non_nbhd_vanishes}), the last {two lines} of the estimate (\ref{eq:a_priori_estimate_iterated}) vanish as $ \varepsilon $ tends to zero, thus yielding
a sequence $ \sigma_{ \varepsilon} \to 0 $ {independent of $ t_n^\eps $ such that \eqref{eq:a_priori_inequality_error_reduced} holds.}

 To show  the uniform bound  \eqref{eq:a_priori_estimate},  we are left with arguing that {the last term of  inequality  \eqref{eq:a_priori_inequality_error_reduced}} stays uniformly bounded.
Denote by $ \pi $ the (possibly not unique) projection onto $ \sporth( 2 ) $. Since $ \diff W $ is Lipschitz in a neighborhood of $ \sporth( 2 ) $ and $ \diff  W ( \pi \nabla y_{ \varepsilon } ) = 0 $, we deduce that, if $ \mathrm{dist}( \nabla y_{ \varepsilon } , \sporth( 2 ) ) < r $,  then 
\begin{align*}
	\abs{ \langle  \diff W ( \nabla y_{ \varepsilon } ),  \nabla \partial_t h   \rangle  }
	& =
	\abs{  \langle  \diff W ( \nabla y_{ \varepsilon } ) - \diff W ( \pi \nabla y_{ \varepsilon } ) , \nabla \partial_t h   \rangle }
	 \lesssim
	\abs{ \nabla y_{ \varepsilon } - \pi \nabla y_{ \varepsilon } }
	\abs{\nabla \partial_t h }
	\\
	& =
	\mathrm{dist} \left( \nabla y_{ \varepsilon } , \sporth( 2 ) \right)
	\abs{ \nabla \partial_t h  }
	  \lesssim
	\sqrt{ W ( \nabla y_{ \varepsilon } ) } \abs{ \nabla \partial_t h },
\end{align*}
where the last inequality is due to the lower growth bound \ref{eq:W_zero_on_sod}.
Applying the previous inequality and Hölder's inequality to the last term on the right-hand side of  inequality  \eqref{eq:a_priori_inequality_error_reduced}, we obtain  
\begin{equation*}
{\dfrac{ 1  }{ \varepsilon }}
	\int_{ 0 }^{ t_{ n }^{ \varepsilon } }
	\int_{ { \{ \nabla y_{ \varepsilon } \in B_{ r } ( \sporth( 2 ) )\} }}
		\inner*{
	\diff W ( \nabla y_{ \varepsilon } ) 
	} {\nabla \partial_t h}
	\dd{ x }
		\dd{t}
	\lesssim
		\int_{ 0 }^{ 1 }
	\left(
	\int_{\Omega'}  \dfrac{ 1 }{ \varepsilon^2 } 
	W ( \nabla y_{ \varepsilon } )
	\dd{ x } \right)^{ 1/2 }
	\norm{ \nabla \partial_t h }_{ \lp^{ 2 } }
	\dd{ t }.
\end{equation*}
By the control on  the  elastic energy (\ref{eq:a_priori_estimate_one}) we get that this term stays uniformly bounded. Inserting this information back into  inequality  \eqref{eq:a_priori_inequality_error_reduced} and using that the initial energies are bounded, see \eqref{eqn:wellPrepInitial}, we thus have  shown the desired a priori estimate \eqref{eq:a_priori_estimate}.  
\end{proof}

\begin{remark}\label{rmk:improvedEpsEnergyBalance}
 We note that the reasoning used to derive  inequality  (\ref{eq:a_priori_inequality_error_reduced}) actually shows that, for $t_n^\eps, t_m^\eps \in I_\eps$ with $t_m^\eps < t_n^\eps$, we have
\begin{align}
	\notag
	& \energyNonlinTot(t_n^\eps)
	-
	\sigma_{ \varepsilon }
	\leq{} 
	\energyNonlinTot (t_m^\eps)
	+ 
	\dfrac{ 1 }{ \varepsilon }
	\int_{ t_m^\eps }^{ t_{ n }^{ \varepsilon } }
	\int_{ { \{ \nabla y_{ \varepsilon } \in B_{ r } ( \sporth( 2 ) )\} } }
	\langle \diff W ( \nabla y_{ \varepsilon } ) ,\nabla  \partial_t h \rangle
	\dd{ x }
	\dd{ t }
\end{align}
for $\sigma_\eps\to 0$  as $\eps \to 0$  (independently of $t_m^\eps$ and $t_n^\eps$). 
\end{remark}

\subsection{From deformations to rescaled displacements} \label{subsec:deformToDisplace}

In this subsection, we define rescaled displacements that are sufficiently close to a rotation everywhere in ${\Omega'}.$ 
 Following   \cite{friedrich_griffith_energies_as_small_strain_limit_of_nonlinear_models_for_nomsimple_brittle_materials},  this requires the introduction of a partition of ${\Omega'}$, where on each piece $y_\eps$ is sufficiently close to a rigid motion. Beyond accounting for rigidity, we will need to carefully zero-out the displacement in a vanishingly small region in $\Omega$. This will later ensure that minimality is stable when passing from deformations to displacements as $\eps\to 0$, see \Cref{sec:approximateRelationsMinimality}.

 Choosing $ \gamma \in (2/3,\beta)$, for any time  $ t \in [0,1] $, we apply the reasoning in \cite[Thm.~2.3]{friedrich_griffith_energies_as_small_strain_limit_of_nonlinear_models_for_nomsimple_brittle_materials} {to} the sequence $ (y_{ \varepsilon } ( t ) )_\eps$ to obtain the modified sequence 
\begin{equation} \label{def:y_rot}
	y_{ \varepsilon}^{ \mathrm{rot} } ( t )
	:=
	\sum_{ j = 1 }^{ \infty }
	R_{ j }^{ \varepsilon } (t) y_{ \varepsilon }(t) \chi_{ P_{ j }^{ \varepsilon } (t) }
	\in \gsbvtwotwo ( \Omega' ; \mathbb{ R }^{ 2 }), 
\end{equation}
for a Caccioppoli partition $ \mathcal{P}_\eps (t) : = ( P_{ j }^{ \varepsilon } (t) )_{ j }  $ of $\Omega'$ and $( R_{ j }^{ \varepsilon } (t) )_j \subseteq \sporth ( 2)$, such that  $ R_{ j }^{ \varepsilon } (t) = \mathrm{Id} $ on all elements $P_{ j }^{ \varepsilon } (t) $ intersecting the boundary in the sense of $ \mathcal{L}^2 ( P_{ j }^{ \varepsilon } (t) \cap ( \Omega' \setminus \overline{\Omega} ) ) > 0 $. We point out that $R_j^\eps$ and the Caccioppoli partition $\mathcal{P}_\eps$ are piecewise constant in time on the intervals determined by $I_\eps$ (as in \eqref{eqn:minMovDeform}). (We frequently omit $(t)$ if no confusion arises.) {As many of the properties of this modified sequence are determined solely by the uniform energy bound in \eqref{eq:a_priori_estimate}, it holds uniformly  for all $ t \in [0,1] $ that}
\begin{enumerate}[labelindent=0pt,labelwidth=\widthof{\ref{item:rotation_property_closeness}},label=(R\arabic*),ref=(R\arabic*),itemindent=1em,leftmargin=!]
	\item 
	\label{item:rotation_properties_boundary}$ y_{ \varepsilon}^{ \mathrm{rot} }
	=
	\mathrm{id} + \varepsilon h_{ \varepsilon } $ on $ \Omega' \setminus \overline{\Omega } $, 
	\item 
	\label{item:rotation_property_jump}
	$ \hm^{ 1}
	\big( 
	\left(
	J_{ y_{ \varepsilon }^{ \mathrm{rot} }}\cup
	J_{ \nabla y_{ \varepsilon}^{ \mathrm{rot} } }
	\right)
	\setminus
	\left(
	J_{ y_{ \varepsilon } } \cup 
	J_{ \nabla y_{ \varepsilon } }
	\right)
	\big)
	\leq 
	\hm^{ 1} \Big(
	\Big( \Omega' \cap  \bigcup_{ j = 1 }^{ \infty }  \partial^{ \ast } P_{ j }^{ \varepsilon } \Big)
	\setminus
	J_{ \nabla y_{ \varepsilon } }
	\Big)
	\lesssim \varepsilon^{ \beta - \gamma } $,
	\item 
	\label{item:rotation_property_sym_grad}$ \norm{ e( y_{ \varepsilon }^{ \mathrm{rot} } ) - \mathrm{Id} }_{ \lp^{ 2 } }
	\lesssim \varepsilon $,
	\item 
	\label{item:rotation_property_grad}
	$ \norm{ \nabla y_{ \varepsilon}^{ \mathrm{rot} } - \mathrm{Id}}_{ \lp^{ 2 } }
	\lesssim \varepsilon^{ \gamma }$, and
	\item
	\label{item:rotation_property_closeness}
	{$ \abs{ \nabla y_{ \varepsilon }^{ \mathrm{rot} } - \mathrm{Id} } \leq \max \left\{ C\varepsilon^{ \gamma }, 2\mathrm{dist} \left( \nabla y_{ \varepsilon } , \sporth( 2 ) \right)\right\} $ pointwise in  $\Omega'$  for fixed $C>0$}.
\end{enumerate}
The last item is a consequence of \cite[Eq.~(4.11)]{friedrich_griffith_energies_as_small_strain_limit_of_nonlinear_models_for_nomsimple_brittle_materials}.
As a result of these properties, the rescaling 
\begin{equation}\label{eqn:tildeuEps}
	u^{ \mathrm{aux}}_{ \varepsilon } ( t) \coloneqq \dfrac{ 1 }{ \varepsilon } \left( y_{ \varepsilon }^{ \mathrm{rot}  }(t) - \mathrm{id} \right)
	\in
	\gsbvtwotwo \left( \Omega'; \mathbb{ R }^{ 2 } \right)
\end{equation}
has boundary value $ h_{ \varepsilon }$ on $ \Omega' \setminus \overline{\Omega } $. However, this will only be an auxiliary object since for the sequence we are actually interested in we will have to multiply the function by a suitable cutoff, see the comments below and \Cref{rmk:properties_of_chieps}. As a consequence of items \ref{item:rotation_property_jump}--\ref{item:rotation_property_sym_grad}  and  {the a priori estimate} \eqref{eq:a_priori_estimate},  the function $ u_{ \varepsilon }^{ \mathrm{aux} } $ satisfies
\begin{align}
	\label{eqn:uTildeAPriori}
	&\sup_{ \varepsilon > 0 } \sup_{t \in [0,1]} \Big(
	\norm{ e(  u^{ \mathrm{aux}}_{ \varepsilon }(t))    }_{ \lp^{ 2 } }
	+
	\hm^{ 1}  ( J_{ u^{ \mathrm{aux}}_{ \varepsilon}(t)} )\Big)
	< 
	\infty
 	\shortintertext{and}
	\label{eq:jumpset_does_not_increase_through_rotation}
	&
	\hm^{1} \left(
	\big( 
	J_{  u^{ \mathrm{aux}}_{ \varepsilon } } \cup J_{ \nabla u^{ \mathrm{aux} }_{ \varepsilon } } 
	\right)
	\setminus
	\left(
	J_{ y_{ \varepsilon } } \cup J_{ \nabla y_{ \varepsilon } }
	\right)
	\big)
	\lesssim
	\varepsilon^{ \beta - \gamma }.
\end{align}
For the linearization results later, we will need to apply a Taylor expansion of $ W $ at the identity with respect to the direction $ \varepsilon \nabla  u^{ \mathrm{aux}}_{ \varepsilon } $. 
In order to make sure that the error does not explode, we introduce a cutoff function  that cuts out the vanishingly small region where $\nabla u^{ \mathrm{aux}}_{ \varepsilon }$ is too big.  Whereas this procedure is by now standard in the derivation of linearized models, see for example \cite{DalMasoEtAl_2002}, \cite{friesecke_james_mueller_geometric_rigidity}, \cite{schmidt_linear_gamma_limits_of_multiwell_energies_in_nonlinear_elasticity}, \cite{braides_solci_vitali_derivation_of_linear_elastic_energies_from_pair_interaction_atomistic_systems}, and \cite{friedrich_griffith_energies_as_small_strain_limit_of_nonlinear_models_for_nomsimple_brittle_materials}, in our setting it is important that  the removed region is not only small in volume but \emph{also controlled in surface} which requires refined arguments. The cutoff function is precisely introduced in the lemma below and is based on ideas used by the
first author in \cite[Thm.~2.7]{friedrich_griffith_energies_as_small_strain_limit_of_nonlinear_models_for_nomsimple_brittle_materials}. We use the notation $u(x,t)$ for $u(t)$ evaluated at $x$.

\begin{lemma} \label{rmk:properties_of_chieps}
There exists a collection of piecewise constant functions $\eta_\eps\colon [0,1]\to (0,\infty)$, indexed by $\eps>0,$   associated characteristic functions $\chi_{ \omega_\eps(t)}$ with 
	\begin{equation*}
		\omega_{ \varepsilon }(t)  \coloneqq
		\left\{
		x  \in \Omega' \, \colon \,
		|\nabla  u^{ \mathrm{aux}}_{ \varepsilon } ( x,t ) |_{ \infty } < \eta_{ \varepsilon }(t) \right\},
	\end{equation*}
	where $  | A |_{ \infty } := \max_{ i j } \abs{ A_{ i j } } $ is the infinity norm of a given matrix, with the following properties {holding} uniformly for $ t \in [0,1] $: \begin{enumerate}[labelindent=0pt,labelwidth=\widthof{\ref{item:properties_of_etaeps_fastness}},label=($\omega$\arabic*),ref=($\omega$\arabic*),itemindent=1em,leftmargin=!]
	\item 
	\label{item:properties_of_etaeps_volume}$ \chi_{ \omega_\eps(t)} \to 1 $ in measure (or $\lp^1(\Omega')$),
	\item
	\label{item:properties_of_etaeps_perimeter} $  \hm^{ 1} ( \partial^{ \ast } \omega_{ \varepsilon }(t)\setminus J_{ \nabla u^{ \mathrm{aux} }_{ \varepsilon }(t) }  ) \to 0 $,
	\item 
	\label{item:properties_of_etaeps_slowness}
	$ \varepsilon \eta_{ \varepsilon }^{ 3 }(t) \to 0 $, and
	\item 
	\label{item:properties_of_etaeps_fastness}
	$ \varepsilon^{ 1 - \gamma } \eta_{ \varepsilon }(t) \to \infty $.
\end{enumerate}
\end{lemma}

\begin{proof}
	Let the parameters $ 2/3 < \gamma < \beta < 1 $ be given and define the sequences 
	\begin{equation*}
		\theta_{ \varepsilon }^{ - }\coloneqq \varepsilon^{ (9 \gamma -10)/12}
		\quad\text{and}\quad
		\theta_{ \varepsilon }^{ + } \coloneqq \varepsilon^{ ( \gamma -2)/4}.
	\end{equation*}
	We note that $-1/3 < (\gamma -2)/4  < (9 \gamma -10)/12  <   \gamma - 1$. In particular, $ \theta_{ \varepsilon }^{ - } <  \theta_{ \varepsilon }^{ + }$ for $\eps < 1$. 
	If $ \eta_{ \varepsilon } \in [ \theta_{ \varepsilon }^{ - }, \theta_{ \varepsilon }^{ + }] $ for all $ \varepsilon > 0 $ and $t\in [0,1]$, then properties \ref{item:properties_of_etaeps_slowness} and \ref{item:properties_of_etaeps_fastness} will hold, along with
	\begin{equation}
	\label{eq:choice_of_eta_three}
	\varepsilon^{ \beta - 1 }/ \left( \theta_{ \varepsilon}^{ + } - \theta_{ \varepsilon }^{ - }
		\right) \to 0.
	\end{equation}
	Indeed, for \ref{item:properties_of_etaeps_slowness}, we estimate
	\begin{equation*}
		\varepsilon \eta_{ \varepsilon }^{ 3 }
		\leq 
		\varepsilon^{ 1 + (3 \gamma - 6)/4}
		=
		\varepsilon^{ (3 \gamma - 2)/4},
	\end{equation*}
	which vanishes because $ \gamma > 2/3 $. To see  convergence  \ref{item:properties_of_etaeps_fastness}, we compute
	\begin{equation*}
		\varepsilon^{1 - \gamma } \eta_{ \varepsilon }
		\geq
		\varepsilon^{ 1 - \gamma + (9 \gamma - 10)/12}
		=
		\varepsilon^{ (2-3\gamma)/12}
	\end{equation*}
	which diverges since $ \gamma > 2/3 $. Lastly, for  convergence  \eqref{eq:choice_of_eta_three} we estimate
	\begin{equation*}
		\dfrac{ \varepsilon^{ \beta - 1 } }{ \varepsilon^{ ( \gamma - 2)/4 } - \varepsilon^{ (9 \gamma - 10 )/12} }
		=
		\dfrac{\varepsilon^{ \beta - 1 -(\gamma - 2)/4 } }{ 1 - \varepsilon^{ (9 \gamma - 10)/ 12 -(\gamma - 2 )/ 4 } }
		=
		\dfrac{ \varepsilon^{ ( -2 +4 \beta - \gamma )/ 4 } }{ 1 - \varepsilon^{ (6 \gamma -4)/12 } },
	\end{equation*}
	which vanishes because $ 4 \beta - \gamma > 3 \gamma > 2 $ and $ 6 \gamma > 4 $.
	
From \ref{item:properties_of_etaeps_fastness} together with  inequality  \ref{item:rotation_property_grad}, for any $\eta_\eps$ in this range, we directly deduce via Markov's inequality that $ \lm^{ 2 } ( \Omega' \setminus \omega_{ \varepsilon} ) \to 0 $ uniformly in time, which is \ref{item:properties_of_etaeps_volume}.
	
To prove  convergence  \ref{item:properties_of_etaeps_perimeter}, it suffices to estimate the perimeter for fixed $\eps>0$ and $t\in I_\eps$ by an appropriate sequence $\rho_\eps \to 0$. Using the coarea formula for $ \gsbv $-functions from \cite[Thm.~4.34]{ambrosio_fusco_pallara_functions_of_bv_and_free_discontinuity_problems} (suppressing dependence on $t$), and H\"older's inequality, we deduce that for all $ i, j \in \{ 1, 2 \} $ we have 
	\begin{align*}
		\notag
	 	\dfrac{ 1 }{ \theta_{ \varepsilon }^{ + } - \theta_{ \varepsilon }^{ - } }  	\int_{ \theta_{ \varepsilon }^{ - } }^{ \theta_{ \varepsilon }^{ + } }
		\hm^{ 1 } \left(
		\partial^{ \ast }  
		\{
		\partial_{ i } {(u^{ \mathrm{aux} }_{ \varepsilon })^j} > s \}
				\setminus
		J_{ \nabla u^{ \mathrm{aux}}_{ \varepsilon } }
		\right)
		\dd { s }
		& \leq
		\dfrac{ 1 }{ \theta_{ \varepsilon }^{ + } - \theta_{ \varepsilon }^{ - } }
		\int_{ - \infty }^{ \infty }
		\hm^{ 1 } \left(
		\partial^{ \ast }  
		 \{
		\partial_{ i } {(u^{ \mathrm{aux} }_{ \varepsilon })^j} > s  \}
				\setminus
		J_{ \nabla u^{ \mathrm{aux}}_{ \varepsilon } }
		\right)
		\dd { s }
		\\
		\notag
		& \leq
		\dfrac{ 1 }{ \theta_{ \varepsilon }^{ + } - \theta_{ \varepsilon }^{ - } }
		\norm{ \nabla^{ 2 } u^{ \mathrm{aux} }_{ \varepsilon } }_{ \lp^{ 1 }(\Omega') }
		   = 
		\dfrac{ 1 }{ \theta_{ \varepsilon }^{ + } - \theta_{ \varepsilon }^{ - } }
		\dfrac{ 1 }{ \varepsilon }
		\norm{ \nabla^{ 2 } y_{ \varepsilon } }_{ \lp^{ 2 }(\Omega') }
		\\
				& \lesssim 
		\dfrac{ \varepsilon^{ \beta-1 } }{ \theta_{ \varepsilon }^{ + } - \theta_{ \varepsilon }^{ - } },
	\end{align*}
 	where the last inequality uses that the energy of $ y_{ \varepsilon } $ stays uniformly bounded   due  to Lemma \ref{lem:aprioriEst}. Applying the same reasoning in $[-\theta_\eps^+,-\theta_\eps^-]$, this implies that we can choose $ \eta_{ \varepsilon } \in [ \theta_{ \varepsilon }^{- }, \theta_{ \varepsilon }^{ + } ] $ such that
	\begin{align*}
	 \hm^{ 1 } \left( \partial^{ \ast  } \omega_{ \varepsilon }  \setminus J_{ \nabla u^{ \mathrm{aux} }_{ \varepsilon } }
		\right)
		  & \leq{}
		\sum_{ i j }
		\hm^{ 1 } \big(
		\partial^{ \ast } 
		 \{
		\partial_{i } {(u^{ \mathrm{aux} }_{ \varepsilon })^j} \geq \eta_{ \varepsilon }
		 \}
			\setminus
		J_{ \nabla u^{ \mathrm{aux} }_{ \varepsilon } }
		\big)
		+
		\hm^{ 1 } \big(
		\partial^{ \ast }  
	 \{
		\partial_{i } {(u^{ \mathrm{aux} }_{ \varepsilon })^j} \leq -\eta_{ \varepsilon }
		 \}
		 		\setminus
		J_{ \nabla u^{ \mathrm{aux} }_{ \varepsilon } }
		\big) \notag \\ &
		\leq C\dfrac{ \varepsilon^{ \beta-1 } }{ \theta_{ \varepsilon }^{ + } - \theta_{ \varepsilon }^{ - } } =:\rho_\eps,
	\end{align*}
	where we used that $\partial^{ \ast  }( \Omega' \setminus \omega_{ \varepsilon }) = \partial^{ \ast  } \omega_{ \varepsilon }$, $ \hm^{ 1 } (\partial^{ \ast } ( A \cup B ) ) \leq \hm^{ 1 } ( \partial^{ \ast } A ) + \hm^{ 1 } ( \partial^{ \ast } B ) $, and 
	\begin{equation*}
		\Omega' \setminus \omega_{ \varepsilon }
		=
		\bigcup_{ i j } \left\{
		\partial_{ i } {(u^{ \mathrm{aux} }_{ \varepsilon })^j}\geq \eta_{ \varepsilon } \right\}
		\cup
		\left\{
		\partial_{ i } {(u^{ \mathrm{aux} }_{ \varepsilon })^j} \leq -\eta_{ \varepsilon } \right\}.
	\end{equation*}
	Note that $\rho_\eps$ vanishes uniformly in time by    \eqref{eq:choice_of_eta_three}, which finishes the proof.
\end{proof}

Finally, recalling the definition of $ u^{ \mathrm{aux} }_\eps$ in \eqref{eqn:tildeuEps}, we define the modified displacements  by
\begin{equation}\label{def:u_eps} 
u_{ \varepsilon } (t)  \coloneqq \chi_{ \omega_\eps(t)}  u^{ \mathrm{aux} }_{ \varepsilon } (t),
\end{equation} 
where $ \omega_\eps(t) $ is as in Lemma \ref{rmk:properties_of_chieps}. 
Note that, since $ \eta_{ \varepsilon } \to \infty $ and $ h \in \lp^{\infty }  ( (0, 1 ); \wkp^{ 2 , \infty }  ( \Omega' ;  \mathbb{ R }^{ 2 }  )   ) $, we have that $ \Omega' \setminus \overline{\Omega} \subseteq \omega_{ \varepsilon } $ for $ \varepsilon $ sufficiently small independently of time. Thus, $ u_{\varepsilon } $ still satisfies the boundary conditions $ u_{ \varepsilon } = h_{ \varepsilon } $ on $ \Omega' \setminus \overline{ \Omega } $. Moreover, as a consequence of  convergence  \ref{item:properties_of_etaeps_perimeter} and estimate (\ref{eq:jumpset_does_not_increase_through_rotation}), we obtain  
\begin{equation}
	\label{eq:modifications_barely_increase_jump}
	\hm^{ 1} \big(
	\big(
	J_{ u_{ \varepsilon } }
	\cup
	J_{ \nabla u_{ \varepsilon } }
	\big) 
	\setminus
	\big(
	J_{ y_{ \varepsilon }}
	\cup
	J_{ \nabla y_{ \varepsilon } }
	\big)
	\big)
	\to 0 \quad \text{ uniformly in time  as $\eps\to 0$}.
\end{equation}

\subsection{Compactness of the displacements}

We use a diagonal argument  to identify the target displacement $u(t)$ for $ t \in I_{ \infty } $ as a suitable limit of $u_\eps$ defined in \eqref{def:u_eps}, see also \eqref{def:y_rot} and \eqref{eqn:tildeuEps}. Motivated by  \Cref{thm:gsbd2_compactness}\ref{item:rescaled_un }, the function
\begin{equation}
	\label{eq:veps_def}
 	v_{ \eps }(t) \coloneqq
 	u_{ \eps }(t)
 	- \chi_{ S_{ 0 }^{ \eps }(t) } (u_{ \eps } (t) -  h_\eps(t)  )
 	- \sum\nolimits_{ j } a_{ j }^{ \eps }(t) \chi_{ S_{ j }^{ \eps }(t) },
\end{equation}  
for {a} suitable {collection $\mathcal{S}_\eps (t) : = (S_j^\eps(t))_{j\ge 0 }$ of disjoint sets of finite perimeter} and infinitesimal rigid motions $(a_j^\eps(t))_{ j\ge 1}$, is of special importance to us since by  subtraction of a piecewise infinitesimal rigid motion pointwise a.e.\ convergence can be guaranteed. We also note that $v_\eps(t) = h_\eps(t)$ on $\Omega' \setminus \overline{\Omega}$ by construction. Using that the previous modifications barely increase the jump set, see inequality (\ref{eq:modifications_barely_increase_jump}),  and the compactness property   \Cref{thm:gsbd2_compactness}\ref{item:caccioppoli_partitions_small_jump}, we see that, for each $t \in [0,1]$,
\begin{equation}
	\label{eq:jump_of_veps_controlled_by_yeps}
	\hm^{ 1 } \big(
	J_{ v_{ \varepsilon } ( t ) } 
	\setminus
	\big( J_{ y_{ \varepsilon } ( t ) } \cup J_{ \nabla y_{ \varepsilon } ( t ) }
	\big)
	\big)
	{ 	\to 0  } \quad \text{ as $\eps\to 0$}.
\end{equation}

\begin{remark}[Limit displacement I]\label{def:displaceAtGoodTimes}
 We define an auxiliary limit displacement. 

\textbf{Limit for $t\in I_\infty$.}\label{def:displacement I}
We apply the compactness \Cref{thm:gsbd2_compactness} to the sequence $ ( u_{ \varepsilon } ( t ) )_{ \varepsilon } $ which yields for a subsequence (not relabeled, $t$-dependent)  a sequence $(v_\eps(t))_\eps$ as in equation \eqref{eq:veps_def}   and a function $ u ( t ) \in \gsbdtwo ( \Omega' ) $ with $ u ( t ) = h ( t ) $ on $ \Omega' \setminus \overline{\Omega } $ such that 
\begin{enumerate}[labelindent=0pt,labelwidth=\widthof{\ref{item:ueps_lsc_jumpset_wrt_u}},label=(C\arabic*),ref=(C\arabic*),itemindent=1em,leftmargin=!]
	\item \label{item:ueps_converges_in_measure_to_u}
	$v_{ \varepsilon}{(t)} \to  u ( t ) $ pointwise almost everywhere in $\Omega'$,
	\item \label{item:weak_converg_symmetric_grads}$ e(u_{ \varepsilon } ( t ) ) \rightharpoonup e( u ( t ) ) $ weakly in $ \lp^{ 2 } ( \Omega' ; \mathbb{ R }^{ 2 \times 2 }_{ \mathrm{sym} }) $,
	\item \label{item:ueps_lsc_jumpset_wrt_u}$ \hm^{ 1 } ( J_{  u ( t ) } \cap U ) \leq \liminf_{ \varepsilon \to 0 } \hm^{ 1 } ( J_{ u_{ \varepsilon } ( t ) } \cap U ) $ for all open Lipschitz subsets $ U \subseteq \Omega' $.
\end{enumerate}
Note that by additionally using the compactness properties \ref{item:caccioppoli_partitions_small_jump} and \ref{item:S0n_vanishes} from \Cref{thm:gsbd2_compactness}, we obtain that $ v_{ \varepsilon }  (t)  $ in place of $ u_{ \varepsilon }  (t)  $ also enjoys the properties \ref{item:weak_converg_symmetric_grads} and \ref{item:ueps_lsc_jumpset_wrt_u}.

\textbf{Limit for $t \in [0,1]\setminus I_\infty$.} We point out that we can apply the same reasoning to find a limit displacement $\hat{u}(t)$ at any time $t\in [0,1]$ such that $u_\eps(t)$  converges to  $\hat{u}(t)$ in the sense above. We denote this displacement by $\hat{u}$ to emphasize that a single subsequence in $\eps$ cannot be used for all $t$, given that $[0,1]$ is uncountable, and there is no reason that $\hat{u}$ should be measurable in time. However, we define $\hat u (t) :=u(t)$ for $t\in I_{\infty}.$
\end{remark}

 Eventually,   the function $ u $ is defined as follows.

\begin{definition}[Limit displacement II and the evolving crack]\label{def:displacement II}

We can finally define the limiting displacement $u$  satisfying $u(t) = h(t)$ on $\Omega' \setminus \overline{\Omega}$ for all $t \in [0,1]$. 

 \textbf{Limit for $t\in I_\infty$.}  	Let $u(t)$ be as in \Cref{def:displacement I} for $t\in I_\infty$. 	 By   a diagonalization argument, we can find a single subsequence in $\eps$ that satisfies \ref{item:ueps_converges_in_measure_to_u}--\ref{item:ueps_lsc_jumpset_wrt_u}  for all times $ t \in I_{ \infty } $. 
	
	\textbf{Limit for $t \in [0, 1 ] \setminus I_{ \infty }$.}
	Take $ (t_{ p } )_{p \in \mathbb{ N } } \subseteq I_{\infty } $ such that $ t_{ p } \uparrow t $. The $ \lp^{ 2 } $-norm of the symmetric gradients and the size of the jump sets of the sequence $ ( u ( t_{ p } ) )_{ p }$ are uniformly bounded by the (not shown yet) lower semicontinuity \eqref{eq:lsc_of_energies} and the a priori estimates in Lemma \ref{lem:aprioriEst}. Thus, by the compactness Theorem \ref{thm:gsbd2_compactness}, we can find a  subsequence (not relabeled) and a function $ u( t ) \in \gsbdtwo ( \Omega' ) $ such that $ u( t_{ p } ) $ converges to $ u( t ) $ in the sense of \Cref{thm:gsbd2_compactness}.  Regularity of the boundary condition $h$ in \eqref{eqn:bdryData} ensures that $u(t) = h(t)$ on $\Omega'\setminus \overline{\Omega}$.

 \textbf{Crack set.}  	Moreover, we introduce the irreversible crack at time $ t \in [0,1 ] $ as
	\begin{equation}
		\label{eq:irreversible_crack}
		\Gamma ( t ) 
		\coloneqq 
		\bigcup_{ \tau \in I_{ \infty }^{ t } }
		J_{ u ( \tau ) }.
	\end{equation}
\end{definition}

We introduce two notions for the limit displacement as each is amenable to different tools allowing us to recover different properties in the limit. We emphasize that the displacements $u$ and {$
\hat u$ constructed above} may not be unique or even measurable (in
space-time) {and, a priori, $u \neq \hat u$.} However, in addition to minimality, \Cref{lemma:cont_extension} shows that the linear strain $e(u) $
is in fact measurable and is uniquely identified by the approximate subsequence $u_{ \varepsilon }$. For this, we prove that the other way of obtaining the limit, namely the function $ \hat{u} $ from  \Cref{def:displacement I},  gives rise to the same symmetric gradient. In fact, since $e(u) = e(\hat u)$, we will be able to show that the limit $u(t) = \hat{u}(t)$ is uniquely determined on the good set $G(t)$ still attached to the boundary $\Omega'\setminus \overline{\Omega}$, see convergence \eqref{baruconv0} above.
 
\section{Stability of minimality as $\eps\to 0$} \label{sec:approximateRelationsMinimality}

In this section, we prove that $u$ (and $\hat u$) found in  Remark \ref{def:displacement I} and Definition \ref{def:displacement II}  are minimizers of the Griffith energy, excluding any previously created crack. 

\subsection{Lower semicontinuity of the energies}

 Recall {$\energyLin$ defined in}  equation  (\ref{linen}). In this subsection, our  first  goal is  to show that for all $ t \in I_{ \infty } $ it holds that
\begin{align}
	\label{eq:lsc_of_energies}
	\energyLin ( u ( t ) )
	& \leq 
	\liminf_{ \varepsilon \to 0 }
	\dfrac{ 1 }{ \varepsilon^{ 2 } }
	\int_{ \Omega' }
	W ( \nabla y_{ \varepsilon } ( t ) )
	\dd{ x }
	+
	\kappa \hm^{ 1 } \left(J_{ y_{\varepsilon }( t ) } \cup J_{ \nabla y_{ \varepsilon } ( t ) } \right),
\end{align}
where $u(t)$ is as in  Definition \ref{def:displacement II}.

For the proof of the lower semicontinuity (\ref{eq:lsc_of_energies}), we use the same strategy as in the proof of  the $ \Gamma $-liminf   inequality  \cite[Thm.2.7]{friedrich_griffith_energies_as_small_strain_limit_of_nonlinear_models_for_nomsimple_brittle_materials} (differences arising as we do not `zero-out' regions that travel off to infinity). 
First, we note that on $ \omega_{ \varepsilon } $ (see Lemma \ref{rmk:properties_of_chieps}), we have that $ \varepsilon \abs{ \nabla u^{ \mathrm{aux}}_{ \varepsilon } } \leq \varepsilon \eta_{ \varepsilon } \to 0 $, see  convergence  \ref{item:properties_of_etaeps_slowness}.  
 As $ \diff W ( \mathrm{Id} ) = 0 $ and as  the tangent space of  $ \sporth ( 2 ) $  at the identity is  the space of  skew-symmetric matrices, we  have that $ Q ( A ) = Q ( (A + A^{ \top } )/ 2 ) $  (recall \eqref{eq: QQQ}).  
 Remember the definition of $u_\eps$ in \eqref{def:u_eps}, see also \eqref{def:y_rot}--\eqref{eqn:tildeuEps}. 
 By applying a Taylor expansion in a ball of radius $\eps\eta_\eps$, we thus have on $ \omega_{ \varepsilon } $ via the rotational invariance of $ W $ that
\begin{equation*}
	W (\nabla  y_{ \varepsilon } )
	=
	W (\nabla  y_{ \varepsilon }^{ \mathrm{rot} } )
	=
	W ( \mathrm{Id} + \varepsilon \nabla u_{ \varepsilon } )
	\geq
	\eps^2 \dfrac{ 1 }{ 2 } Q ( e ( u_{ \varepsilon } ) ) 
	-
	C
	\varepsilon^{ 3 } \eta_{ \varepsilon }^{ 3 }.
\end{equation*}
 By   definition  \eqref{def:u_eps}, the weak convergence of $e(u_\eps)$ to $e(u)$  (see   \Cref{def:displacement II} for $t \in I_\infty$),  and the convexity of $ Q $, it follows that (suppressing the dependence on $ t \in I_{ \infty } $)
\begin{align}
	\notag
	\int_{ \Omega' }
	\dfrac{ 1 }{ 2 }
	Q( e ( u ) )
	\dd{ x }
	& \leq
	\liminf_{ \varepsilon \to 0 }
	\int_{ \Omega' }
	\dfrac{ 1 }{ 2 }
	Q ( e ( u_{ \varepsilon } ) )
	\dd{ x }
	 =
	\liminf_{ \varepsilon \to 0 }
	\int_{ \omega_{ \varepsilon } }
	\dfrac{ 1 }{ 2 }
	Q ( e ( u^{ \mathrm{aux}}_{ \varepsilon } ) )
	\dd{ x }
	\\
	\label{eq:lsc_for_quadratic_term}
	& \leq
	\liminf_{ \varepsilon \to 0 }\left[
	\dfrac{ 1 }{ \varepsilon^{ 2 } }
	\int_{ \Omega' }
	W ( \nabla y_{ \varepsilon } )
	\dd{ x }
	+ C \eps \eta_{ \varepsilon }^{ 3 } \right].
\end{align}
By  convergence  \ref{item:properties_of_etaeps_fastness} the last summand vanishes. Combining the fact that the modifications barely increase the jump set by inequality (\ref{eq:modifications_barely_increase_jump}) with the lower semicontinuity  estimate in  \ref{item:ueps_lsc_jumpset_wrt_u}  we get 
\begin{equation}
	\label{eq:lsc_for_jump_set}
	\hm^{ 1 } \left( J_{ u ( t ) } \cap U
	\right)
	\leq
	\liminf_{ \varepsilon \to 0 }
	\hm^{ 1 } \left( 
	\left( J_{ y_{ \varepsilon } ( t ) } \cup J_{ \nabla y_{ \varepsilon } ( t ) } \right)
	\cap
	U\right)
\end{equation}
for all open Lipschitz sets $ U \subseteq \Omega' $.
Now, the estimates (\ref{eq:lsc_for_quadratic_term}) and (\ref{eq:lsc_for_jump_set}) yield the desired lower semicontinuity (\ref{eq:lsc_of_energies}).

As in \cite{francfort_larsen_existence_and_convergence_for_quasi_static_evolution_in_brittle_fracture}, we  also  argue that the total crack energy is lower semicontinuous, but we use a different argument here. Precisely, recalling the definition of $\Gamma(t)$ in \eqref{eq:irreversible_crack}, for all $ t \in I_{ \infty } $, it holds that
\begin{equation}
	\label{eq:lsc_jump_sets_union_times}
	\hm^{ 1 } \left( \Gamma ( t ) \right)
	\leq
	\liminf_{ \varepsilon \to 0 }
	\hm^{ 1 } \left(
	\Gamma_{ \varepsilon } ( t )
	\right)
	\leq C.
\end{equation}
The last inequality has already been established by  the  a priori estimate (\ref{eq:a_priori_estimate}). Thus, we are left with showing the first inequality.
First, we note that it suffices to show that, given any number of points $ t_{ 1 } , \ldots, t_{ n } \in  I^t_{ \infty }  $, we have that
\begin{equation*}
	\hm^{ 1 } \left(
	\bigcup_{ k = 1 }^{ n }
	J_{ u ( t_{ k } ) }
	\right)
	\leq 
	\liminf_{ \varepsilon \to 0 }
	\hm^{ 1 } \left( \Gamma_{ \varepsilon } (t)
	\right)
\end{equation*}
since we can then argue by continuity from below sending $n \to \infty$. Let such points be given. Then, given any $  \epsilon  > 0 $, by a standard measure theory argument (\Cref{lemma:separating_with_disjoint_open_sets}) we find  open, pairwise disjoint Lipschitz sets $ U_{ 1 } , \ldots , U_{ n } $ such that 
\begin{align}
	\notag\hm^{ 1 } 
	\left(
	\bigcup_{ k = 1 }^{ n }
	J_{ u ( t_{ k } ) }
	\right)
	-   \epsilon 
	& \leq
	\sum_{ k = 1 }^{ n }
	\hm^{ 1 } \left( J_{ u( t_{ k } ) } \cap U_{ k } \right)
	\leq
	\sum_{ k = 1 }^{ n }
	\liminf_{ \varepsilon \to 0 }
	\hm^{ 1 } \left( \left( J_{ y_{ \varepsilon } ( t_{ k } ) } \cup J_{ \nabla y_{ \varepsilon } ( t_{ k } ) } \right) \cap U_{ k } \right)
	\\
	& \label{for the end}\leq
	\liminf_{ \varepsilon \to 0 }
	\hm^{ 1 } \left(
	\bigcup_{ k = 1 }^{ n }
	J_{y_{ \varepsilon } ( t_{k } )} \cup J_{ \nabla y_{ \varepsilon } ( t_{ k } ) } 
	\right)
	 \leq
	\liminf_{ \varepsilon \to 0 }
	\hm^{ 1 } \left(
	\Gamma_{ \varepsilon } ( t )
	\right).
\end{align}
Here, the second inequality follows from inequality (\ref{eq:lsc_for_jump_set}).

\subsection{Almost minimality}

The goal of this subsection is to prove that the functions $ u_{ \varepsilon } $ are `almost' minimizers of the linearized functional, in the sense of the following lemma. An essential piece of this result is the quantification of  the meaning of  `almost' for a given test function, which is strong enough to pass to the limit in both the minimality property and an associated  variational inequality,    see  Lemma~\ref{lemma:u_minimzer_of_total_energy} and  Lemma   \ref{cor:approx_euler_lagrange_consequence} below.

\begin{lemma}
	\label{lemma:linearization_of_minimization_problem}
	For any $ t \in I_{ \infty } $ and any $ \phi \in \gsbvtwotwo ( \Omega' ;  \mathbb{ R }^{ 2 }  ) $ with  $ \phi = 0 $ on $ \Omega' \setminus \overline{\Omega } $  and $|\nabla \phi| \in \lp^3(\Omega')$,  there exists a sequence $ \rho_{ \varepsilon} = \rho_{ \varepsilon} (\phi)  \to 0 $ such that we have 
	\begin{align*}
		&\int_{ \Omega' }
		\dfrac{ 1 }{2 }
		Q \big( e ( u_{ \varepsilon } ( t ) ) \big)
		\dd{ x }
		- \rho_{ \varepsilon }
		\le
		\int_{ \Omega' }
		\dfrac{ 1 }{ 2 }
		Q \big( e( u_{ \varepsilon }( t ) +\phi ) \big) 
		\dd{ x }
		+
		\kappa
		\hm^{ 1 } \left(  \left( J_{ \phi } \cup J_{ \nabla \phi } \right) \setminus 
		\Gamma_{ \varepsilon } ( t ) 
		\right).
	\end{align*}
	Moreover, there exists $C > 0$ and $\delta_{ \varepsilon} \to 0  $ independently of $ t $ and $ \phi $ such that
	\begin{equation}
		\label{eq:estimate_rhoeps}
		\rho_{ \varepsilon } \leq C \delta_{ \varepsilon} 
		\left( 
		1 + 
		\int_{ \Omega' }
		\abs{ \nabla \phi }^{ 2 }
		\dd{ x }
		+
		\int_{ \Omega' }
		\abs{ \nabla^{ 2 } \phi }^{ 2 }
		\dd{ x }
		\right)
		+
		C
		\int_{ \Omega' }
		\varepsilon \abs{ \nabla \phi }^{ 3 }
		\dd{ x }. 
	\end{equation}
\end{lemma}


\begin{proof}
	The idea of the proof is to use the minimality property \eqref{eq:discrete_min_problem} for $y_\eps$ and  to  apply the linearization procedure, but  to  keep track of the terms we throw away.
	For notational simplicity, we suppress the dependence on {$ t $ for $y_\eps$ and $u_\eps$.}

	Take $ \varepsilon $ sufficiently small such that $ t \in I_{ \varepsilon } $ and let $ z \in \gsbvtwotwo ( \Omega'; \mathbb{ R }^{ 2 } ) $ with $ z=\mathrm{id} +\varepsilon h_{ \varepsilon}  $ on $ \Omega' \setminus \overline{\Omega} $.  Since $ y_{ \varepsilon } $ is a minimizer of the functional (\ref{eq:min_problem_wrt_own_jumpset}), we have  that
	\begin{equation}
		\dfrac{ 1 }{ \varepsilon^{ 2 } }
		\int
		W ( \nabla y_{ \varepsilon })
		\dd{ x }
		+
		\dfrac{ 1 }{ \varepsilon^{ 2 \beta } }
		\int
		\abs{ \nabla^{ 2 } y_{ \varepsilon } }^{ 2 }
		\dd{ x }
		\label{eq:yeps_minimizer_of_own_jump_set}
		\leq
		\dfrac{ 1 }{ \varepsilon^{ 2 } }
		\int
		W( \nabla z )
		\dd{ x }
		+
		\dfrac{ 1 }{ \varepsilon^{ 2 \beta } }
		\int
		\abs{ \nabla^{ 2 } z }^{ 2 }
		\dd{ x }
		+
		\kappa
		\hm^{ 1 } \big(
		\left(
		J_{ z } \cup J_{ \nabla z }\right)
		\setminus 
		\Gamma_{ \varepsilon } ( t )
		\big).
	\end{equation} 
	We first estimate the left-hand side of this inequality. As in inequality (\ref{eq:lsc_for_quadratic_term}), we obtain via a Taylor expansion that
	\begin{align*}
		\dfrac{ 1 }{ \varepsilon^{ 2 } }
		\int_{ \Omega' } 
		W ( \nabla y_{ \varepsilon } )
		\dd{ x }
		& \geq
		\dfrac{ 1 }{ \varepsilon^{ 2 } }
		\int_{ \Omega' }
		\chi_{ \omega_{ \varepsilon} }
		W ( \mathrm{Id} + \varepsilon u_{ \varepsilon } ) 
		\dd{ x }
		  \geq
		\int_{ \Omega'}
		\dfrac{ 1 }{ 2 }
		Q ( e(u_{ \varepsilon } ) )
		\dd{ x }
		-  C\eps\eta_\eps^3. 
	\end{align*}
 Therefore,  by choosing $\delta_\eps \ge \eps\eta_\eps^3$ in  inequality \eqref{eq:estimate_rhoeps}   we conclude that
	\begin{equation}
		\label{eq:lsc_for_minim_problem}
		\dfrac{ 1 }{\varepsilon^{ 2 } }
		\int_{ \Omega' }
		W ( \nabla y_{ \varepsilon } )
		\dd{ x }
		+
		\dfrac{ 1 }{ \varepsilon^{ 2 \beta } }
		\int_{ \Omega' }
		\abs{ \nabla^{ 2 } y_{ \varepsilon } }^{ 2 }\dd{ x }
		+ \rho_{ \varepsilon}
		\geq
		\int_{ \Omega' }
		\dfrac{ 1 }{ 2 }
		Q ( e ( u_{ \varepsilon } ) )
		\dd{ x }
		+
		\dfrac{ 1 }{ \varepsilon^{ 2 \beta } }
		\int_{ \Omega' }
		\abs{ \nabla^{ 2 } y_{ \varepsilon } }^{ 2 }\dd{ x } 
		.
	\end{equation}
	Now, let us take care of the right-hand side of inequality \eqref{eq:yeps_minimizer_of_own_jump_set}. 
	For $\phi$ as in the  statement,  define the sequence $ z_{ \varepsilon } \coloneqq \mathrm{id}+\varepsilon (u_{ \varepsilon } +\phi ) $. The function $ z_{ \varepsilon } $ is an admissible competitor in  (\ref{eq:yeps_minimizer_of_own_jump_set}).  Below we will show that there exists a sequence $ \rho_{ \varepsilon } \to 0 $ satisfying inequality (\ref{eq:estimate_rhoeps}) such that
	\begin{align}\label{finallytoshow}
		&\dfrac{ 1 }{ \varepsilon^{ 2 } }
		\int_{ \Omega' }
		W( \nabla z_{ \varepsilon } )
		\dd{ x }
		+
		\dfrac{ 1 }{ \varepsilon^{ 2 \beta } }
		\int_{ \Omega' }
		\abs{ \nabla^{ 2 } z_{ \varepsilon } }^{ 2 }
		\dd{ x }
		+
		\kappa
		\hm^{ 1 } \big(
		\left(
		J_{ z_{ \varepsilon } } \cup J_{ \nabla z_{ \varepsilon } }\right)
		\setminus 
		\Gamma_{ \varepsilon } ( t )
		\big)\notag
		\\
		\leq{} &
		\dfrac{ 1 }{ 2 }
		\int_{ \Omega' } 
		Q ( e( u_{ \varepsilon } + \phi ) ) \dd{ x }
		+
		\dfrac{ 1 }{ \varepsilon^{ 2 \beta } }
		\int_{ \Omega' } 
		\abs{ \nabla^{ 2 } y_{ \varepsilon } }^{ 2 }
		\dd{ x }
		+
		\kappa \hm^{ 1 } \big(  \left( J_{ \phi } \cup J_{ \nabla \phi } \right) \setminus 
		\Gamma_{ \varepsilon } ( t )
		\big)
		+
	 \rho_{ \varepsilon }.
	\end{align}
Once this is shown,  combining this  with inequality  (\ref{eq:lsc_for_minim_problem}) and the minimization property (\ref{eq:yeps_minimizer_of_own_jump_set}), we obtain the desired result  (for $2\rho_\eps$ in place of $\rho_\eps$).

 Let us now come to the proof of  inequality \eqref{finallytoshow}.  		Similarly to the a priori estimate  in  \Cref{lem:aprioriEst}, we want to  perform  a Taylor expansion on $ W ( \nabla z_{ \varepsilon } ) $ in a  neighborhood   of the identity and use the growth conditions of $ W $ away from the identity.
Let $ r  \in (0,1)  $  be  such that $ W $ is $ \cont^{ 3 } $ on $ B_{ 2 r } ( \mathrm{Id} ) $. 
Moreover, recall from the choice of the cutoff (see \ref{item:properties_of_etaeps_slowness}) that $ \varepsilon \abs{ \nabla u_{ \varepsilon } } \to 0 $ uniformly. In particular, this means that for $ \varepsilon $ sufficiently small, $ \varepsilon \abs{ \nabla \phi } < r $ already implies $ \varepsilon \abs{ \nabla ( u_{ \varepsilon } + \phi ) } < 2 r $.
		We start by estimating  the integral on    $ \{ \varepsilon \abs{ \nabla \phi } \leq r\} $ using  a Taylor expansion. This  yields
		\begin{align*}
		\notag 		\dfrac{ 1 }{ \varepsilon^{ 2 } }
		\int_{ \lbrace \varepsilon \abs{ \nabla \phi } \leq r \rbrace }
		W( \nabla z_{ \varepsilon } )
		\dd{ x }
		& \leq
		\int_{ \Omega' }
		\dfrac{1}{2}
		Q ( e( u_{ \varepsilon } + \phi ) )
		\dd{ x }
		+
		C \varepsilon
		\int_{ \Omega' }
		\abs{ \nabla ( u_{ \varepsilon } + \phi ) }^{ 3 }
		\dd{ x }
		\\
				&\leq
		\int_{ \Omega' } 
		\dfrac{1}{2}
		Q ( e( u_{ \varepsilon } + \phi ) )
		\dd{ x }
		+
		C \varepsilon \eta_{ \varepsilon }^{ 3 }
		+
		C \varepsilon 
		\int_{ \Omega' }
		\abs{ \nabla \phi }^{ 3 }
		\dd{ x}.
		\end{align*}
	The last   two terms can be   absorbed into $ \rho_{ \varepsilon } $ {while satisfying}   (\ref{eq:estimate_rhoeps}), due to  the choice of the cutoff parameter $ \eta_{ \varepsilon } $ in \ref{item:properties_of_etaeps_slowness}, if we choose $ \delta_{ \varepsilon } \geq \varepsilon \eta_{ \varepsilon }^{ 3 } $.

Next, we estimate  	
		away from the identity. By the local Lipschitz continuity of $ W $, see \ref{eq:W_growth_assumptions},  Young's inequality,  and  {the} fact that  $   \abs{ \nabla u_{ \varepsilon } }  \lesssim  \eta_\eps$,
		 which implies $ \varepsilon \abs{ \nabla \left( u_{ \varepsilon } + \phi \right) } \gtrsim 1 $ on $ \{ \varepsilon \abs{ \nabla \phi } \geq r \} $ for $ \varepsilon $ sufficiently small, we have 
		\begin{align}
			\notag
			\dfrac{ 1 }{ \varepsilon^{ 2 } }
			\int_{\lbrace \varepsilon \abs{ \nabla \phi } \geq r \rbrace }
			W ( \nabla z_{ \varepsilon } )
			\dd{ x }
			& =
			\dfrac{ 1 }{ \varepsilon^{ 2 } }
			\int_{ \lbrace  \varepsilon \abs{ \nabla \phi } \geq r \rbrace }
			W( \mathrm{Id} + \varepsilon \nabla ( u_{ \varepsilon } + \phi ) )
			\dd{ x }
			\\
			\notag
			& \lesssim
			\dfrac{ 1 }{ \varepsilon^{ 2 } }
			\int_{  \lbrace \varepsilon \abs{ \nabla \phi } \geq r  \rbrace }
			\left( 1 + \varepsilon \abs{ \nabla ( u_{ \varepsilon } + \phi ) } \right)
			\varepsilon \abs{ \nabla ( u_{ \varepsilon } + \phi ) }
			\dd{ x }
			\\
			\notag
			&
			\lesssim
			\int_{ \lbrace   \varepsilon \abs{ \nabla \phi } \geq r \rbrace  }
			\abs{ \nabla u_{ \varepsilon } }^{ 2 }
			\dd{ x }
			+
			\int_{ \lbrace  \varepsilon \abs{ \nabla \phi } \geq r  \rbrace }
			\abs{ \nabla \phi }^{ 2 }
			\dd{ x }
						\\
						\label{eq:approx_minimality_away_from_id}
			&  \lesssim
			\dfrac{\eta_{ \varepsilon }^{ 2 }\eps^2}{r^{ 2 } }
			\int_{ \Omega' }
			\abs{ \nabla \phi }^{2 }
			\dd{ x }
			+
			\dfrac{\eps}{r}\int_{  \Omega'  }
			\abs{ \nabla \phi }^{ 3 }
			\dd{ x }.
		\end{align}
		By  convergence  \ref{item:properties_of_etaeps_slowness}, the term  $ \eta_{ \varepsilon }^{ 2 } \varepsilon^{ 2 } $ vanishes as $ \varepsilon $ tends to zero. Thus, by choosing $\delta_\eps \ge \eta_\eta^2 \eps^2$,  the term (\ref{eq:approx_minimality_away_from_id}) can be absorbed into $ \rho_{ \varepsilon } $, see  inequality  (\ref{eq:estimate_rhoeps}).

	Now, let us take care of the second derivative term $ \nabla^{ 2 }z_{ \varepsilon } $. 
	We introduce a parameter $ \varepsilon^{ 1- 2 \beta } \ll \theta_{ \varepsilon} \ll \varepsilon^{ -1  } $. Then,   by expanding the square and by applying Young's inequality to the resulting product   we have  that
	\begin{align}
			 \dfrac{ 1 }{ \varepsilon^{2 \beta } }
		\int_{ \Omega' }
		\abs{ \nabla^{ 2 } z_{ \varepsilon } }^{ 2 }
		\dd{ x }
		&=
		\dfrac{ \varepsilon^{ 2 } }{ \varepsilon^{ 2 \beta } }
		\int_{ \Omega' }
		\abs{ \nabla^{ 2 } ( u_{ \varepsilon } + \phi ) }^{ 2 }
		\dd{ x } \\
		&\leq 
		 \dfrac{ 1 }{ \varepsilon^{ 2 \beta } }  
		\int_{ \Omega' }\abs{ \nabla^{ 2 } y_{ \varepsilon } }^{ 2 }
		\dd{ x }
{+}   \dfrac{ \theta_{ \varepsilon} \varepsilon }{  \varepsilon^{ 2 \beta } }   
		\int_{ \Omega' }\abs{ \nabla^{ 2 } y_{ \varepsilon } }^{ 2 }
		\dd{ x }		
		+
		\left( \dfrac{ \varepsilon^{ 2 } }{ \varepsilon^{ 2 \beta } } + 
		\dfrac{ \varepsilon }{ \theta_{ \varepsilon} \varepsilon^{ 2 \beta } } 
		\right)
		\int_{ \Omega' }
		\abs{ \nabla^{ 2} \phi }^{2 }
		\dd{ x }, \notag 	\label{eq:upper_semicontinuity_for_second_derivative}
	\end{align}
	 where we also used that $\eps |\nabla^2 u_\eps| \le |\nabla^2 y_\eps|$ by  the modifications   \eqref{def:y_rot}, \eqref{eqn:tildeuEps}, and  \eqref{def:u_eps}.  	By choice of $ \theta_{ \varepsilon } $ and using $ \beta < 1 $, the last summand vanishes in the limit. Moreover, since $ \varepsilon^{ - 2 \beta } \int_{ \Omega' } \abs{\nabla^{ 2 } y_{ \varepsilon } }^{ 2 } \dd{ x } $ stays uniformly bounded  by \Cref{lem:aprioriEst}  and $ \theta_{ \varepsilon} \varepsilon \to 0 $,  the second  summand  vanishes in the limit. Both can be absorbed into  $ \rho_{ \varepsilon }$   as in inequality (\ref{eq:estimate_rhoeps}) if we choose 
	\begin{equation*}
		\delta_{ \varepsilon }
		\geq
		{ \theta_{ \varepsilon} \varepsilon }
		{\left(\sup_{t\in [0,1]}\dfrac{ 1 }{  \varepsilon^{ 2 \beta } }\int_{ \Omega' }\abs{ \nabla^{ 2 } y_{ \varepsilon } (t) }^{ 2 }
		\dd{ x }\right)}
		+
		\left( \dfrac{ \varepsilon^{ 2 } }{ \varepsilon^{ 2 \beta } } + 
		\dfrac{ \varepsilon }{  \theta_{ \varepsilon} \varepsilon^{ 2 \beta } }
		\right).
	\end{equation*} 	
	Lastly, we take care of the jump set by estimating
	\begin{align}
 \hm^{ 1 } \big(
		\left(
		J_{ z_{ \varepsilon } } \cup J_{ \nabla z_{ \varepsilon } }\right)
		\setminus
		\Gamma_{ \varepsilon } ( t) 
		\big)	
		\leq{} &
	 		\hm^{ 1 } \big( \left( J_{ \phi } \cup J_{ \nabla \phi } \right) \setminus 
		\Gamma_{ \varepsilon } ( t )
		\big) + \hm^{ 1 } \big( 
		\left( J_{ u_{ \varepsilon } } \cup J_{ \nabla u_{ \varepsilon } } \right) 
		\setminus
		\Gamma_{ \varepsilon } ( t ) 
		\big).    	\label{eq:upper_semicontinuity_for_jumpset}
	\end{align}
 	The last  term can again be absorbed into $\rho_\eps$ since  modifications are barely increasing the jump set, see inequality (\ref{eq:modifications_barely_increase_jump}) {and definition \eqref{eqn:GammaEps}}.
	We collect inequalities (\ref{eq:approx_minimality_away_from_id})--(\ref{eq:upper_semicontinuity_for_jumpset}) to deduce     inequality  \eqref{finallytoshow}. This concludes the proof. 
\end{proof}

\begin{remark}
	\label{rmk:approximate_minimality_alter}
 Let $t \in I_\infty$. 	In the case that the competitor is of the form $ \psi \in \gsbvtwotwo ( \Omega' ; \mathbb{ R }^{ 2 } ) $ with $ \psi =   h_{ \eps }(t)   $ on $ \Omega' \setminus \overline{ \Omega } $  and $|\nabla \psi| \in \lp^3(\Omega')$,  we immediately deduce from the proof that we instead get the inequality
	\begin{equation*}
		\int_{ \Omega' }
		\dfrac{ 1 }{ 2 }
		Q \big( e ( u_{ \varepsilon } ( t ) ) \big) 
		\dd{ x }
		- 
		\rho_{ \varepsilon }
		\leq
		\int_{ \Omega' }
		\dfrac{ 1 }{ 2 }
		Q ( e ( \psi ) ) 
		\dd{ x }
		+
		\kappa \hm^{ 1 } \left( \left( J_{ \psi } \cup J_{ \nabla \psi } \right) 
		\setminus
		\Gamma_{ \varepsilon } ( t )
		\right),
	\end{equation*}
	where the sequence $ \rho_{ \varepsilon} $ can be bounded for some $  C > 0  $ and $  \delta_\eps \to 0 $ independently of $ t $ and $ \psi $ by
	\begin{equation}
		\label{eq:estimate_rhoeps_alt}
		 \rho_{ \varepsilon }
		\leq
		C  \delta_{ \varepsilon } 
		+ C \varepsilon^{ 2 \left( 1 - \beta \right) } \int_{ \Omega' } \abs{ \nabla^{ 2 } \psi }^{ 2 } \dd{ x } 
		 +
		C \int_{ \Omega' }
		\varepsilon \abs{ \nabla \psi }^{ 3 }
		\dd{ x } .
	\end{equation}
\end{remark}
 
Since we have an approximate minimization property, we also have an approximate variational inequality, which we will need later for the energy balance. 
\begin{corollary}
	\label{cor:approx_euler_lagrange}
	There exists a constant $ C > 0 $ such that for every $ t \in [0,1] $ and $ \phi \in \gsbvtwotwo ( \Omega' ; \R^2 ) $ with $ \phi = 0 $ on $ \Omega' \setminus \overline{\Omega} $,  $ \hm^{ 1 } ( J_{ \nabla \phi } \setminus J_{ \phi })= 0 $,   and $|\nabla \phi| \in \lp^3(\Omega')$,   we have that 
	\begin{equation*}
		\abs{
			\int_{ \Omega' }
		 \mathbb{ C } e ( u_{ \varepsilon }  (t)  ) \colon { e ( \phi ) }
			\dd{ x }
		}^{ 2}
	\leq
		C   \max\left\{ 1 , 
		\int_{ \Omega' }
		\dfrac{ 1 }{ 2 }
		Q ( e ( \phi ) )
		\dd{ x }  \right\} 
		\left(
		\kappa
		\hm^{ 1 } \left(J_{ \phi } 
		\setminus
		\Gamma_{ \varepsilon } ( t )
		\right)
		+
		\rho_{ \varepsilon }  (\phi) 
		\right) ,
	\end{equation*}
	 
where $({\rho_\eps(\phi)})_\eps$ denotes the sequence depending on $\phi$ given in (\ref{eq:estimate_rhoeps}) satisfying  $ {\rho_\eps(\phi)} \to 0 $. 
\end{corollary}

\begin{proof}
For the sake of brevity, we write $ H_\eps  \coloneqq \kappa
\hm^{ 1 } \left(J_{ \phi } 
\setminus
 \Gamma_\eps  ( t )
\right) $.
If 
\begin{equation}\label{eqn:ruleOut}
		\int_{\Omega'}\dfrac{1}{2} Q(e(\phi)) \dd{x} \leq 
		H_\eps
		+
		\rho_{ \varepsilon } ( \phi )
		,
		\end{equation} 
	the desired estimate follows from H\"older's inequality and the uniform energy bound (\ref{eqn:uTildeAPriori}),  where the constant  depends on $\mathbb{C}$  (see \eqref{eq: QQQ}).  

	Otherwise, for every $ 1 > \theta > 0 $, the function $ \theta \phi $ is an admissible competitor for \Cref{lemma:linearization_of_minimization_problem}. Thus,
	\begin{align*}
		0 & \leq
		\dfrac{ 1 }{ \theta }
		\left(
		\int_{ \Omega' }
		\dfrac{ 1 }{ 2 }
		Q ( e ( u_{ \varepsilon } (t)   + \theta \phi ) )
		\dd{ x }
		-
		\int_{ \Omega' }
		\dfrac{ 1 }{ 2 }
		Q ( e ( u_{ \varepsilon }  (t)  ) )
		\dd{ x }
		+
		H_\eps
		+ \rho_{ \varepsilon } ( \theta \phi )
		\right)
		\\
		& =
		\int_{ \Omega' }
		  \mathbb{ C } e ( u_{ \varepsilon }  (t)   ) \colon { e ( \phi ) }
		\dd{ x }
		+
		\theta
		\int_{ \Omega' }
		\dfrac{ 1 }{ 2 }
		Q ( e ( \phi ) )
		\dd{ x }
		+
		\dfrac{ 1 }{ \theta }
		\left(
		H_\eps
		+
		\rho_{ \varepsilon } ( \theta \phi )
		\right).
	\end{align*}
	We choose
	\begin{equation*}
		\theta_0 
		=
		\left(H_\eps
		+
		\rho_{ \varepsilon } ( \phi )
		\right)^{ 1/2 }
		\left(
		\int_{ \Omega' }
		\dfrac{ 1 }{ 2 }
		Q ( e ( \phi ) )
		\dd{ x }
		\right)^{ -1/2 }
	\end{equation*}
	and note that $ \theta_{ 0 } <1 $  since  we ruled out inequality (\ref{eqn:ruleOut}).
	Plugging this into the above inequality yields
	\begin{equation*}
		- \int_{ \Omega' }
		\mathbb{ C } e ( u_{ \varepsilon }  (t)  ) \colon e ( \phi )
		\dd{ x }
		\leq
		\left( H_\eps + \rho_{ \varepsilon} ( \phi ) \right)^{ 1/ 2 }
		\left( \int_{ \Omega' } \dfrac{ 1 }{ 2 } Q ( e ( \phi ) ) \dd{ x } \right)^{ 1/2 }
		+ 
		\dfrac{ H_\eps + \rho_{ \varepsilon } ( \theta_{ 0 } \phi ) }{ \left( H_\eps + \rho_{ \varepsilon } ( \phi ) \right)^{1/2}}
		\left( \int_{ \Omega' } \dfrac{ 1 }{ 2 } Q ( e ( \phi ) ) 
		\right)^{ 1/2}.
	\end{equation*}
	By observing inequality (\ref{eq:estimate_rhoeps}), we note that, since $ \theta_{ 0 } < 1 $, we have $ \rho_{ \varepsilon } ( \theta_{ 0 } \phi ) \leq \rho_{ \varepsilon } ( \phi )  $, and thus we obtain 
	\begin{equation*}
		- \int_{ \Omega' }
		\mathbb{ C } e ( u_{ \varepsilon }  (t)  ) \colon e ( \phi )
		\dd{ x }
		\leq
		2
		\left( H_\eps + \rho_{ \varepsilon} ( \phi ) \right)^{ 1/ 2 }
		\left( \int_{ \Omega' } \dfrac{ 1 }{ 2 } Q ( e ( \phi ) ) \dd{ x } \right)^{ 1/2 }.
	\end{equation*}
	Repeating the argument with $ - \phi $ then yields the desired inequality.
\end{proof}

\subsection{Transfer of minimality}

We now want to pass to the limit $ \varepsilon \to 0 $ in \Cref{lemma:linearization_of_minimization_problem} to argue that  the auxiliary limit  $\hat{u} (t) $ given by \Cref{def:displaceAtGoodTimes}    is a minimizer of the Griffith energy accounting for the existing crack, in analogy to \cite[Lem.~3.3]{francfort_larsen_existence_and_convergence_for_quasi_static_evolution_in_brittle_fracture}. The proof is similar, but since we are not in the antiplanar case, we require the jump transfer  lemma  for $ \gsbd $ given in \Cref{thm:jump_transfer_refined} (a refinement of \cite[Thm.~5.1]{friedrich_solombrino_quasistatic_crack_growth_in_2d_linearized_elasticity}).

\begin{lemma}
	\label{lemma:u_minimzer_of_total_energy}
	\label{rmk:u_hat_min_of_total_energy}
	For each $ t \in [0,1] $, the displacement $ \hat{u}( t ) $ from  Remark  \ref{def:displacement I} minimizes
	\begin{equation}
		\label{eq:u_min_of_total_energy}
		\int_{ \Omega' }
		\dfrac{ 1 }{ 2 }
		Q ( e ( v ) )
		\dd{ x }
		+ 
		\kappa 
		\hm^{ 1 } \left(
		J_{ v }
		\setminus
		\left( 
		\Gamma ( t )
		\cup 
		J_{ \hat{u} ( t ) }
		\right)
		\right)
	\end{equation}
	among all $ v \in \gsbdtwo ( \Omega' ) $ with $ v = h( t ) $ on $ \Omega' \setminus \overline{\Omega } $. 
\end{lemma}

We point out that, for $t\in I_\infty$, $J_{ \hat{u} ( t )}$ is redundant in \eqref{eq:u_min_of_total_energy} as $\hat{u}(t) = u(t).$ Naturally, this  lemma  will be critical for verifying the minimality property \ref{item:minimality} of \Cref{def:linearQuasistatic}, but further, the analysis at times $t\not \in I_\infty$ will help to correctly identify the displacement gradient $e(u)\in \lp^\infty ({(0,1)};\lp^2(\Omega; \R^{2 \times 2}_{\rm sym}) ).$ 

 \begin{proof} 
{Let $t\in [0,1].$} By the density result in \Cref{thm:density_with_boundary_values}  (applied for $v-\hat{u}(t)$),  it suffices to prove the minimiality of $\hat u(t)$ with respect to competitors $ \hat u(t)+ \phi $ for $ \phi \in \mathcal{ W } ( \Omega' ; \mathbb{ R }^{ 2 } ) $ with $ \phi = 0 $ on $ \Omega' \setminus \overline{\Omega } $.

	 {We begin by approximating the crack $\Gamma(t)$ with finitely many jump sets. Precisely, letting $\eta>0$ and recalling} that the total crack has finite length, see (\ref{eq:lsc_jump_sets_union_times}), we can find $ t_{ 1 }, \ldots, t_{ p } \in  I^t_{ \infty }  $ such that
	\begin{equation}
		\label{eq:approx_jump_set_finite_times}
		\hm^{ 1 } \left(
		\bigcup_{ k = 1 }^{ p } J_{ u( t_{ k } ) } 
		\cup J_{ \hat{u} ( t ) }
		\right)
		\geq
		\hm^{ 1 } \left(
		\Gamma ( t )
		\cup
		J_{ \hat{u} ( t ) }
		\right)
		- \eta.
	\end{equation}
	Choose $t_\eps \in I_\eps$ as the largest element of $I_\eps$ which is smaller than $t$. As $u_\eps$ and $y_\eps$ are piecewise constant in time,  we find $u_\eps(t) = u_\eps(t_\eps)$ and $y_\eps(t) = y_\eps(t_\eps)$. Then,  if $\varepsilon$ is sufficiently small so that $ (t_{ k })_{ k=1 }^p  \subseteq I_{ \varepsilon} $,  by  \Cref{lemma:linearization_of_minimization_problem} (applied at time $t_\eps \in I_\infty$)  we have for all $ \phi \in \gsbvtwotwo ( \Omega' ;  \mathbb{ R }^{ 2 }  ) $ with  $ \phi = 0 $ on $ \Omega' \setminus \overline{\Omega } $  and $|\nabla \phi| \in \lp^3(\Omega')$ that
	\begin{align}
		\label{eq:min_problem_total_mass_finitely_many_times}
		&\int_{ \Omega' }
		\dfrac{ 1 }{ 2 } Q \big( e(u_{ \varepsilon } ( t ) ) \big)
		\dd{ x }
		-
		\rho_{ \varepsilon }
		\\
		\notag
		\leq{} &
		\int_{ \Omega' }
		\dfrac{ 1 }{ 2 }
		Q \big( e( u_{ \varepsilon } ( t ) + \phi ) \big)
		\dd{ x }
		+ 
		\kappa \hm^{1 } \left(
		\left( 
		J_{ \phi } \cup J_{ \nabla \phi }
		\right)
		\setminus \left( 
		\bigcup_{ k = 1 }^{ p }
		\left(
		J_{ y_{ \varepsilon } ( t_{ k } ) } \cup J_{ \nabla y_{ \varepsilon } ( t_{ k } ) }
		\right)
		\cup
		\left(J_{ y_{  \varepsilon }(t) }
		\cup J_{ \nabla y_{  \varepsilon  }(t ) }
		\right)
		\right)
		\right),
	\end{align}
	 where $\rho_\eps  =\rho_\eps(\phi)$ satisfies (\ref{eq:estimate_rhoeps}).
	
	{Now, fix $ \phi \in  \mathcal{ W } ( \Omega' ; \mathbb{ R }^{ 2 } ) $ with $ \phi = 0 $} on $ \Omega' \setminus \overline{ \Omega } $. Since for every $ 1 \leq k \leq p $ the sequence $ v_{ \varepsilon } ( t_{ k } ) $ (defined in equation (\ref{eq:veps_def})) converges in measure to $ u( t_{ k } ) $,
	has uniformly bounded jump-set measure, and has uniformly $ \lp^{2 } $-bounded symmetric gradients, and the same holds true for $ v_{ \varepsilon } ( t ) $ with respect to $ \hat{u} ( t ) $, we can apply the $ \gsbd $ jump {transfer  lemma (see \Cref{thm:jump_transfer_refined})  to $\phi$ to find a} sequence $ ( \phi_{ \varepsilon }  )_{ \varepsilon }   \subseteq \gsbvtwotwo  ( \Omega' ; \mathbb{ R }^{ 2 }  )$ such that $ \phi_{ \varepsilon } = 0 $ on $ \Omega' \setminus \overline{\Omega } $  and 
	\begin{enumerate}[labelindent=0pt,labelwidth=\widthof{\ref{eq:jump_transfer_item_four}},label=(\roman*),ref=(\roman*),itemindent=1em,leftmargin=!]
		\item \label{eq:jump_transfer_item_one} $ \phi_{ \varepsilon } \to \phi $ in measure,
		\item \label{eq:jump_transfer_item_two} $ e( \phi_{ \varepsilon} ) \to e ( \phi ) $ strongly in $ \lp^{ 2 } ( \Omega' ; \mathbb{ R }^{ 2 \times 2}_{ \mathrm{sym} }  ) $,
		\item \label{eq:jump_transfer_item_three}
		$ \limsup_{ \varepsilon \to 0 }
		\hm^{ 1 } \left( 
		\left(
		\left( J_{ \phi_{ \varepsilon } } \cup J_{ \nabla \phi_{ \varepsilon } } \right) \setminus
		\left( 
		\bigcup_{ k= 1 }^{ p } 
		J_{ v_{ \varepsilon } ( t_{ k } ) }
		\cup 
		J_{ v_{ \varepsilon } ( t ) }
		\right)
		\right)
		\right)
		\leq
		\hm^{ 1 } \left( 
		\left(
		J_{ \phi } \setminus
		\left( 
		\bigcup_{ k = 1 }^{ p }
		J_{ u ( t_{ k } ) }
		\cup 
		J_{ \hat{u} ( t ) } 
		\right)
		\right)
		\right)
		$,
		\item \label{eq:jump_transfer_item_four}
		$ \norm{ \nabla \phi_{ \varepsilon } }_{ \lp^{ \infty } }
		\leq C  \norm{ \nabla \phi }_{ \lp^{ \infty } } $ and 
		$ \norm{ \nabla^{ 2 } \phi_{ \varepsilon } }_{ \lp^{ \infty } }
		\leq C  \norm{ \nabla^{ 2 } \phi }_{ \lp^{ \infty } } $.
	\end{enumerate}
	Since $ \phi_{ \varepsilon } $ is an admissible competitor within inequality (\ref{eq:min_problem_total_mass_finitely_many_times}), we find
	\begin{align}
		& \label{eqn:jumptransCompare} \int_{ \Omega' }
		\dfrac{ 1 }{2}
		Q \big( e ( u_{ \varepsilon } ( t ) ) \big)
		\dd{ x }
		- \rho_{ \varepsilon }
		\\
		\leq{} &
		\int_{ \Omega' }
		\dfrac{ 1 }{ 2 }
		Q \big( e ( u_{ \varepsilon } ( t ) + \phi_{ \varepsilon } ) \big)
		\dd{ x }
	    +
		\kappa \hm^{ 1 } \left(
		\left( 
		J_{  \phi_{ \varepsilon } } 
		\cup 
		J_{ \nabla \phi_{ \varepsilon} }
		\right)
		\setminus
		\left(
		\bigcup_{ k = 1 }^{ p }
		\left( J_{ y_{ \varepsilon }  ( t_{ k } )} \cup J_{ \nabla y_{ \varepsilon } ( t_{ k } ) }
		\right)
		\cup 
		\left( J_{ y_{ \varepsilon } ( t) } 
		\cup
		J_{ \nabla y_{ \varepsilon } ( t ) }
		\right)
		\right)
		\right),\nonumber
	\end{align}
	where the error $ \rho_{ \varepsilon } = \rho_\eps(\phi_\eps) $ is estimated by inequality (\ref{eq:estimate_rhoeps}) with a vanishing sequence $ \delta_{ \varepsilon} \to 0 $ and  a  constant $C> 0 $   as
	\begin{equation*}
		\rho_{ \varepsilon } \leq C \delta_{ \varepsilon} 
		\left( 
		1 + 
		\int_{ \Omega' }
		\abs{ \nabla {\phi_\eps} }^{ 2 }
		\dd{ x }
		+
		\int_{ \Omega' }
		\abs{ \nabla^{ 2 } {\phi_\eps} }^{ 2 }
		\dd{ x }
		\right)
		+
		C
		\int_{ \Omega' }
		\varepsilon \abs{ \nabla {\phi_\eps} }^{ 3 }
		\dd{ x }.
	\end{equation*}
	By the uniform bound on  $ \norm{   \nabla {\phi_\eps}  }_{\lp^{ \infty } } $ and $\norm{   \nabla^{ 2 } {\phi_\eps}  }_{\lp^{ \infty } } $ given in item \ref{eq:jump_transfer_item_four}, we see that $ \rho_{ \varepsilon} $ vanishes as $ \varepsilon $ tends to zero. 
	By subtracting $ \int_{\Omega'} \tfrac12 Q ( e ( u_{ \varepsilon } ( t ) ) )  \dd{x}$ on both sides of  inequality  \eqref{eqn:jumptransCompare}, we get
	\begin{align}
		\notag
		-\rho_{ \varepsilon }
		\leq{} &
		\int_{ \Omega' }
		 \mathbb{ C }  e ( u_{ \varepsilon } ( t ) ) \colon {e ( \phi_{ \varepsilon } ) }
		\dd{ x }
		+
		\int_{\Omega' }
		\dfrac{ 1 }{ 2 }
		Q ( e ( \phi_{ \varepsilon } ) ) \dd{ x }
		 +
		\kappa \hm^{ 1 } \left(
		\left( 
		J_{ \phi_{ \varepsilon } } 
		\cup
		J_{ \nabla \phi_{ \varepsilon} }
		\right)
		\setminus
		\left(
		\bigcup_{ k = 1 }^{ p }
		J_{ v_{ \varepsilon } ( t_{ k } ) }
		\cup 
		J_{ v_{ \varepsilon } ( t ) }
		\right)
		\right)
		\\
		\label{eq:rho_eps_estimate}
		& +
		\sum_{ k = 1 }^{ p }
		\kappa \hm^{ 1 } \big(
		J_{ v_{ \varepsilon } ( t_{ k } ) }
		\setminus
		\left( J_{ y_{ \varepsilon } ( t_{ k } ) } 
		\cup
		J_{ \nabla y_{ \varepsilon } (t_{ k } ) }
		\right)
		\big)
		+
		\kappa\hm^{ 1 } \big(
		J_{ v_{ \varepsilon } ( t ) }
		\setminus
		\left(
		J_{ y_{ \varepsilon } ( t ) }\cup J_{ \nabla y_{ \varepsilon } ( t ) } 
		\right)
		\big).
	\end{align}
	The last line converges to zero as $\varepsilon \to 0 $ using that the modifications barely increase the jump set as determined by equation (\ref{eq:jump_of_veps_controlled_by_yeps}).
	Now passing $\eps\to 0$ in inequality (\ref{eq:rho_eps_estimate}), we can pair the weak convergence of $ e ( u_{ \varepsilon } ) $ with the strong convergence of $ e ( \phi_{ \varepsilon } ) $  (see item \ref{eq:jump_transfer_item_two})  and use  the jump set estimate \ref{eq:jump_transfer_item_three} to obtain
	\begin{equation*}
		0 \leq{}
		\int_{ \Omega' }
		 \mathbb{ C }  e ( \hat{u} ( t ) )   \colon  {  e ( \phi ) } 
		\dd{ x }
		+
		\int_{ \Omega' }
		\dfrac{ 1 }{ 2 } Q ( e ( \phi ) )
		\dd{ x }
		+
		\kappa \hm^{ 1 } \left(
		J_{ \phi }
		\setminus
		\left(
		\bigcup_{ k = 1 }^{ p }
		J_{ u ( t_{ k } )} \cup J_{ \hat{u} ( t ) }
		\right)
		\right).
	\end{equation*}
	We add $ \int_{\Omega'} \tfrac12 Q ( e ( \hat{u} ( t) ) ) \dd{x}$ on both sides of the inequality and note that $J_{\phi}\setminus J_{\hat u(t)} = J_{\hat u(t) + \phi}\setminus J_{\hat u(t)}$, which yields the desired inequality \eqref{eq:u_min_of_total_energy} for $v = \hat u(t)+\phi$ up to an arbitrary $\eta>0$ coming from the approximation \eqref{eq:approx_jump_set_finite_times}. Sending $\eta\to 0$ concludes the minimality property.
\end{proof}

\subsection{Elastic   energy convergence}  
As a consequence of \Cref{lemma:u_minimzer_of_total_energy}, we have established that $ \hat u (t)  $ is a minimizer with respect to its own jump set in the linearized functional, {i.e.,
one may  replace $\Gamma(t)\cup J_{\hat u(t)}$ by $J_{\hat u(t)}$ in the minimization problem  \eqref{eq:u_min_of_total_energy}.} In a similar fashion,  the functions $ y_{ \varepsilon } $ are minimizers with respect to their own jump set of the nonlinear functional,  {cf.} \eqref{eq:min_problem_wrt_own_jumpset}. This minimality property now induces convergence of  the elastic energies. Whereas for Griffith energies this simply follows from \cite{friedrich_griffith_energies_as_small_strain_limit_of_nonlinear_models_for_nomsimple_brittle_materials} along with the fundamental theorem of $ \Gamma $-convergence, in the present evolutionary setting with irreversibility condition, this is a consequence of the   jump transfer  lemma (see  {\Cref{thm:jump_transfer_refined}}). 

\begin{corollary}[Elastic   energy convergence]
	\label{cor:convergence of energies}
	For all $ t \in [0,1] $, we have the convergence of  elastic  energies
	\begin{equation}
		\label{eq:energy_convergence}
		\lim_{ \varepsilon \to 0 }
		\dfrac{ 1 }{ \varepsilon^{ 2 } }
		\int_{ \Omega' }
		W  ( \nabla y_{ \varepsilon } ( t ) ) 
		\dd{ x }
		+
		\dfrac{ 1 }{ \varepsilon^{ 2 \beta } }
		\int_{ \Omega' }
		\abs{ \nabla^{ 2 } y_{ \varepsilon } ( t ) }^{ 2 }
		\dd{ x }
		=
		\int_{ \Omega' }
		\dfrac{ 1 }{ 2 }
		Q \big( e ( \hat{u} ( t ) ) \big) 
		\dd{ x },
	\end{equation}
	 again along a $t$-dependent subsequence for $t \notin I_\infty$.  	In particular, we have
	\begin{equation}
		\label{eq:second_derivative_vanishes}
		\lim_{ \varepsilon \to 0 }
		\dfrac{ 1 }{ \varepsilon^{ 2 \beta } }
		 \int_{ \Omega' }
		 \abs{ \nabla^{ 2 } y_{ \varepsilon } ( t ) }^{ 2 }
		 \dd{ x }
		 = 0.
	\end{equation}
\end{corollary}
\begin{proof}
 First, from  lower semicontinuity, see  inequality  (\ref{eq:lsc_for_quadratic_term}), we obtain the estimate
	\begin{align}\label{eq: second removed}
	\liminf_{ \varepsilon \to 0 }
		\dfrac{ 1 }{ \varepsilon^{ 2 } }
		\int_{ \Omega' }
		W \left( \nabla y_{ \varepsilon } ( t ) \right) 
		\dd{ x }
		+
		\dfrac{ 1 }{ \varepsilon^{ 2 \beta } }
		\int_{ \Omega' }
		\abs{ \nabla^{ 2 } y_{ \varepsilon } ( t ) }^{ 2 }
		\dd{ x } & \ge 	\liminf_{ \varepsilon \to 0 }
		\dfrac{ 1 }{ \varepsilon^{ 2 } }
		\int_{ \Omega' }
		W \left( \nabla y_{ \varepsilon } ( t ) \right) 
		\dd{ x }
		 		\notag \\ & \geq
		\int_{ \Omega' }
		\dfrac{ 1 }{ 2 }
		Q \big( e ( \hat{u} ( t ) ) \big) 
		\dd{ x }.
	\end{align}
	 (Note that the proof was only formulated for  $ t \in I_{ \infty } $, but still applies to   $ t \notin I_{\infty } $.)
	Since $ y_{ \varepsilon } ( t ) $ is a minimizer of the functional (\ref{eq:min_problem_wrt_own_jumpset}),  for the reverse limit superior inequality it suffices to find a recovery sequence $ z_{ \varepsilon } \in \gsbvtwotwo ( \Omega' ; \mathbb{ R }^{ 2 } ) $ with $ z_{ \varepsilon } = \mathrm{id} + \varepsilon h_{ \varepsilon } ( t ) $ on $ \Omega' \setminus \overline{ \Omega } $ such that  
	\begin{equation}\label{newandshort}
		\limsup_{ \varepsilon \to 0 }
		\dfrac{ 1 }{ \varepsilon^{ 2 } }
		\int_{ \Omega' }
		W \left( \nabla z_{ \varepsilon } \right) 
		\dd{ x }
		+ 
		\dfrac{ 1 }{ \varepsilon^{ 2 \beta } }
		\int_{ \Omega' }
		\abs{ \nabla^{2 } z_{ \varepsilon } }^{2 }
		\dd{ x }
		+
		\kappa \hm^{ 1 } \left( 
		\left( J_{ z_{ \varepsilon } } \cup J_{ \nabla z_{ \varepsilon } } \right)
		\setminus
		\Gamma_{ \varepsilon } ( t ) 
		\right)
		\leq
		\int_{ \Omega' }
		\dfrac{ 1 }{ 2 }
		Q \left( e \left( \hat{u} ( t ) \right) \right)
		\dd{ x }.
	\end{equation}
	To find this sequence, we approximate $ \hat{u} ( t ) $ and apply the jump transfer  lemma ({\Cref{thm:jump_transfer_refined}}) to the approximations. Precisely, by \Cref{thm:density_with_boundary_values} we find a sequence $  ( w_{n })_{ n }  \subseteq  \mathcal{ W }  ( \Omega' ; \mathbb{ R }^{ 2 }  ) $ with $ w_{ n } = h ( t ) $ on $ \Omega' \setminus \overline{ \Omega } $ such that, as $ n \to \infty $, we have
	\begin{enumerate}[labelindent=0pt,labelwidth=\widthof{\ref{item:jump_set_approx}},label=(\arabic*),ref=(\arabic*),itemindent=0em,leftmargin=!]
	\item $ w_{ n } \to \hat{u} ( t ) $ in measure, 
	\item \label{item:l2_approx} $ e ( w_{ n } ) \to e ( \hat{u} ( t ) ) $ in $ \lp^{ 2 }  ( \Omega' ; \mathbb{ R }^{ 2 \times 2 }_{\rm sym}) $ and
	\item \label{item:jump_set_approx} $ \hm^{ 1 } \left( J_{ w_{ n } } \triangle J_{ \hat{u} ( t ) } \right) \to 0 $.
	\end{enumerate}
 Observe that   the sequence $ v_{ \varepsilon } ( t ) $ defined in equation (\ref{eq:veps_def}) converges in measure to $ \hat{u}( t  ) $,
	has uniformly bounded jump-set measure, and has uniformly $ \lp^{2 } $-bounded symmetric gradients.  For all $ n \in \mathbb{ N } $, we apply the jump transfer  lemma (\Cref{thm:jump_transfer_refined}) to the function $ \phi = w_{ n } $ which yields a sequence $ ( w_{ \varepsilon }^{ n } )_{ \varepsilon}   \subseteq  \gsbvtwotwo  ( \Omega' ; \mathbb{ R }^{ 2 }  )$ with $ w_{ \varepsilon }^{ n } = h ( t ) $ on $ \Omega' \setminus \overline{ \Omega } $ such that
	\begin{enumerate}[labelindent=0pt,labelwidth=\widthof{\ref{item:boundedness_wepsn}},label=(\arabic*'),ref=(\arabic*'),itemindent=0em,leftmargin=!]
		\item $ e \left( w_{ \varepsilon}^{ n } \right) \to e\left( w_{ n } \right) $ in $ \lp^{ 2 }  ( \Omega' ; \mathbb{ R }^{ 2 \times 2 }_{\rm sym}) $,
		\item 
		\label{item:jump_set_comparison}$  \limsup_{\eps \to 0}  \hm^{ 1 }
		\left( 
		\left( 
		J_{ w_{ \varepsilon }^{ n } } \cup J_{ \nabla w_{ \varepsilon }^{ n } } 
		\right)
		\setminus
		 J_{ v_{ \varepsilon } ( t ) } 
		\right)
		\le 
		\hm^{ 1 } \left(  
		J_{ w_{ n } } \setminus J_{ \hat{u} ( t ) }
				\right) $,
		\item \label{item:boundedness_wepsn}$ \norm{ \nabla w_{ \varepsilon }^{ n } }_{ \lp^{ \infty } }
		\lesssim \norm{ \nabla w_{ n } }_{ \lp^{ \infty } } $ and 
		$ \norm{ \nabla^{ 2 } w_{ \varepsilon }^{ n } }_{ \lp^{ \infty } }
		\lesssim \norm{ \nabla^{ 2 } w_{ n } }_{ \lp^{ \infty } } $.
	\end{enumerate}
	We define $ z_{ \varepsilon }^{ n } \coloneqq \mathrm{id} + \varepsilon  ( w_{ \varepsilon }^{ n } + h_{ \varepsilon } ( t ) - h ( t )  ) $. Then $ z_{ \varepsilon }^{ n } = \mathrm{id} + \varepsilon h_{ \varepsilon } ( t ) $ on $ \Omega' \setminus \overline{\Omega } $. Moreover, using the boundedness of $ w_{ \varepsilon }^{ n } $ and $ h $ given by \ref{item:boundedness_wepsn} {and (\ref{eqn:bdryData}), respectively,}   the comparison of the jump sets \ref{item:jump_set_comparison}  along with
	  \eqref{eq:jump_of_veps_controlled_by_yeps},   $\beta < 1$, and $h_\eps(t) \to h(t)$ in $ \wkp^{ 1,2 } (\Omega'; \mathbb{ R }^{ 2 } )$,     we obtain through a Taylor expansion
	\begin{align}
		\notag
		&\limsup_{ \varepsilon \to 0 }
		\dfrac{ 1 }{ \varepsilon^{ 2 } }
		\int_{ \Omega' }
		W \left( \nabla z_{ \varepsilon }^{ n } \right)
		\dd{ x }
 		+
		\dfrac{ 1 }{ \varepsilon^{ 2 \beta } }
		\int_{ \Omega' }
		\abs{ \nabla^{ 2 } z_{ \varepsilon }^{ n } }^{ 2 }
		\dd{ x }
		+
		\kappa
		\hm^{ 1 } \left(
		\left(
		J_{ z_{ \varepsilon }^{ n } } \cup J_{ \nabla z_{ \varepsilon }^{ n } }
		\right)
		\setminus
		\Gamma_{ \varepsilon } ( t ) 
		\right)
		\\ 
		\label{eq:limsup_in_eps} 
		\leq{} &
		\int_{ \Omega' }
		\dfrac{ 1 }{ 2 }
		Q \left( e  ( w_{ n }  ) \right)
		\dd{ x }
		+
		 \kappa 
		\hm^{ 1 } \left( J_{ w_{ n } } \setminus J_{ \hat{u}  (t)  } \right).
	\end{align}
By the approximation of $ \hat{u}  (t) $ through $ w_{ n } $ (see items \ref{item:l2_approx} and \ref{item:jump_set_approx}), the limit superior in $ n $ of the right-hand side of inequality (\ref{eq:limsup_in_eps}) can be estimated by
	\begin{equation}
		\label{eq:limsup_in_n}
		\limsup_{ n \to \infty }
		\int_{ \Omega' }
		\dfrac{ 1 }{ 2 }
		Q \left( e \left( w_{ n } \right) \right)
		\dd{ x }
		+
		\kappa \hm^{ 1 } \big( J_{ w_{ n } } \setminus J_{ \hat{u}  (t)  } \big)
		\leq
		\int_{ \Omega' }
		\dfrac{ 1 }{ 2 }
		Q \big( e ( \hat{u} ( t ) )\big) \dd{ x}.
	\end{equation}
	 By a diagonal argument and by combining inequalities (\ref{eq:limsup_in_eps}) and {(\ref{eq:limsup_in_n})}  we get  the existence of a sequence $ n( \varepsilon ) \to \infty $ such that $z_{ \varepsilon }^{ n( \varepsilon )}$ satisfies \eqref{newandshort}. 
 This  finishes the proof of the energy convergence (\ref{eq:energy_convergence}).

 To show the convergence (\ref{eq:second_derivative_vanishes}), we simply note that  by   the elastic energy convergence   \eqref{eq:energy_convergence} and  inequality \eqref{eq: second removed}  the contribution of the second-derivative integral must  vanish.
\end{proof}

\begin{corollary}
	\label{cor:strong_l2_convergence_of_uepsaux}
	For each $ t \in [0,1] $, we have that  $ e  ( u_{ \varepsilon } ( t )  )$ and $ e  ( u_{ \varepsilon }^{ \mathrm{aux} } ( t )  )$ converge  to $e (\hat{u} ( t )  ) $ in $ \lp^{ 2 }  ( \Omega' ; \mathbb{ R }^{ 2 \times 2 }_{\rm sym} ) $  (up to a $t$-dependent subsequence in $\eps$ for $t\not\in I_\infty$).
\end{corollary}

\begin{proof} Using the convergence of elastic energies in \Cref{cor:convergence of energies}, one can argue precisely as in the proof of \cite[Thm.~3.11]{almi_davoli_friedrich_non_interpenetration_conditions_in_the_passage_from_nonlinear_to_linearized_griffith_fracture} to conclude  that $ e  ( u_{ \varepsilon }^{ \mathrm{aux} } ( t )  ) \to e (\hat{u} ( t )  ) $ in $ \lp^{ 2 }  ( \Omega' ; \mathbb{ R }^{ 2 \times 2 }_{\rm sym}) $.  As $\chi_{\Omega'\setminus \omega_\eps}\to 0$ by  convergence \ref{item:properties_of_etaeps_volume}, it also follows that $e(u_\eps(t))\to e(\hat u (t))$ in $\lp^2(\Omega;\R^{2\times 2}_{\rm sym})$. 
\end{proof}

\section{Approximate energy balance} \label{sec:approximateRelationsBalance}

We begin the preliminary steps of the proof of the energy-balance equation of  \ref{item:energy_balance_def} in \Cref{def:linearQuasistatic}.  To this end, recall the definition of the total energy in \eqref{totalenergy}. To derive the energy-balance equation, we make rigorous the idea that once a piece of material is disconnected from the  Dirichlet  boundary, it will relax to an undeformed equilibrium configuration. Precisely, from the approximate variational inequality  in  \Cref{cor:approx_euler_lagrange} we will deduce that the partition elements $ P_{ j}^{ \varepsilon } $ which do not intersect the  Dirichlet  boundary do not contribute to the energy-balance equation in the limit, in the sense of the following  lemma.

\begin{lemma}
	\label{cor:approx_euler_lagrange_consequence}
	{For $t\in [0,1]$,} let $(P_j^\eps {(t)} )_j$ be the Caccioppoli partition from  equation  (\ref{def:y_rot}) and let $(\hat R_{j}^\eps {( t)} )_j  \subseteq \sporth( 2) $ be any {associated collection} of matrices that is   piecewise constant in time  with respect to the same partition of $[0,1]$.  For all $ t \in [0, 1 ] $, we have
	\begin{equation}
		\label{eq:pieces_inside_dont_contribute}
		\lim_{ \varepsilon \to 0 }
		\abs{
			\int_{ 0 }^{ t }
			\sum_{ P_j^\eps(s) \in \intpcs_\eps(s)  }
			\int_{ P_{ j }^{ \varepsilon } ( s ) }
			  \mathbb{ C } e ( u_{ \varepsilon } (s) )  \colon { e ( \hat R_{ j }^{ \varepsilon }  (s)   \partial_t h (s) ) }
			\dd{ x }
			\dd{ s } 
		}
		=
		0,
	\end{equation}
	where $\intpcs_\eps(s) := \lbrace P_j^\eps(s) \colon  \mathcal{L}^2( P_{ j }^{ \varepsilon } ( s ) \cap ( \Omega' \setminus \overline{\Omega} )) = 0\rbrace $ are the interior partition pieces for all $s\in [0,1]$.  
\end{lemma}

As will be seen in the proof, the only constraint on the test function ($\sum_{ P_j^\eps(s) \in \intpcs_\eps(s)} \chi_{P_j^\eps(s)} \hat R_j^\eps(s) \partial_t h(s) $ above) is  that it is piecewise regular with respect to the partition and $0$ on the boundary. 
{Formally, this result shows that for the limit Caccioppoli partition $\mathcal{P}$ with $P_j\in \intpcs$, defined as above, and a regular test function $\phi$, one has $ \int_{P_j}\C{e(u)}:{e(\phi)}\, dx  = 0$,  which    implies  that $u$ is  an infinitesimal rigid motion  on $P_j$, meaning  the material has fully relaxed  to an equilibrium configuration  on $P_j$.   }

\begin{proof}
	We know by   \eqref{eqn:bdryData} that $  h \in \wkp^{ 1, 1 }  ( (0 , 1) ;  \wkp^{ 2 , 2 } ( \Omega' ; \mathbb{ R }^{ 2 })  ) $. Since $ e ( u_{ \varepsilon } ) $ is uniformly $ \lp^{ 2 }$-bounded (see estimate (\ref{eqn:uTildeAPriori})  and \eqref{def:u_eps}), via an approximation   it  thus suffices  to prove the claim under the assumption $  h \in \cont^{ \infty }  ( [0,1] ;  \wkp^{ 2 , 2 } ( \Omega' ; \mathbb{ R }^{ 2 }) ) $ {(note  that  this is just an assumption on the regularity of $h$ explicitly appearing in \eqref{eq:pieces_inside_dont_contribute}, not the boundary data)}.
	
	{Let $  s  \in [0, 1 ] $.} We apply the approximate variational inequality   in  \Cref{cor:approx_euler_lagrange} to the function 
	\begin{equation*}
		\phi_\eps   (s) 
		= 
		\sum_{  P_j^\eps(s) \in \intpcs_\eps(s) } 
		\chi_{ P_{ j }^{ \varepsilon } ( s ) }
		\hat R_{ j }^{ \varepsilon } ( s ) 
		\partial_t h (s)
	\end{equation*}
	to get that the term on the left-hand side of equality (\ref{eq:pieces_inside_dont_contribute}) can be estimated by the limit of
	\begin{equation*}
	 C 	\int_{ 0 }^{ t }
		 \left( \max  \left\{ 1, 
		\int_{ \Omega' }
		\dfrac{ 1 }{ 2 }
		Q ( e ( \phi_\eps ) ) 
		\dd{ x } \right\} 
				\left(
		\kappa
		\hm^{ 1 } \left(J_{ \phi_\eps } 
		\setminus
		\Gamma_{ \varepsilon } ( t )
		\right)
		+
		\rho_{ \varepsilon }  (\phi_\eps) 
		\right)
		\right)^{ 1/2 }
		\dd{ s }.
	\end{equation*}
	Since $  h \in \cont^{ \infty }  ( [0,1] ;  \wkp^{ 2 , 2 } ( \Omega' ; \mathbb{ R }^{ 2 }) ) $, the term $ \int_{\Omega'} Q ( e ( \phi_\eps ) ) \dd{ x } $ stays uniformly bounded in time.
	We have $ J_{ \phi_\eps  (s)  } \subseteq \bigcup_{ j } \partial^{ \ast } P_{ j }^{ \varepsilon } ( s ) $ and thus by construction of the Caccioppoli partition for the rotations, see \ref{item:rotation_property_jump}, we have 
	\begin{equation*}
		\hm^{ 1 } \left(
		J_{ \phi_\eps   (s)  } \setminus 
		\Gamma_{ \varepsilon } ( s )
		\right)
		\lesssim 
		\varepsilon^{ \beta - \gamma}
	\end{equation*}
	uniformly in time.  Thus, by recalling the estimate for $ \rho_{ \varepsilon } $  in  (\ref{eq:estimate_rhoeps}), it suffices to show that
	\begin{equation}
		\label{eq:convergence_of_error_term}
		\int_{ 0 }^{ t }
		\Big( \delta_{ \varepsilon }\left( 1 + 
		\int_{ \Omega' } \abs{ \nabla \partial_{ t } h }^{ 2 } \dd{ x }
		+
		\int_{ \Omega' } \abs{ \nabla^{ 2 }\partial_t  h }^{ 2 } \dd{ x }
		\right)
		  +
	 \int_{ \Omega' } \eps \abs{ \nabla \partial_{ t } h }^{ 3 } \dd{ x }  \Big)^{ 1/2 }
		\dd{ s }
		 \to 0 
	\end{equation}
	as $ \varepsilon $ tends to zero.
	By the Sobolev embedding, we know that $ \nabla \partial_{ t } h \in \lp^{ \infty }((0,1);  \lp^{ 3 }  (\Omega';\R^{2\times 2})   ) $, where we again used $ h \in \cont^{ \infty }  ([0,1]; \wkp^{ 2, 2 }   (\Omega';\R^2)  ) $. Thus,  equation (\ref{eq:convergence_of_error_term}) holds true, which finishes the proof. 
\end{proof}
As the next step for the desired energy balance, we derive inequalities describing the limiting energy evolution.
This  lemma,  in formulation and proof, closely follows \cite[Lem.~3.7]{francfort_larsen_existence_and_convergence_for_quasi_static_evolution_in_brittle_fracture}.  Recall the {definitions in \eqref{totalenergy} and} \eqref{totalepsenergy}. 
\begin{lemma}
	\label{lemma:energy_balance_inequalities}
	Assume that the initial conditions are well-prepared in the sense that 
	\begin{equation}
		\label{eq:well_preparedness_initial_condition}
		\limsup_{ \varepsilon \to 0 }
		\energyNonlin ( y_{ 0 }^{ \varepsilon } )
		\leq
		\energyLinTot ( 0 ).
	\end{equation}
	Then, for any $ t \in I_{ \infty } $, we respectively have the lower and upper estimates
	\begin{align}
	\label{eq:energy_balance_lsc}
		\energyLinTot ( t )
		&   \le \liminf_{\eps \to 0}\energyNonlinTot(t) 	
		 \leq
		\energyLinTot ( 0 )
		+
		\liminf_{ \varepsilon \to 0 }
		\int_{ 0 }^{ t }
		\int_{  { \Omega '} }
		 \mathbb{ C } e( u_{ \varepsilon } ( s ) ) \colon { e( \partial_t h ( s ) ) }
		\dd{ x }
		\dd{ s }
		\shortintertext{and}
		\label{eq:energy_balance_usc}
		\energyLinTot ( t )
		& \geq
		\energyLinTot ( 0 )
		+
		\limsup_{ \varepsilon \to 0 }
		\sum_{ k = 0 }^{ N( \varepsilon ) - 1 }
		\int_{ t_{ k }^{ \varepsilon } }^{ t_{ k + 1 }^{ \varepsilon } }
		\int_{  {\Omega' } }
		 \mathbb{ C }  e( u( t_{ k + 1 }^{ \varepsilon } ) ) \colon { e ( \partial_t  h ( s ) ) } 
		\dd{ x }
		\dd{ s },
	\end{align}
	where $ N ( \varepsilon ) $ is chosen such that $ t = t_{ N ( \varepsilon ) }^{ \varepsilon } $. Furthermore, $ \energyLinTot $ has no negative jumps.
\end{lemma}

Existence of well-prepared initial conditions follows from  \Cref{rem: well}.

\begin{proof}
	While the $ \limsup $-inequality is proven exactly as in \cite[Lem.~3.7]{francfort_larsen_existence_and_convergence_for_quasi_static_evolution_in_brittle_fracture}  (exploiting \Cref{lemma:u_minimzer_of_total_energy}),  we have to modify the proof for the $ \liminf $-inequality. In \cite[Lem.~3.7]{francfort_larsen_existence_and_convergence_for_quasi_static_evolution_in_brittle_fracture}, it is an immediate consequence of the computation performed for the a priori estimate.  Instead,  due to the rotational invariance of $W$, we have to work more: by  the frame indifference  \ref{eq:frami},  the derivative $ \diff W $ satisfies
	\begin{equation}
		\label{eq:rotation_of_nabla_W}
		{\diff W ( A ) = R^{ \top } \diff  W ( R A )}
	\end{equation}
	for any rotation $ R \in \sporth ( 2 ) $.  Thus, even though $ W( \nabla y_{\varepsilon } ) = W ( \nabla y_{\varepsilon }^{ \mathrm{rot} } ) $ holds, {we have}
	\begin{equation*}
		{\diff W ( \nabla y_{ \varepsilon } )
		\neq 
		\diff W ( \nabla y_{ \varepsilon }^{ \mathrm{rot} })}
	\end{equation*}
	in general.  For this reason, we need to  throw the rotations onto $ \nabla \partial_t h $ and argue that the pieces where the rotation is not the identity already have negligible contribution   due to  \Cref{cor:approx_euler_lagrange_consequence}. 
	
	\emph{Step 1 (Lower estimate).} Let us first prove the limit inferior estimate (\ref{eq:energy_balance_lsc}). As mentioned before, we proceed similarly to the a priori estimate. Let $ t \in I_{ \infty } $ and $ \varepsilon $ be sufficiently small such that $ t \in I_{ \varepsilon } $.
	From the a priori estimate (\ref{eq:a_priori_inequality_error_reduced}), we see that there exists some  vanishing sequence $ \sigma_{ \varepsilon} $    such that 
	\begin{align}
		\notag
		&
		\dfrac{ 1 }{ \varepsilon^{ 2 } }
		\int_{ \Omega'}
		W ( \nabla y_{ \varepsilon }(t) ) 
		\dd{ x }
		+
		\dfrac{ 1 }{ \varepsilon^{ 2 \beta } }
		\int_{ \Omega' }
		\abs{ \nabla^{ 2 } y_{ \varepsilon }(t) }^{ 2 }
		\dd{ x }
		+
		\kappa \hm^{ 1 } \left(
		\Gamma_{ \varepsilon } ( t )
		\right)
		-\sigma_{ \varepsilon }
		\\
		\leq{} &
		\energyNonlin ( y_{ 0 }^{ \varepsilon } )
		\label{eq:integral_over_time}
		+
		\dfrac{ 1 }{ \varepsilon }
		\int_{ 0 }^{  t  }
		\int_{ \lbrace \nabla y_{ \varepsilon } \in B_{ r } ( \sporth( 2 ) ) \rbrace }
		{\langle \diff W ( \nabla y_{ \varepsilon } ) ,\nabla  \partial_t h \rangle}
		\dd{ x }
		\dd{ s },	
	\end{align}
 	where $r$ is as in  item \ref{eq:W_growth_assumptions}. 
To handle the integral over time in (\ref{eq:integral_over_time}), we again apply Taylor's formula. For this, we  recall that   $ \diff W ( \mathrm{Id} ) = 0 $  and we note that  
	\begin{equation}\label{useaswell}
		{\mathcal{N}_\eps:= \{ \nabla y_{ \varepsilon } \in B_{ r } ( \sporth ( 2 ) ) \}
		\subseteq 
		\{ \varepsilon \abs{ \nabla u^{ \mathrm{aux} }_{ \varepsilon } } \leq 2 r \}}
	\end{equation}
	for $ \varepsilon $ sufficiently small by construction of the rotations (see \ref{item:rotation_property_closeness}), and the definition (\ref{eqn:tildeuEps}) of $ u^{ \mathrm{aux}}_{ \varepsilon } $ {(we suppress the dependence on $t$ for $\mathcal{N}_\eps$)}. Thus, by using the rotational  property (\ref{eq:rotation_of_nabla_W}) and  equations  \eqref{def:y_rot}--\eqref{eqn:tildeuEps},  on $ {\mathcal{N}_\eps} $ we get that 
	\begin{align}
		\notag \langle \diff W \left( \nabla y_{ \varepsilon } \right) , \nabla \partial_t h \rangle
		& =
		\left\langle \diff W \left( \nabla y_{ \varepsilon}^{ \mathrm{rot} }  \right), \sum\nolimits_{ j } \chi_{ P_{ j }^{ \varepsilon } } R_{ j }^{ \varepsilon } \nabla \partial_t h \right\rangle
		\\
		\notag & =
		\left\langle \diff W \left( \mathrm{Id} + \varepsilon \nabla u^{ \mathrm{aux}}_{ \varepsilon } \right), \sum\nolimits_{ j } \chi_{ P_{ j }^{ \varepsilon } } R_{ j }^{ \varepsilon } \nabla  \partial_t  h \right\rangle
		\\
		& \leq
		\varepsilon 
		\left( 
		\mathbb{ C }  e ( u^{ \mathrm{aux}}_{ \varepsilon }) \colon \sum\nolimits_{ j } \chi_{ P_{ j }^{ \varepsilon } } e \left( R_{ j }^{ \varepsilon } \partial_t h \right)
		\right)
		+
		C  \varepsilon^{ 2 }  \abs{  \nabla u^{ \mathrm{aux}}_{ \varepsilon } }^{ 2 } \abs{ \nabla \partial_t h }.\label{eqn:timeIntAdjust}
	\end{align}
	Since $ d = 2 $, the space $ \wkp^{ 2 , 2 } ( \Omega' ) $ embeds continuously into $ \wkp^{ 1 , p } ( \Omega' ) $ for every $ p \in [1 , \infty) $, so by the regularity of the boundary data \eqref{eqn:bdryData}, we have $ \nabla \partial_t  h \in \lp^{ 1 } ( {(0,1)}; \lp^{ p } ( \Omega' ;  \mathbb{ R }^{ 2 \times 2 }  ) ) $ for every $ p \in [1, \infty ) $. Moreover, we recall from  inequality  \ref{item:rotation_property_grad} that $ \norm{ \nabla u^{ \mathrm{aux} }_{ \varepsilon } }_{ \lp^{ 2 } } \lesssim \varepsilon^{ \gamma - 1 } $. {To control the last term of \eqref{eqn:timeIntAdjust}, we recall \eqref{useaswell} and then} apply \Cref{lemma:lp_bound_lemma} to find some $ {\alpha} > 0 $ and $ q \in [1 ,\infty) $ such that 
	\begin{equation}
		\varepsilon   
		\int_{ 0 }^{ t } 
		\int_{ { \mathcal{N}_\eps} }
		\abs{ \nabla  u^{ \mathrm{aux} }_{ \varepsilon } }^{ 2 } \abs{ \nabla \partial_t h }
		\dd{ x }
		\dd{ s }
		\lesssim
		\varepsilon^{ {\alpha} } 
		\int_{ 0 }^{ t }
		\norm{ \nabla\partial_t h }_{ \lp^{ q }( \Omega'  ) }
		\dd{ s },\label{eqn:errorvan}
	\end{equation}
	which vanishes as $ \varepsilon $ goes to zero.
	{Looking to the term involving the elastic tensor $\C$ within \eqref{eqn:timeIntAdjust}, for any $ p \in (2 , \infty ) $ with $ p' = p/(p-1) $, we estimate
	\begin{align*}
		&\left|\int_{ 0 }^{  t  }
		\int_{ \Omega' }
		( 1 - \chi_{ {\mathcal{N}_\eps} } ) 
		\mathbb{ C }  e(  u^{ \mathrm{aux} }_{ \varepsilon } ) \colon \sum\nolimits_{ j } \chi_{ P_{ j }^{ \varepsilon } } e \left( R_{ j }^{ \varepsilon } \partial_t h \right) 
		\dd{ x }
		\dd{ s }\right|
		\lesssim 
		\int_{ 0 }^{ t }
		\norm{ e ( u^{ \mathrm{aux}}_{ \varepsilon } ) }_{ \lp^{ p' } ( { \Omega'  \setminus \mathcal{N}_\eps} ) }
		\norm{ \nabla \partial_t  h }_{ \lp^{ p } (  \Omega'  ) } 
		\dd{ s }.
	\end{align*}
	As $\mathcal{L}^2(\Omega'\setminus \mathcal{N}_\eps)\lesssim \eps^2$ uniformly in time (see \eqref{eq:measure_of_non_nbhd_vanishes}) and $ e ( u^{ \mathrm{aux} }_{ \varepsilon } ) $ is uniformly in time $\lp^{p'}$-equiintegrable due to the $ \lp^{ 2 } $-bound in \eqref{eqn:uTildeAPriori}, the above term vanishes as $\eps\to 0$.  Combining this with \eqref{useaswell}--\eqref{eqn:errorvan} we get 
\begin{align}\label{5Y}
 \liminf_{\eps \to 0}	\dfrac{ 1 }{ \varepsilon }
		\int_{ 0 }^{  t  } \int_{ \mathcal{N}_\eps}
		{\langle \diff W ( \nabla y_{ \varepsilon } ) ,\nabla  \partial_t h \rangle}
		\dd{ x } \dd{ s }
		\le 	 \liminf_{\eps \to 0}
		\int_{ 0 }^{  t  } \int_{ \Omega'}
						\mathbb{ C }  e ( u^{ \mathrm{aux}}_{ \varepsilon }) \colon \sum\nolimits_{ j } \chi_{ P_{ j }^{ \varepsilon } } e \left( R_{ j }^{ \varepsilon } \partial_t h  
		\right)
		\dd{ x } \dd{ s }.		
\end{align}
 	Likewise, as $\mathcal{L}^2(\Omega'\setminus \omega_\eps)$ vanishes uniformly in time, we may replace $u^{\rm aux}_\eps$ by $u_\eps$ in the integral  on the right-hand side.  }
	
%
%
 Due to  the well-preparedness of the initial conditions in the sense of  inequality  (\ref{eq:well_preparedness_initial_condition}) and the lower semicontinuity of the total energies (see  inequalities  \eqref{eq:lsc_for_quadratic_term} and \eqref{eq:lsc_jump_sets_union_times}),  {we take the limit inferior as $\eps\to 0$ of     \eqref{eq:integral_over_time}   and use    \eqref{5Y} (with $u_\eps$ in place of $u^{\rm aux}_\eps$)    to recover}
	\begin{align*}
		\energyLinTot ( t )
		& \leq
		\liminf_{ \varepsilon \to 0 }\left[
		\dfrac{ 1 }{ \varepsilon^{ 2 } }
		\int_{ \Omega' }
		W ( \nabla y_{ \varepsilon }(t) ) 
		\dd{ x }
		+
		\dfrac{ 1 }{ \varepsilon^{ 2 \beta } }
		\int_{ \Omega' }
		\abs{ \nabla^{ 2 } y_{ \varepsilon }(t) }^{ 2 }
		\dd{ x }
		+
		\kappa \hm^{ 1 } \left(
		\Gamma_{ \varepsilon } ( t )
		\right)\right] \\
		&
		\leq
		\energyLinTot ( 0 )
		+
		\liminf_{ \varepsilon \to 0 }
		\int_{ 0 }^{ t }
		\int_{  \Omega'  }
		\mathbb{ C } e ( u_{ \varepsilon } ) \colon \sum\nolimits_{ j } \chi_{ P_{ j }^{ \varepsilon } } e \left( R_{ j }^{ \varepsilon } \partial_t h \right) 
		\dd{ x }
		\dd{ s }.
	\end{align*}
	 {In view} of \eqref{totalepsenergy},  to obtain the correct estimate (\ref{eq:energy_balance_lsc}), we must get rid of the rotations in the above estimate.
	First,  with the notation in \Cref{cor:approx_euler_lagrange_consequence},  for  partition elements that  intersect the boundary,  i.e.,  $ P_{ j }^{ \varepsilon } \notin \intpcs_\eps $, we have $ R_{ j }^{ \varepsilon } = \mathrm{Id} $,  see  equation  \eqref{def:y_rot} and the subsequent explanation.   For the remaining elements $ P_{ j }^{ \varepsilon } \in \intpcs_\eps $ that do not intersect the boundary, we apply \Cref{cor:approx_euler_lagrange_consequence} twice: once to remove the broken off pieces with $\hat R_j^\eps = R_j^\eps$ and subsequently with $\hat R_j^\eps = {\rm Id}$ to put them back in the inequality. This yields the claim.
	
	\emph{Step 2 (Upper estimate).}  As mentioned,  we can follow the proof of  \cite[Lem.~3.7]{francfort_larsen_existence_and_convergence_for_quasi_static_evolution_in_brittle_fracture}, up to  replacing the full gradient   with the symmetric gradient. We have written it down for the convenience of the reader in \Cref{sec:appendix_proofs}. 
\end{proof}

\begin{remark}\label{rmk:totalEnergyUC}
One can also show that for each $t',t\in [0,1]$ with $t'<t$ it holds that  
\begin{align}\label{rmk:totalEnergyUC-equ}
\energyNonlinTot(t) - \sigma_\eps \leq \energyNonlinTot(t') + \int_{t'}^t \int_{ \Omega' }
		 \mathbb{ C } e( u_{ \varepsilon } ( s ) ) \colon { e( \partial_t h ( s ) ) }
		\dd{ x }
		\dd{ s }
		\end{align}
		for $\sigma_\eps\to 0$ as $\eps \to 0$ (independently of $t'$ and $t$). Indeed, for $t',t \in I_\eps$ this follows by applying the linearization argument used above to the inequality found in \Cref{rmk:improvedEpsEnergyBalance}. 		For general $t,t' \in [0,1]$, we can choose $t_\eps,t_\eps' \in I_\eps$  with $|t-t_\eps|, |t'-t_\eps'|\le 	\Delta_{ \varepsilon }$ (see \eqref{eqn:deltaEps}) and $\energyNonlinTot(t) = \energyNonlinTot(t_\eps)$, $\energyNonlinTot(t') = \energyNonlinTot(t_\eps')$. Then, we use  inequality  \eqref{rmk:totalEnergyUC-equ} for $t_\eps$ and $t_\eps'$ and observe that, due to  the uniform bounds on the symmetric gradients  \eqref{eqn:bdryData} and \eqref{eqn:uTildeAPriori}, the integral $ \int_{t'}^t $ can be replaced $ \int_{t_\eps'}^{t_\eps} $ up to an error that can be absorbed in $\sigma_\eps$. 		
		 \end{remark}

\section{Proof of Theorem \ref{thm:main}}\label{sec:completeProof}
We complete the proof of the linearization result. Recall \ref{item:initial_condition_quasistatic_fracture_evolution}-\ref{item:energy_balance_def} of \Cref{def:linearQuasistatic}, which must be shown for $u$ and $ \Gamma $ from \Cref{def:displacement II}.
First, \ref{item:initial_condition_quasistatic_fracture_evolution} (initial condition) follows from the choice of $u_0$ and the well-preparedness of $(y^\eps_0)_\eps$ as stated in Theorem \ref{thm:main}. (For sequences of initial data provided by \cite[Thm.~2.7]{friedrich_griffith_energies_as_small_strain_limit_of_nonlinear_models_for_nomsimple_brittle_materials}, see Remark \ref{rem: well}(i), no modifications in \eqref{def:y_rot}  and \eqref{def:u_eps}  are needed, i.e.,  $u_\eps(0) = \bar{u}_\eps(0) =  (y_{ \varepsilon }^{ 0 } - \mathrm{id})/\varepsilon$ which converges to $u_0$ in measure on $\Omega'$.)
 By definition of $\Gamma(t)$, the crack will satisfy \ref{item:irreversibilty } (irreversibility). In Lemma \ref{lemma:cont_extension}, we prove that it satisfies \ref{item:displacement_and_boundary_conditions} (displacement and boundary conditions) and \ref{item:minimality} (minimality), along with the convergence $e(u_\eps)\to e(u)$.  Then,  in \Cref{lem:energyBalance}, we prove \ref{item:energy_balance_def} (energy balance).   Eventually, we will show energy convergence and $y_\eps \rightsquigarrow u $. 

\begin{lemma}
	\label{lemma:cont_extension}
	Let $ u $ and $\Gamma$ be as in Definition \ref{def:displacement II}.
	{Then, we} have $ e ( u ) \in \lp^{ \infty }  ( {(0 , 1)}  ; \lp^{ 2 } ( \Omega';  \R^{2 \times 2}_{\rm sym}   )  ) $  and  $u(t) = h(t)$ in $\Omega'\setminus \overline{\Omega}$ as well as  $ \hm^{ 1 } ( \Gamma ( t ) ) < \infty $ for all $ t \in [0,1] $. Furthermore, for any $ t \in  [0, 1 ]  $,  it holds that
	\begin{equation}
		\label{eq:minimality_u_whole_interval}
		\int_{ \Omega' }
		\dfrac{ 1 }{ 2 }
		Q \big( e ( u ( t ) ) \big)
		\dd{ x }
		\leq
		\int_{ \Omega' }
		\dfrac{ 1 }{ 2 }
		Q ( e ( v ) )
		\dd{ x }
		+
		\kappa 
		\hm^{ 1 } ( J_{ v } \setminus \Gamma ( t ) )
	\end{equation}
	among all $ v \in \gsbdtwo ( \Omega' ) $ such that $ v = h( t ) $ on $ \Omega' \setminus \overline{\Omega } $. Moreover,
	\begin{equation}
		\label{eq:jump_contained_in_gamma}
		J_{ u ( t ) } \subseteq \Gamma ( t ) 
		\text{ up to a set of }\hm^{ 1 }\text{-measure zero.}
	\end{equation}
	Finally, without taking subsequences in $\eps$, we have that for almost every $ t \in [0,1 ] $,  both $ e ( u_{ \varepsilon } ( t ) ) $ and $e  ( u_{ \varepsilon }^{ \mathrm{aux} } ( t )  ) $
 converge  to $ e ( u(  t) ) $ in $ \lp^{ 2 }  ( \Omega' ; \mathbb{ R }^{ 2 \times 2 }_{\rm sym} )$.
\end{lemma}
\begin{proof}
The proof is the same (up to obvious modifications) as the proof of \cite[Lem.~3.8]{francfort_larsen_existence_and_convergence_for_quasi_static_evolution_in_brittle_fracture}, except for the strong convergence. We therefore prove only the strong convergence and the measurability here. For  convenience of the reader, the rest of the proof can be found in \Cref{sec:appendix_proofs}.

 We want to show the strong convergence of the symmetric gradients by comparing the functions $ u( t ) $ and $ \hat{u}( t ) $ from  Remark \ref{def:displacement I} and Definition \ref{def:displacement II}.   The  minimizing property of $ \hat{u}( t ) $ from Lemma \ref{lemma:u_minimzer_of_total_energy} yields in particular that it minimizes
\begin{equation*}
	\int_{\Omega' }
	\dfrac{ 1 }{2 }
	Q ( e ( v ) ) 
	\dd{ x }
\end{equation*}
among all $ v \in \gsbdtwo( \Omega' ) $ with $ v = h ( t ) $ on $ \Omega' \setminus \overline{\Omega } $ and $ J_{ v } \subseteq \Gamma ( t ) \cup J_{ \hat{u} ( t ) } $. By $u(t) = h(t)$ in $\Omega'\setminus \overline{\Omega}$  and \eqref{eq:jump_contained_in_gamma} (as shown in in \Cref{sec:appendix_proofs}), the function $ u ( t ) $ is an admissible competitor, from which we deduce that
\begin{equation}
	\int_{ \Omega'}
	\dfrac{ 1 }{ 2 }
	Q \big( e ( \hat{u} ( t ) ) \big)
	\dd{ x }
	\leq
	\int_{ \Omega' }
	\dfrac{ 1 }{ 2 }
	Q \big( e ( u ( t ) ) \big)
	\dd{ x }.\label{eqn:toBeReversed}
\end{equation}
The main point now is to prove the reverse inequality  for a.e.\ $t \in [0,1]$.  In fact, once this is shown, we must have $ e ( u ( t ) ) = e ( \hat{u} ( t ) ) $ almost everywhere due to the strict convexity of $Q$  on $\R^{2\times 2}_{\rm sym}$ which implies that a minimizer over the convex subspace 
$${\big\{e(v) : v\in \gsbdtwo(\Omega'), \ v = h(t) \text{ in }\Omega'\setminus \overline{\Omega},\text{ and }J_v\subseteq \Gamma(t) \cup J_{\hat{u}(t)}  \big\} \subseteq \lp^{ 2 }(\Omega'; \R^{2 \times 2}_{\rm sym})}$$
must be unique. Therefore,  by  \Cref{cor:strong_l2_convergence_of_uepsaux},  $e(u_\eps(t))\to e(\hat u(t)) = e(u(t))$ in $\lp^{ 2 } (\Omega';\R^{2\times 2}_{\rm sym})$ up to a $t$-dependent subsequence. The analog convergence of $e(u^{\rm aux}_\eps(t))$ follows  from the same corollary. 
However, because the limit $e(u  (t)  )$ is independent of the subsequence $\eps$, Urysohn's property (if all subsequences converge to the same point, then the whole sequence converges to this point) shows that we do not need to take a subsequence in $\eps$ to conclude $e(u_\eps(t))\to e(u  (t)  )$ in $\lp^2(\Omega';\R^{2 \times 2}_{\rm sym})$. As we now have pointwise convergence in time, an application of Mazur's lemma says the pointwise and weak limits are the same. Thus,  using  the uniform boundedness \eqref{eqn:uTildeAPriori},  $e(u_\eps)\rightharpoonup e(u)$ in $\lp^p ({(0,1)};\lp^{ 2 } (\Omega; \R^{2 \times 2}_{\rm sym}))$ for any $1<p<\infty$, so in particular, $e(u)$ is space-time measurable, and we finally have $e(u)\in \lp^\infty ({(0,1)};\lp^2 (\Omega; \R^{2 \times 2}_{\rm sym}))$. 

We are left to prove the  reverse of inequality \eqref{eqn:toBeReversed}. To this end,  we consider
\begin{equation*}
	{l_{ \varepsilon } ( t ) 
	\coloneqq 
	\hm^{ 1 } ( {\Gamma}_{ \varepsilon } ( t ) ).}
\end{equation*}
The functions $ l_{ \varepsilon } $ are nondecreasing and uniformly bounded (see    \eqref{eq:a_priori_estimate}), thus by Helly's Theorem we find a non-decreasing function $ \lambda $ on $ [0,1] $ such that for a  subsequence of $\eps$ (not relabeled), we have $ l_{ \varepsilon } \to \lambda $ pointwise.
Let $ H $ be the at most countable set of discontinuity points of $ \lambda $. Let $ t \notin H $. If $t \in I_\infty$, \eqref{eqn:toBeReversed} follows directly as $u(t) = \hat{u}(t)$ by Remark \ref{def:displaceAtGoodTimes}. We thus assume $t \notin I_\infty$ and let $ (t_{ p })_p  \subseteq I_\infty   $ be as from Definition \ref{def:displacement II}. Using the approximate minimality of $ u_{ \varepsilon } ( t_{ p } ) $  given {by \Cref{rmk:approximate_minimality_alter}} with the admissible competitor $ u_{ \varepsilon } ( t ) + {h_\eps}( t_{ p } ) - {h_\eps}( t ) $ {along with  convergence  \eqref{eq:modifications_barely_increase_jump},} we get 
\begin{align}
	\int_{ \Omega' }
	\dfrac{ 1 }{ 2 }
	Q \big( e ( u_{\varepsilon } ( t_{ p } ) ) \big)
	\dd{ x }
	-
	\rho_{ \varepsilon }
	& \leq{} 
	\int_{ \Omega' }
	\dfrac{ 1 }{ 2 }
	Q \big( e( u_{ \varepsilon } ( t ) + {h_\eps}( t_{ p } ) - {h_\eps}( t ) ) \big)
	\dd{ x }
	+
	\kappa \hm^{ 1 } \left(
	\left( J_{ u_{ \varepsilon } ( t ) }   \cup  J_{\nabla  u_{ \varepsilon } ( t ) } \right) 
	\setminus {{\Gamma}_\eps} ( t_p ) 
	\right)
	\notag	\\
	\label{eq:min_equation_in_tp}
	& \leq{}
	\int_{ \Omega' }
	\dfrac{ 1 }{ 2 }
	Q \big( e( u_{ \varepsilon } ( t ) + {h_\eps}( t_{ p } ) - {h_\eps}( t ) ) \big)
	\dd{ x }
	+
	\kappa ( l_{ \varepsilon } ( t ) - l_{ \varepsilon } ( t_{ p } ) )  + \sigma_\eps, 
\end{align}
where  $\sigma_\eps \to 0 $ as $\eps \to 0$  and where by inequality (\ref{eq:estimate_rhoeps_alt}) we can estimate
\begin{equation}\nonumber
	\rho_{ \varepsilon }
	\lesssim 
	 \delta_{ \varepsilon }
	+
	\varepsilon^{ 2 \left( 1 - \beta \right)}
	\int_{ \Omega' }
	\abs{ \nabla^{ 2 } (u_{ \varepsilon } ( t ) + {h_\eps}( t_{ p } ) - {h_\eps}( t ) ) }^{ 2 } \dd{ x }
	 +
	\int_{ \Omega' }
	\varepsilon \abs{ \nabla ( u_{ \varepsilon } ( t ) + {h_\eps}( t_{ p } ) - {h_\eps}( t ) ) }^{ 3 }
	\dd{ x }. 
\end{equation}
The first summand vanishes since $ \delta_{ \varepsilon} \to 0 $. The second summand vanishes by {\eqref{eq:second_derivative_vanishes} in} \Cref{cor:convergence of energies} and because $h \in  \lp^{ \infty }( (0, 1 );  \wkp^{2, \infty }(\Omega';\R^{2}))$. 
 The third summand vanishes by the choice of the cutoff, see  \ref{item:properties_of_etaeps_slowness},   the definition of $ u_{ \varepsilon } $   in  (\ref{def:u_eps}),  and again  by the regularity of $h$.   

Sending $ \varepsilon $ to zero in inequality (\ref{eq:min_equation_in_tp}) yields by lower semicontinuity on the left-hand side and the strong convergence given by  \Cref{cor:strong_l2_convergence_of_uepsaux}  that
\begin{equation*}
	\int_{ \Omega' }
	\dfrac{ 1 }{ 2 }
	Q \big( e ( u ( t_{ p  } ) ) \big)
	\dd{ x }
	\leq
	\int_{\Omega'}
	\dfrac{ 1 }{ 2 }
	Q \big( e ( \hat{u} (t ) + h( t_{ p } ) - h ( t ) ) \big)
	\dd{ x }
	+
	\kappa ( \lambda ( t ) - \lambda ( t_{ p } ) ).
\end{equation*}
Sending $p \to \infty$ yields the desired inequality, again by lower semicontinuity and the choice of $ t $ as a continuity point of $ \lambda $. This completes the proof.
\end{proof}

\begin{remark}\label{rmk:L2StrongEverywhere}
The result of \Cref{lemma:cont_extension} may be strengthened in the sense that $e(u_\eps(t))$ and $e(u_\eps^{\rm aux}(t))$ converge to $ e(u(t)) = e(\hat{u}(t))$ in $\lp^2(\Omega';\R^{2\times 2}_{\rm sym})$ for \emph{all} $t\in [0,1]$. Indeed, the proof above shows that this convergence and the identification of $e(u)$ with $e(\hat{u})$ holds on  $ [0,1]\setminus H$, where $H$ is the set of discontinuity points for $\lambda$. Moreover, by \Cref{cor:strong_l2_convergence_of_uepsaux} this convergence also always holds on $ I_\infty$ as $u = \hat{u}$ on $I_\infty$, see Remark \ref{def:displaceAtGoodTimes}. Therefore, to obtain the claimed property, it suffices to show that $H \subseteq I_\infty$.    In the proof of \Cref{thm:main} below, we will see that $\lambda(t) = l(t) := \mathcal{H}^{1}(\Gamma(t))$ for a.e.\ $t\in [0,1]$. Assuming this for the moment, since $\lambda$ is monotone increasing and $l$ is continuous at all times $t \not \in I_\infty$ (see  equation  \eqref{eq:irreversible_crack}), this implies that $\lambda(t)  = l(t)$ for all $t\not \in I_\infty$, and  then  $H\subseteq I_\infty$. 
\end{remark}


The next lemma establishes the desired energy balance.
\begin{lemma}\label{lem:energyBalance}
	The function $ \energyLinTot $ is absolutely continuous on $ [0, 1 ] $ and satisfies the energy balance
	\begin{equation*}
		\energyLinTot ( t )
		=
		\energyLinTot ( 0 )
		+ \int_{ 0 }^{ t }
		\int_{ \Omega' }
		  \mathbb{ C } e ( u ( s ) )  \colon  e ( \partial_t h )
		\dd{ x }
		\dd s.
	\end{equation*}
\end{lemma}
\begin{proof}
	Relying on \Cref{lemma:energy_balance_inequalities}  and  inequality  \eqref{eq:minimality_u_whole_interval} in \Cref{lemma:cont_extension}, the proof is exactly the same as in \cite[Lem.~3.9]{francfort_larsen_existence_and_convergence_for_quasi_static_evolution_in_brittle_fracture}.
\end{proof}

We close with the proof of \Cref{thm:main}.

\begin{proof}[Proof of \Cref{thm:main}]
By combining  the choice of $u_0$  have to discuss this, see earlier comment ,  \Cref{def:displacement II}, \Cref{lemma:cont_extension}, and \Cref{lem:energyBalance} we find that $(u(t),\Gamma(t))_{t \in [0,1]}$ satisfies \ref{item:initial_condition_quasistatic_fracture_evolution}--\ref{item:energy_balance_def} in Definition \ref{def:linearQuasistatic}. {It remains to prove convergence of the total energies and the rescaled deformations.}

{\emph{Step 1 (Convergence of the total energies).}
By \Cref{lemma:cont_extension}, $e(u_\eps)\weaklystar e(u)$ in $\lp^\infty((0,1) ; \lp^2(\Omega';\R^{2\times 2}_{\rm sym}))$  up to a subsequence,  so that using \Cref{lem:energyBalance}, we deduce from \eqref{eq:energy_balance_lsc} that
$\energyLinTot(t)  = \liminf_{\eps\to 0} \energyNonlinTot(t) $ for all $t\in I_\infty.$
Since this holds for any further subsequence in $\eps$, we recover 
\begin{equation}\label{eqn:totalOnGoodTimes}
\energyLinTot(t)  = \lim_{\eps\to 0} \energyNonlinTot(t) \quad \text{ for all }t\in I_\infty.
\end{equation} To obtain convergence at any time $t\in [0,1]$, take $t'\in I_\infty$ with $t'<t$ and use \Cref{rmk:totalEnergyUC} with  equation \eqref{eqn:totalOnGoodTimes} to find
$$\limsup_{\eps\to 0}\energyNonlinTot(t) \leq \energyLinTot(t') + \int_{t'}^t \int_{ \textcolor{teal}{\Omega'} }
		 \mathbb{ C } e( u ( s ) ) \colon { e( \partial_t h ( s ) ) }
		\dd{ x }
		\dd{ s }.$$
Taking $t'\uparrow t$ and using the continuity of $\energyLinTot$ (from \Cref{lem:energyBalance}), we have $\limsup_{\eps\to 0}\energyNonlinTot(t) \leq \energyLinTot(t)$. 
Analogously, we find $\liminf_{\eps\to 0}\energyNonlinTot(t) \geq \energyLinTot(t)$ by taking $t'>t$, thereby concluding convergence of the energies at every time.}
 Using the convergence of the energy along with \Cref{cor:convergence of energies} and the fact that $e(\hat{u}) = e(u)$   a.e.\ on $[0,1]$   (as seen in the proof of  \Cref{lemma:cont_extension}), for a.e.\ $t \in [0,1]$ we get that
\begin{align}\label{separate}
\lim_{ \varepsilon \to 0 }
		\dfrac{ 1 }{ \varepsilon^{ 2 } }
		\int_{ \Omega' }
		W  ( \nabla y_{ \varepsilon } ( t ) ) 
		\dd{ x }
		+
		\dfrac{ 1 }{ \varepsilon^{ 2 \beta } }
		\int_{ \Omega' }
		\abs{ \nabla^{ 2 } y_{ \varepsilon } ( t ) }^{ 2 }
		\dd{ x }
		=
		\int_{ \Omega' }
		\dfrac{ 1 }{ 2 }
		Q \big( e ( u ( t ) ) \big) 
		\dd{ x }, \quad \mathcal{H}^1(\Gamma_\eps(t)) \to  \mathcal{H}^1(\Gamma(t)).
\end{align}
In view of    \eqref{separate},   we can now apply  \Cref{rmk:L2StrongEverywhere}  (which still needed $ \lambda ( t ) =  l   (t  ) $ for a.e.\ $ t \in [0,1] $ for its conclusion)  and we   infer  that  $e(u(t))= e(\hat u(t))$ in $\lp^2(\Omega';\R^{2\times 2}_{\rm sym})$ for \emph{all}  $t\in [0,1]$.  Again using \Cref{cor:convergence of energies},   this  then implies that \eqref{separate} holds for all $t\in [0,1]$.

{\emph{Step 2 (Convergence of the rescaled deformations).}}
To show  $y_\eps \rightsquigarrow u $, i.e., to confirm  convergences  \eqref{baruconv0} and \eqref{baruconv00}, we will construct sequences of deformations, which are almost minimizers for the minimization problem \eqref{eq:discrete_min_problem}, and use these to identify the limit of $\bar u_\eps(t): = (y_\eps(t)-{\rm id})/\eps$ on $G(t)$.

Recalling the definition of $\intpcs_\eps$ in \Cref{cor:approx_euler_lagrange_consequence}, we define  $D^1_\eps(t) = \bigcup_{ P_j^\eps(t) \notin \intpcs_\eps(t)  } P_j^\eps(t)  $  and 
$$y^*_\eps(t) =  \chi_{D^1_\eps(t)} y_{ \varepsilon }(t) + \chi_{\Omega'\setminus D^1_\eps(t)} {\rm id}, \quad \quad u^*_\eps(t) = \dfrac{ 1 }{ \varepsilon } ( y^*_{ \varepsilon }(t) - \mathrm{id} ). $$
Recalling the definition of $\bar{u}_\eps$ and  the definition in \eqref{eqn:tildeuEps}, this shows $u_\eps^{\rm aux}(t) \chi_{D^1_\eps(t)}  = u^*_\eps(t)$,  and
\begin{align}\label{DDD}
u^*_\eps(t) = \bar{u}_\eps(t) \quad \text{ on } D^1_\eps(t). 
\end{align} 
  We apply \Cref{thm:gsbd2_compactness} {to} the sequence $(u^*_\eps(t))_\eps$ to find collections of sets $\mathcal{S}_\eps  =(S_j^\eps)_{j\ge 0 } $ contained in $\Omega$ and  a limit function $ u^*(t)   \in  \gsbdtwo ( \Omega' )$  for the modified sequence, see \Cref{thm:gsbd2_compactness}\ref{item:rescaled_un }.
Denoting by $D^2_\eps(t) : = \Omega' \setminus  \bigcup_{j \ge 0 }  S_j^\eps$, {we introduce} $D_\eps(t) := D^1_\eps(t) \cap D^2_\eps(t)$ and, for each $\eta \in \R^2$, define $u^*_{\eps,\eta}(t) := \chi_{D_\eps(t)} u^*_\eps(t) + \chi_{\Omega' \setminus D_\eps(t)}{\eta}$. 
By \Cref{lem:aprioriEst},  inequality \ref{item:rotation_property_jump}, \Cref{thm:gsbd2_compactness}\ref{item:caccioppoli_partitions_small_jump}, and compactness of sets of finite perimeters we find $D(t) \subseteq \Omega'$ such that $\chi_{D_\eps(t)} \to \chi_{D(t)}$ {in $\lp^1(\Omega')$}  for a $t$-dependent subsequence of $\eps$.  Since $ D_{ \varepsilon } ( t ) \supseteq \Omega' \setminus \overline{ \Omega } $ for all $ \varepsilon > 0 $,  we have $D(t) \supseteq  \Omega' \setminus \overline{\Omega}$. 
Moreover, by the already applied \Cref{thm:gsbd2_compactness}, we have, for a subsequence depending on $t$, that  
\begin{align}\label{strangecon}
u^*_{\eps,\eta}(t) \to {u^*_\eta}(t)  := {\chi_{D(t)}  {u^*}(t) + \chi_{\Omega' \setminus D(t)} \eta} \quad \text{in measure on $\Omega'$,}
\end{align}
where  $u^*(t)$ is the limit of the sequence $(u^*_\eps(t))_\eps$  found above. Note that ${u^*_\eta}(t) = h(t)$ on $\Omega' \setminus \overline{\Omega}$  and $ \hm^{ 1 } ( J_{ u_{ \varepsilon , \eta }^{ \ast} ( t ) } \setminus \Gamma_\eps ( t ) ) \to 0 $. 
Thus, with the same proof as for \Cref{lemma:u_minimzer_of_total_energy},  ${u^*_\eta}(t)$ is a minimizer of the functional
\begin{equation*}
	\int_{ \Omega' } \frac{ 1 }{ 2 }
	Q ( e ( v ) ) \dd{ x }
	+ 
	\kappa 
	\hm^{ 1 } \big( J_{ v } \setminus \big( \Gamma ( t ) \cup J_{ u^{ \ast }_{ \eta }( t ) }\big) \big)
\end{equation*}
among all $ v \in \gsbdtwo ( \Omega' ) $ with $ v = h( t ) $ on $ \Omega' \setminus \overline{ \Omega } $. 



  To identify the strain $e({u^*_\eta})$, we first repeat the lower semicontinuity argument in  inequalities  \eqref{eq:lsc_jump_sets_union_times}--\eqref{for the end}, by adding  the jump set of ${u^*_\eta}(t)$ to the union  on the left-hand side. 
  This gives $$	 \hm^{ 1 } 
	 \left(
	\Gamma(t) \cup J_{{u^*_\eta}(t)}   
 \right)
	 \le 
	\liminf_{ \varepsilon \to 0 }
	\hm^{ 1 } \left(
	\Gamma_{ \varepsilon } ( t )
	\right). $$
More precisely, if $t\notin I_\infty$, we perform this estimate with $\Gamma(\tau)$ in place of $\Gamma(t)$ on the left-hand side for any $\tau \in I_\infty$, $\tau \le t$, and then we pass to the limit $\tau  \uparrow t$ using the continuity of $t \mapsto \mathcal{H}^1(\Gamma(t))$ on $[0,1] \setminus I_\infty$. This along with \eqref{separate} (for every $t\in [0,1]$) shows  $J_{{u^*_\eta}(t)} \subseteq  \Gamma(t)$ up to an $\mathcal{H}^1$-negligble set. Then arguing as for inequality \eqref{eqn:toBeReversed} in \Cref{lemma:cont_extension}, since ${u^*_\eta}(t)$ and $u  (t)  $ solve the same minimization problem with the crack $\Gamma  (t)  $ fixed, we have that 
\begin{align}\label{samestrain}
 e({u^*_\eta}(t))  =  e(u(t)) \quad {\text{ for every }t\in [0,1]}.
\end{align} 
In view of the definition of ${u^*_\eta}(t)$, see \eqref{strangecon}, for a particular choice of $\eta\in\R^2$ (actually, for a.e.\ choice) we get $\partial^*  D(t)\subseteq  J_{{u^*_\eta}(t)} \subseteq \Gamma(t) $ up to {an $\mathcal{H}^1$-negligible set}. 
Then, by the definition of $G(t)$ preceding \eqref{baruconv0} and the fact that $D(t) \supseteq \Omega' \setminus \overline{\Omega}$, this shows that $G(t) \subseteq D(t)$ up to an $\mathcal{L}^2$-negligible set. 
Therefore,  $\mathcal{L}^2(G(t) \setminus D_\eps(t)) \to 0$ and thus, as $\bar{u}_\eps(t)   = u^*_{\eps}(t)  = u^*_{\eps,\eta}(t)$ on $D_\eps(t)$  by  equation  \eqref{DDD},   by \eqref{strangecon} we deduce
 \begin{align}\label{baruconv}
 \bar{u}_\eps(t) \to    {u^*}(t) \quad \text{ in measure on $G(t)$}. 
 \end{align} 
 In view of  equation  \eqref{samestrain},  \cite[Thm.~A.1]{chambolle_giacomini_ponsiglione_piecewise_rigidity} shows that $u(t) -  {u^*_0}(t)  $ is a piecewise rigid function {on a Caccioppoli partition $\mathcal{P}$ of $\Omega'$} with boundaries contained in $J_{u(t)} \cup  J_{{u^*_0}(t)}  \subseteq   \Gamma(t)$. Let $\mathcal{P}_D$ be the elements  of $\mathcal{P}$  whose intersection with $\Omega'\setminus \overline{\Omega}$ have positive measure.  Due to the boundary conditions, $u(t) - u^*_0(t) = 0$ on every $P\in \mathcal{P}_D$. Further, we must have that $G(t)\subset \bigcup_{P\in \mathcal{P}_D} P,$ otherwise $[G(t)\cap (\bigcup_{P\in \mathcal{P}_D} P)]^c$ contradicts the maximality of the broken off piece $B(t)$. Altogether, this shows $u(t) - {u^*_0}(t) = 0$ on $G(t)$. Using  convergence \eqref{baruconv}  and $u^*(t) = u^*_0(t)$ on $G(t)$  this argument uniquely identifies the limit of $\bar u_\eps(t)$ on $G(t)$ for any subsequence as $u(t)$, thereby implying the first part of \eqref{baruconv0}.


As  $e  ( u_{ \varepsilon }^{ \mathrm{aux} } ( t )  ) $
 converges  to $ e ( u(  t) )    $ in $ \lp^{ 2 }  ( \Omega' ; \mathbb{ R }^{ 2 \times 2 }_{\rm sym} )$ for  all  $t\in [0,1]$ by \Cref{lemma:cont_extension}  and Remark \ref{rmk:L2StrongEverywhere}, we particularly find      $e  ( u_{ \varepsilon }^{ \mathrm{aux} } ( t )  )  \to  e ( u(  t) )$ in measure on $G(t)$.    As $u^{\rm aux}_\eps(t) = \bar{u}_\eps(t)$ on $D_\eps(t)$  by  equation  \eqref{DDD} and   $\mathcal{L}^2(G(t) \setminus D_\eps(t)) \to 0$,    the second  part of \eqref{baruconv0} follows   (noting that    this reasoning holds for every subsequence).

Eventually, we come to the proof of  convergence  \eqref{baruconv00}.  For any $t \in [0,1]$, recalling $\partial^* G(t) \subseteq \Gamma(t) $ up to an $\mathcal{H}^1$-negligible set and using $\chi_{G(t)} u(t)$ as a competitor in  inequality  \eqref{eq:minimality_u_whole_interval}, we find that
$$e(u(t)) = 0 \quad  \text{ on } B(t). $$
Then, repeating the lower semicontinuity argument \eqref{eq:lsc_for_quadratic_term} on $G(t)$, we find
 \begin{align*}
 \int_{ \Omega' }
	\dfrac{ 1 }{ 2 }
	Q\big( e ( u(t) ) \big)
	\dd{ x } & = \int_{ G(t) }
	\dfrac{ 1 }{ 2 }
	Q\big( e ( u(t) ) \big)
	\dd{ x } \le \liminf_{ \varepsilon \to 0 }
		\dfrac{ 1 }{ \varepsilon^{ 2 } }
		\int_{ G(t)}
		W  ( \nabla y_{ \varepsilon } ( t ) ) 
		\dd{ x } \\ & =  \int_{ \Omega' }
	\dfrac{ 1 }{ 2 }
	Q\big( e ( u(t) ) \big)
	\dd{ x } - \limsup_{\eps \to 0}  \dfrac{ 1 }{ \varepsilon^{ 2 } }
		\int_{ B(t)}
		W  ( \nabla y_{ \varepsilon } ( t ) ) 
		\dd{ x },
  \end{align*}
 where the last identity follows from \Cref{cor:convergence of energies} and the fact that $e(\hat{u}(t)) = e(u(t))$ {for every} $t \in [0,1]$  (see Remark \ref{rmk:L2StrongEverywhere}).  This along with  \ref{eq:W_zero_on_sod} shows  convergence   \eqref{baruconv00},  which finishes the proof . 
\end{proof}

%

%

 \section*{Acknowledgements} 
 This research was funded by the Deutsche Forschungsgemeinschaft (DFG, German Research Foundation) - 377472739/GRK 2423/2-2023. M.F.\ is very grateful for this support.  
  K.S. and P.S. were supported by funding from the Deutsche Forschungsgemeinschaft (DFG, German Research Foundation) under Germany’s Excellence Strategy – EXC-2047/1 – 390685813, the DFG project 211504053 - SFB 1060.
 K.S. was also supported by funding from the NSF (USA) RTG grant DMS-2136198.

\appendix

\section{Bad sets, integral bounds, and measure theory}\label{sec:appendix}

\begin{lemma}[Bad set] 
	\label{lemma:characterization_of_bad_set}
	Let $ \Omega \subseteq \Omega' \subseteq \R^{ 2 } $ be bounded Lipschitz domains such that  also  $\Omega'\setminus \overline{\Omega}$ is  a Lipschitz  set, and suppose $ \Gamma \subseteq \Omega' \cap \overline{ \Omega } $ is a Borel set with $ \hm^{ 1 } ( \Gamma ) < \infty $. Then there exists a unique maximal measurable set $ B \subseteq \Omega $ (up to $ \lm^{ 2 } $-equivalence class and with respect to set inclusion) such that $ \partial^{ \ast } B \subseteq \Gamma $ (up to an $ \hm^{ 1 } $-nullset). If $ \Gamma $ is additionally closed, we have
	\begin{equation}
		\label{eq:charact_bad_set}
		B = \left\{ x \in \Omega \colon x \text{ is not path-connected to } \Omega' \setminus \overline{ \Omega } \text{ in } \Omega' \setminus \Gamma \right\}.
	\end{equation}
\end{lemma}

\begin{remark}
	It is convenient for the proof if we denote by $ \partial^{ \ast} $ the \emph{essential boundary} of a set (see \cite[Def.~3.60]{ambrosio_fusco_pallara_functions_of_bv_and_free_discontinuity_problems}), which is defined without assuming the set to be of finite perimeter. Note however that if $ \hm^{ 1 } ( \partial^{ \ast } A ) < \infty $, then by \cite[4.5.11]{federer_gmt}, $ A $ is already a set of finite perimeter and the essential and reduced boundary agree up to a set of $ \hm^{ 1 } $-measure zero.
\end{remark}

\begin{proof}
	We first prove the existence of a maximal set by using the direct method.
	Let 
	\begin{equation*}
		M \coloneqq 
		\left\{ 
		B \subseteq \Omega \text{ measurable } 
		\colon 
		\hm^{ 1 } ( \partial^{ \ast } B \setminus \Gamma ) = 0 
		\right\}.
	\end{equation*}
	Note that the empty set is an element of $ M $ and by boundedness of $ \Omega $, the Lebesgue measure of sets in $ M $ is uniformly bounded. Moreover, the essential boundary of every $ B \in M $ is bounded with respect to $ \hm^{ 1 } $. 
	Thus, for a maximizing sequence $ ( B_{ n })_{ n } \subseteq M $  with respect to $ \lm^{ 2 } $, we find a  subsequence (not relabeled) and some measurable set $ B $ such that $ B_{ n } \to B $ in measure. 
	By this convergence it follows that $ \lm^{ 2 } ( B ) = \max_{ B' \in  M  } \lm^{ 2 } ( B' ) $. Using lower semicontinuity, we have for every compact set $ K \subseteq \Gamma $ that
	\begin{equation*}
		\hm^{ 1 } ( \partial^{ \ast } B \setminus K )
		\leq
		\liminf_{ n \to \infty }
		\hm^{ 1 } ( \partial^{ \ast } B_{ n } \setminus K )
		\leq 
		\hm^{ 1 } ( \Gamma \setminus K ).
	\end{equation*}
	Sending $ K \to \Gamma $ yields that $ \hm^{ 1 } ( \partial^{ \ast } B \setminus \Gamma ) = 0 $, and thus $ B \in M $. Moreover $ B $ is maximal: If $ B ' \in M $, then $ B \cup B' \in M $ since $ \partial^{ \ast } ( B \cup B' ) \subseteq \partial^{ \ast } B \cup \partial^{ \ast } B' $. Thus, by $ \lm^{ 2 } $-maximality of $ B $, we must already have $ \lm^{ 2 } ( B' \setminus B ) = 0 $. This also proves the uniqueness.
	
	Now let us assume that $ \Gamma $ is closed. We want to show equation (\ref{eq:charact_bad_set}). Denoting by  $ \hat{B}  $ the right-hand side of  (\ref{eq:charact_bad_set}), we aim at checking  $B =  \hat{B} $ with $B$ above. First, we show $  \hat{B} \in M$.  Assume by contradiction that there exists  $ x \in \partial^{ \ast }   \hat{B} \setminus \Gamma $. Then,  by the closedness of $ \Gamma  $   we find some $ \delta > 0 $ such that $ B_{ \delta } ( x ) \subseteq \Omega' \setminus \Gamma $. Moreover, by the definition of the essential boundary, we find  some $  y \in B_{ \delta } ( x ) \setminus \hat{B} $. 
	By definition of $  \\hat{B}  $ and because $ B_{ \delta } ( x ) \subseteq \Omega' \setminus \Gamma $, by connecting elements $ z \in B_{ \delta ( x ) } $ with $ y $ by a straight segment, this yields  that $ B_{ \delta } ( x ) $ is also path-connected to $ \Omega' \setminus \overline{ \Omega } $ in $ \Omega' \setminus \Gamma $. Thus $ B_{ \delta } ( x ) \subseteq \Omega' \setminus \hat{B} $, a contradiction to $ x \in \partial^{ \ast } \bar{B} $. We thus conclude that $ \partial^{ \ast }  \hat{B} \subseteq \Gamma $, i.e., $ \hat{B} \in M$. 

To conclude $B =  \hat{B} $, it suffices to show that  $ \hat{B} $ is maximal. 
Let $ B ' \in M $ and suppose by contradiction that $ x \in B' \setminus \hat{B} $ is a point of Lebesque density $ 1 $ of $ B' \setminus  \hat{B} $. 
Then, we find a continuous path $ \gamma $ between $ x $ and $ \Omega' \setminus \overline{ \Omega } $ in $ \Omega' \setminus \Gamma $. 
Since $ \gamma $ is compact and $ \Gamma $ closed, we find some $ \delta > 0 $ such that $ B_{ \delta } ( \gamma ) \subseteq \Omega' \setminus \Gamma $. Since we have chosen $ x $ as a point of density $ 1 $ of $ B' \setminus  \hat{B}  $, we have $ \lm^{ 2 } (  B_{ \delta } ( \gamma ) \cap B' ) > 0 $. Moreover, because $ \Omega' \setminus \overline{ \Omega } $ is open, $ \gamma $ ends in $ \Omega' \setminus \overline{ \Omega } $,  and $ B' \subseteq \Omega $, we have $ \lm^{ 2 } ( B_{ \delta } ( \gamma ) \setminus B' ) > 0 $. As a consequence, we must have $ \hm^{ 1 } ( \partial^{ \ast } B' \cap B_{  \delta } ( \gamma ) ) > 0 $. This is a contradiction to $ B' \in M $ because $ \Gamma \cap B_{  \delta } ( \gamma ) = \emptyset $.
\end{proof}

\begin{lemma}[$\lp^{ p } $-bounds]
	\label{lemma:lp_bound_lemma}
	Let $ \Omega \subseteq \mathbb{ R }^{ 2 } $  with $ \lm^{ 2 } ( \Omega) < \infty $, $ f_{ \varepsilon } \in \lp^{ 2 } ( \Omega; [0,\infty) ) $ be a sequence such that $ \norm{ f_{ \varepsilon } }_{ \lp^{ 2 } ( \Omega )} \lesssim \varepsilon^{ \gamma - 1 } $ for some $ \gamma \in ( 2/3, 1) $, and let $ g \in \lp^{ p } ( \Omega; [0,\infty) ) $ for every $ p \in [1, \infty ) $. Then we find an exponent $ q \in ( 1, \infty) $ and some $ {\alpha} > 0 $ such that
	\begin{equation*}
		\varepsilon \int_{ \{ \varepsilon f_{ \varepsilon } \lesssim 1 \} }
		f_{ \varepsilon }^{ 2 } g 
		\dd{ x }
		\lesssim 
		\varepsilon^{ {\alpha} } \norm{ g }_{ \lp^{ q }  (\Omega)   }.
	\end{equation*}
\end{lemma}
\begin{proof}
Let $ 0<{\zeta} < 2 $ and let $ q \in  ( 1, \infty) $ be such that $ 1/q + (2- {\zeta} )/ 2  = 1 $. By Hölder's inequality we have
\begin{align*}
	\varepsilon \int_{ \{ \varepsilon f_{ \varepsilon } \lesssim 1 \} }
	f_{ \varepsilon }^{ 2 } g
	\dd{ x }
	&\lesssim 
	\varepsilon^{ 1 - {\zeta} }
	\int_{ \Omega } 
	f_{ \varepsilon }^{ 2 - {\zeta} } g 
	\dd{ x }
	\\
	& \leq
	\varepsilon^{ 1 - {\zeta} } \norm{ f_{ \varepsilon } }_{ \lp^{ 2 }  (\Omega)  }^{ 2 - {\zeta} } \norm{ g }_{ \lp^{ q } (\Omega)  }
	\\
	& \lesssim
	\varepsilon^{ 1 - {\zeta} + ( \gamma - 1 )( 2 - {\zeta} ) }
	\norm{ g }_{ \lp^{ q } (\Omega)  }.
\end{align*}
Since 
\begin{equation*}
	{\alpha} \coloneqq 
	1- {\zeta} + ( \gamma - 1 ) ( 2 - {\zeta} )
	=
	-1 + 2 \gamma - \gamma {\zeta} 
\end{equation*}
and $ -1 + 2 \gamma > 0 $, we can choose $ {\zeta} > 0 $ sufficiently small such that $ {\alpha} > 0 $, which finishes the proof.
\end{proof}

\begin{lemma}[Measure theory]
	\label{lemma:separating_with_disjoint_open_sets}
	Let $ A_{ 1 }, A_{ 2 } \subseteq  \mathbb{ R }^{ 2 }  $ be Borel sets with $ \hm^{ 1 } ( A_{ 1 } ), \hm^{ 1 } ( A_{ 2 } ) < \infty $. Then, for all $   \epsilon  > 0 $, we find open Lipschitz sets $ U_{ 1 }, U_{ 2 } $ such that $ U_{ 1 } \cap U_{ 2 } = \emptyset $ and
	\begin{equation*}
		{ \hm^{ 1 } ( A_{ 1 } \cup A_{ 2 } )
		-   \epsilon 
		\leq
		 \hm^{ 1 } ( A_{ 1 } \cap U_{ 1 } )
		+
		 \hm^{ 1 } ( A_{ 2 } \cap U_{ 2 } ).}
	\end{equation*}
\end{lemma}
\begin{proof}
	The measure $ \mu \coloneqq  \hm^{ 1 }  \llcorner_{A_{ 1 } \cup A_{ 2 } } $ is a Radon measure. Thus, by inner regularity, we find compact sets $ K_{ 1 } \subseteq A_{ 1 } $, $ K_{ 2 } \subseteq A_{ 2 } \setminus A_{ 1 } $ such that $ \mu ( A_{ 1 } \setminus K_{ 1 } ), \mu ( (A_{ 2 } \setminus A_{ 1 } ) \setminus K_{ 2 } ) \leq   \epsilon /2 $. Since $ K_{ 1 } \cap K_{ 2 } = \emptyset $ and both are compact, we have $ \mathrm{dist} ( K_{ 1 }, K_{ 2 } )  =:  \delta > 0 $. We therefore find open, disjoint Lipschitz sets $ U_{ 1 } $ and $ U_{ 2 } $ such that $ K_{ 1 } \subseteq U_{ 1 } $ and $ K_{ 2 } \subseteq U_{ 2 }$. Then,
	\begin{align*}
		\hm^{ 1 } ( A_{ 1 } \cup A_{ 2 } )
		& =
		\hm^{ 1 } ( A_{ 1 } ) + \hm^{ 1 } ( A_{ 2 } \setminus A_{ 1 } )
		\\
		& \leq \hm^{ 1 } ( A_{ 1 } \cap K_{ 1 } )
		+
		\hm^{ 1 } ( ( A_{ 2 } \setminus A_{ 1 } ) \cap K_{ 2 } )
		+
		  \epsilon  
		\\
		& \leq 
		\hm^{ 1 } ( A_{ 1 } \cap U_{ 1 } )
		+
		\hm^{ 1 } ( ( A_{ 2 } \setminus A_{ 1 } ) \cap U_{ 2 } )
		+
		  \epsilon  ,
	\end{align*}
	which finishes the proof.
\end{proof}

\section{Proofs}\label{sec:appendix_proofs}

 In this section we collect the remaining  proofs that have been omitted in the paper.

\begin{proof}[Proof of Theorem \ref{thm:density_with_boundary_values}]
We first observe that by \cite[Thm.~2.1, Eq.~(2.4)]{friedrich2018-piecewiseKorn} we find a sequence $\tilde{v}_k \in \gsbdtwo ( \Omega ) \cap \lp^2(\Omega;\R^2)$ such that $E_k := \lbrace x\in \Omega\colon \tilde{v}_k(x) \neq v(x)\rbrace$ satisfies $\mathcal{L}^2(E_k) + \mathcal{H}^1(\partial^*E_k) \to 0$ as $k \to \infty$. Therefore,  the sequence 
$$v_k :=  \begin{cases}  \tilde{v}_k \chi_{\Omega \setminus E_k} & \text{ on } \Omega, \\  h & \text{ on } \Omega' \setminus \overline{\Omega} \end{cases} $$
lies in $\gsbdtwo ( \Omega' ) \cap \lp^2(\Omega';\R^2)$ and satisfies
$$v_{k} \to v \text{ in measure on $\Omega'$}, \quad  e( v_{ k} ) \to e( v )  \text{ in $\lp^2(\Omega';\mathbb{R}^{2\times 2}_{\rm sym})$}, \quad \hm^{ 1} \big( J_{ v_{ k } }  \triangle J_{ v } \big) \to 0.  $$
By means of  a diagonal argument this shows that  it suffices to prove the statement for functions $v  \in \gsbdtwo ( \Omega' ) \cap \lp^2(\Omega';\R^2)$. (Here, we use that the convergence in measure is metrizable.) By another diagonal argument it suffices to find a sequence $(w_\delta)_\delta$ as in the statement satisfying items \ref{item:density_measure_conv}, \ref{item:density_l2_convergence_sym_grad} and item \ref{item:density_jump_set_diff_vanishes} is replaced by 
\begin{align}\label{statement1}
\limsup_{\delta \to 0} \hm^{ 1} \big( J_{ w_{ \delta } }  \triangle J_{ v } \big) \le \eta
\end{align}
for an arbitrary $\eta >0$. 

Let us from now on  assume that $v  \in \gsbdtwo ( \Omega' ) \cap \lp^2(\Omega';\R^2)$  with $v = h$ on $\Omega'\setminus \overline{\Omega}$  and fix $\eta>0$. Let $\partial_D \Omega := \partial \Omega \cap \Omega'$, and let $\partial_D^* \Omega \subseteq \partial_D \Omega$ be the set of differentiablity points of $\partial_D \Omega$, where the corresponding outer normal is denoted by $\nu(x)$ for $x \in \partial_D^* \Omega$. By Rademacher's theorem we have $\mathcal{H}^1(\partial_D \Omega \setminus \partial_D^* \Omega)= 0$.   We introduce a fine cover of $\partial^*_D \Omega \setminus J_v$ up to a set of negligible $\mathcal{H}^1$-measure: consider all closed squares $Q_r(x)  \subset \subset  \Omega'$ centered at $x \in \partial_D^* \Omega$   with radius $r$ and two sides parallel to $\nu(x)$ such that 
\begin{align}\label{dens1}
{\rm (i)} & \ \ \mathcal{H}^1\big(  Q_r(x) \cap J_v    \big) \le \eta r, \notag \\
{\rm (ii)} & \ \   \mathcal{H}^1\big(  Q_r(x) \cap \partial_D \Omega   \big) \ge 2 r.
\end{align} 
This indeed provides a fine cover as  $J_v$ has $\mathcal{H}^1$-density zero $\mathcal{H}^1$-a.e.\ in  $\partial_D \Omega^* \setminus J_v$. We now apply the Besicovitch covering theorem  with respect to $\mathcal{H}^1\llcorner_{\partial^*_D \Omega \setminus J_v}$ which induces a finite disjoint collection $(Q_{r_i}(x_i))_{i=1}^N$ with $(x_i)_{i=1}^N \subseteq \partial^*_D \Omega \setminus J_v$, abbreviated by $(Q_i)_{i=1}^N$ in the sequel, such that  \eqref{dens1} holds for each   $Q_i$  and 
\begin{align}\label{dens2} 
\mathcal{H}^1\left(    (\partial^*_D \Omega \setminus J_v) \setminus \bigcup_{i=1}^N  Q_i    \right) \le \eta . 
\end{align}
Following the proof of {\cite[Prop.~2.5]{gia05}}, we show that for each $Q_i$ there exists a sequence $(z_\delta^i)_\delta \subseteq \mathcal{W}(Q_i;\R^2)$ such that
\begin{align}\label{dens3} 
{\rm (i)} & \ \    \Vert z_{ \delta}^i  - v  \Vert_{\lp^2(Q_i)} \to 0 \text{ as $\delta\to 0$},\notag\\
{\rm (ii)} & \ \     \Vert e(z_{ \delta}^i)  \to e(v)  \Vert_{\lp^2(Q_i)} \to 0 \text{ as $\delta\to 0$}, \notag\\
{\rm (iii)} & \ \  \mathcal{H}^1(J_{z_\delta^i}) \lesssim  \eta r_i, \notag\\
{\rm (iv)} &  \ \      z_\delta^i = h \text{ on } Q_i \setminus \overline{\Omega}.
\end{align}
To   this end, we fix $Q_i$ and   drop the index $i$ for convenience, i.e., we write $Q$ as well as $x$ and $r$. We assume without restriction that $\nu(x) = e_2$ and $Q = (-1,1)^2$. We choose a Lipschitz function $f \colon (-1,1) \to \R$  such that
$$\Omega \cap Q = \big\{ y \in Q  \colon  y_2 < f(y_1)     \big\}, \quad    \partial_D \Omega \cap Q  = \big\{ y \in Q  \colon  y_2 = f(y_1)     \big\}. $$
Let $v_\delta(y) := v(y+\delta e_2)$ and observe that
\begin{align}\label{dens3.5} 
\text{  $v_\delta  \to v $ in $\lp^2(Q;\R^2)$ \quad and  \quad  $e(v_\delta)  \to e(v) $ in $\lp^2(Q;\R^{2\times 2}_{\rm sym})$     \quad  as $\delta \to 0$.}
\end{align}
 By \cite[Thm.~3.1]{iurlano_density} (see also {\cite[Thm.~1.1]{Cri18}} for control on higher Sobolev norms) we choose $(w_\delta)_\delta \in \mathcal{W}(Q;\R^2)$ such that 
\begin{align}\label{dens4} 
\Vert  v_\delta  - w_\delta \Vert_{\lp^2(Q)} + \Vert  e(v_\delta)  - e(w_\delta) \Vert_{\lp^2(Q)} \le \delta^2, \quad \quad    \mathcal{H}^1(J_{w_\delta} \cap Q) \lesssim r\eta,
\end{align}
 where we used that $\mathcal{H}^1(J_{v} \cap Q) \le r\eta$, see inequality  \eqref{dens1}(i), and the definition of $v_\delta$ which implies that $ \hm^{ 1 } ( J_{ v_{ \delta } } \cap Q ) \leq \hm^{ 1 } ( J_{ v } \cap Q ) $ due to the regularity of the boundary data. Let $\psi_\delta$ be a cut-off function with $\psi_\delta = 1$ on $\lbrace y_2 \ge f(y_1) - \delta/3\rbrace$,  $\psi_\delta = 0$ on $\lbrace y_2 \le f(y_1) - \delta/2\rbrace$, and $\Vert \nabla \psi_\delta \Vert_\infty \lesssim 1/\delta$. Define $z_\delta := \psi_\delta h + (1-\psi_\delta) w_\delta$. By construction and $h \in \wkp^{ 2 , \infty } ( \Omega'  ;  \mathbb{R}^2  )$  we have $z_\delta \in \mathcal{W}(Q;\R^2)$,  $\mathcal{H}^1(J_{z_\delta} \cap Q) \lesssim r\eta$, and that      $z_\delta = h$ on $Q \setminus \overline{\Omega}$. Thus,  point (iii) and (iv) of  \eqref{dens3} hold.   We estimate
\begin{align}\label{dens4-5}
 \Vert  z_\delta  - v \Vert_{\lp^2(Q)} \le  \Vert  (\psi_\delta h + (1-\psi_\delta) v)  - v \Vert_{\lp^2(Q)}  + \Vert  v_\delta  - w_\delta \Vert_{\lp^2(Q)} + \Vert  v_\delta  - v \Vert_{\lp^2(Q)}.
 \end{align}
This shows $z_\delta  \to v $ in $\lp^2(Q;\R^2)$ as $\delta \to 0$ by \eqref{dens3.5}--\eqref{dens4},   the fact that  $\mathcal{L}^2({\rm supp}\, \psi_\delta \cap \Omega)\to 0$, and $v = h$ on $\Omega' \setminus \overline{\Omega}$.   Thus, it remains to check \eqref{dens3}(ii). We note that  $e( z_\delta)  =    \psi_\delta e(h) + (1-\psi_\delta)   e(w_\delta) +   ( h - w_\delta ) \odot \nabla \psi_\delta   $. 
By an argument similar to \eqref{dens4-5} employing  \eqref{dens3.5}--\eqref{dens4}, we get  $\Vert   \psi_\delta e(h) + (1-\psi_\delta)   e(w_\delta) - e(v) \Vert_{\lp^2(Q)} \to 0$. On the other hand, setting $\Psi_\delta: =\lbrace 0 <  \psi_\delta < 1\rbrace$, by  the first inequality in  \eqref{dens4}  and the Lipschitz continuity of $h$, we get
\begin{align*}
\Vert (h - w_\delta) \odot \nabla \psi_\delta  \Vert_{\lp^2(Q)} &\lesssim \delta^{-1} \Vert w_\delta  - h \Vert_{\lp^2( \Psi_\delta )}  
\\
& \lesssim \delta^{-1} \Vert v_\delta  - h \Vert_{\lp^2( \Psi_\delta )} + \delta 
\\ &  = \delta^{-1} \Vert v(\cdot+\delta e_2)  - h \Vert_{\lp^2( \Psi_\delta )} + \delta 
\\
&\lesssim  \Vert \nabla h \Vert_\infty \mathcal{L}^2( \Psi_\delta)^{ 1/2 } + \delta.
\end{align*}
where we used that  $v(\cdot+\delta e_2) =  h(\cdot + \delta e_2)$ on $\Psi_\delta$. Since this vanishes as $\delta \to 0$,  \eqref{dens3}(ii) follows.

After having found the functions $z^i_\delta$ satisfying \eqref{dens3}, we are now ready to construct the sequence $(w_\delta)_\delta$. To this end, we choose slightly smaller squares $Q'_i  \subset \subset  Q_i$ for $i =1,\ldots, N$ such that 
\begin{align}\label{dens6} 
\mathcal{H}^1\left( \partial_D \Omega  \cap \bigcup_{i=1}^N (Q_i \setminus Q_i')   \right) \le \eta
\end{align} 
 and we choose cut-off functions $(\varphi_i)_{i=1}^N \subseteq C^\infty(\Omega)$  such that $\varphi_i = 1$ on $Q'_i \cap \Omega$ and $\varphi_i= 0$ on $\Omega \setminus Q_i$. Eventually, we let $\varphi_0 := 1 - \sum_{i=1}^N \varphi_i$. By \cite[Thm.~3.1]{iurlano_density} (see also \cite[Thm.~1.1]{Cri18} for control on higher Sobolev norms)  we find a sequence $(z^0_\delta)_\delta \subseteq \mathcal{W}(\Omega;\R^2)$ with 
\begin{align}\label{dens7} 
\Vert  z^0_\delta  - v \Vert_{\lp^2(\Omega)} + \Vert  e(z^0_\delta)  - e(v) \Vert_{\lp^2(\Omega)} +  \mathcal{H}^1\big( (J_{z^0_\delta} \triangle  J_v)   \cap \Omega  \big) \to 0 \quad \text{as $\delta \to 0$}.   
\end{align} 
We define $w_\delta \in \mathcal{W}(\Omega';\R^2)$ by
$$w_\delta  = \begin{cases}  \sum_{i=0}^N \varphi_i z_\delta^i & \text{ on } \Omega, \\  h & \text{ on } \Omega' \setminus \overline{\Omega}. \end{cases}$$
By construction and by  \eqref{dens3}(iv) we find $ J_{w_\delta}  \cap Q_i' = J_{z^i_\delta}  \cap Q_i'$ for all $i=1,\ldots, N$. {Furthermore, by adding a vanishingly small constant to $z^0_\delta$, we may ensure $J_{w_\delta}  \supseteq  \partial_D \Omega \setminus \bigcup_i Q_i$.} Therefore, we get 
\begin{align*}
\mathcal{H}^1\big(J_{w_\delta} \triangle J_v\big) \le  \mathcal{H}^1\big((J_{z_\delta^0} \triangle J_v)  \cap \Omega  \big) + \sum_{i=1}^N {\left(\mathcal{H}^1(J_{z_\delta^i} \cap Q_i)+\mathcal{H}^1(J_{v} \cap Q_i) \right)} + \mathcal{H}^1\left( (\partial_D \Omega   \setminus J_v) \setminus \bigcup\nolimits_i Q_i'\right). 
\end{align*}
Thus, combining  {\eqref{dens1}(i)}, \eqref{dens2}, \eqref{dens3}(iii),  \eqref{dens6}, and  \eqref{dens7}  we find
$$\limsup_{\delta \to 0} \mathcal{H}^1\big(J_{w_\delta} \triangle J_v\big) \lesssim \eta + \eta \sum_{i=1}^N  r_i \lesssim \eta, $$
where in the last step we also used \eqref{dens1}(ii) and the fact that $\mathcal{H}^1(\partial_D \Omega ) < \infty$. This shows \eqref{statement1}. Combining \eqref{dens3}(i)  and  \eqref{dens7}, we find \ref{item:density_measure_conv}.  Eventually, since $e(w_\delta) =  \sum_{i=0}^N  \varphi_i e(z_\delta^i) +  \nabla \varphi_i  \odot z_\delta^i$ on $\Omega$ and $e(v)$ can be written as  $e(v) =  \sum_{i=0}^N  \varphi_i e(v) +  \nabla \varphi_i  \odot v $,  by using   \eqref{dens3}(i), (ii)  and  \eqref{dens7} we conclude \ref{item:density_l2_convergence_sym_grad}.  
\end{proof}

\begin{proof}[Proof of the upper estimate in \Cref{lemma:energy_balance_inequalities}]
	We prove the upper semicontinuity estimate (\ref{eq:energy_balance_usc}). As we will show, this is a consequence of $ u $ being a minimizer with respect to its own jump set, see \Cref{lemma:u_minimzer_of_total_energy}. Let $ 0 \leq s \leq t  \text{ with }s,t \in I_{ \infty } $ and, for $\eps$ sufficiently small, suppose that $n=N(\eps)$ is such that $t = t_{N(\eps)}^\eps$. Then, the function $ u( t ) + h( s ) - h( t ) $ is an admissible competitor of the minimization problem (\ref{eq:u_min_of_total_energy}) at time $ s $. Thus,
	\begin{align*}
		& \int_{ {\Omega'} }
		\dfrac{ 1 }{ 2 }
		Q \big( e ( u( s ) ) \big) \dd{ x }
		\\
		\leq{} & 
		\int_{ {\Omega'} }
		\dfrac{ 1 }{ 2 } Q \big( e ( u( t ) ) \big)
		+
		\mathbb{ C } e ( u( t ) ) \colon { e ( h( s ) - h( t ) ) }
		+
		\dfrac{ 1 }{ 2 }
		Q \big( e ( h( s ) - h ( t ) ) \big)
		\dd{ x }
		+
		\kappa \hm^{ 1 } \left(
		J_{ u( t ) } 
		\setminus 
		\Gamma ( s )
		\right),
	\end{align*}
	which is equivalent to
	\begin{align*}
		\energyLinTot ( s ) \leq &\energyLinTot ( t )
		+
		\int_{ {\Omega'} }
		\mathbb{ C } e( u( t ) ) \colon { e ( h( s ) - h( t ) ) } 
		+
		\dfrac{1}{ 2 }
		Q \big( e ( h( s ) - h( t ) ) \big)
		\dd{ x }
		\\
		&+
		\kappa \left(
		\hm^{ 1 } \left( \Gamma ( s ) \right)
		+
		\hm^{ 1 } \left(
		J_{ u ( t ) } \setminus 
		\Gamma ( s )
		\right)
		-
		\hm^{ 1 } \left(
		\Gamma ( t )
		\right)
		\right).
	\end{align*}
	Note that the term in the brackets is nonpositive since $ \Gamma( s ) \subseteq \Gamma ( t) $  and $J_{u(t)} \subseteq \Gamma ( t) $. 
	Thus, since $ h $ is continuous in time with respect to  $ \norm{\cdot}_{ \wkp^{ 2 , 2} } $,   we find that $ \energyLinTot $ has no negative jumps as claimed in the  lemma.   Moreover, letting $ s= t_{N( \varepsilon ) -1}^\eps $ and iterating the above estimate (and re-using the variable~$s$), we get that for all $ \varepsilon> 0 $ sufficiently small
	\begin{align*}
		\energyLinTot( t )
		\geq
		\energyLinTot ( 0 )
		+
		\sum_{ k = 0 }^{ N( \varepsilon ) - 1 }
		\int_{ t_{ k }^{ \varepsilon } }^{ t_{ k + 1 }^{ \varepsilon } }
		\int_{ {\Omega'} }
		\mathbb{ C } e ( u( t_{k + 1 }^{ \varepsilon } ) ) \colon { e ( \partial_t h( s ) ) } 
		\dd{ x }
		\dd{ s }
		- \omega ( \Delta_{ \varepsilon } ) \int_{ 0 }^{ t } \norm{ \nabla \partial_t h }_{ \lp^{ 2 } }
		\dd{ s }
	\end{align*}
	for a function $ \omega $ which vanishes as $ \Delta_{\varepsilon} $ tends to zero as  in \eqref{eq:a_priori_estimate_iterated}.  Taking the limit superior gives the desired inequality (\ref{eq:energy_balance_usc}).
\end{proof}

\begin{proof}[Proof of Minimality and Jump-set Inclusion of \Cref{lemma:cont_extension}]
	The bound on $ \hm^{ 1 } ( \Gamma ( t ) ) $ is an immediate consequence of the lower semicontinuity \eqref{eq:lsc_jump_sets_union_times}. The fact that $u(t) = h(t)$ on $\Omega' \setminus \overline{\Omega}$ for all $t \in [0,1]$ follows from Definition \ref{def:displacement II}. 
	
	\textit{Minimality.} For $t \in I_\infty$ the minimality follows from Lemma \ref{lemma:u_minimzer_of_total_energy}.	Let us now prove the minimality (\ref{eq:minimality_u_whole_interval}) for $ t \in [0,1]\setminus I_{ \infty } $. We know by \Cref{lemma:u_minimzer_of_total_energy} that for all $t_p\in I_\infty$ (used to define $ u ( t ) $ in Definition~\ref{def:displacement II}) and all $ v \in \gsbdtwo (\Omega') $ with $ v = h( t_{p } ) $ on $ \Omega' \setminus \overline{\Omega } $ it holds that
	\begin{equation*}
		\int_{ \Omega' }
		\dfrac{ 1 }{ 2 }
		Q \big( e ( u ( t_{ p } ) ) \big)
		\dd{ x }
		\leq
		\int_{ \Omega' }
		\dfrac{ 1 }{ 2 }
		Q ( e ( v ) ) 
		\dd{ x }
		+
		\kappa 
		\hm^{ 1 } \left(
		J_{ v } \setminus \Gamma ( t_{ p } ) \right).
	\end{equation*}
	Take $ w \in \gsbdtwo (\Omega') $ with $ w = h ( t ) $ on $ \Omega' \setminus \overline{\Omega } $ and test the above inequality with $ w + h ( t_{ p } ) - h ( t ) $ to obtain
	\begin{equation*}
		\int_{ \Omega' }
		\dfrac{ 1 }{ 2 }
		Q \big( e ( u ( t_{ p } ) ) \big)
		\dd{ x }
		\leq
		\int_{ \Omega' }
		\dfrac{ 1 }{ 2 }
		Q \big( e ( w + h ( t_{ p } ) - h( t ) ) \big)
		\dd{ x }
		+
		\kappa 
		\hm^{ 1 } \left(
		J_{ w } \setminus \Gamma ( t_{ p } )
		\right).
	\end{equation*}
	By weak convergence, in the limit $p \to \infty$, the left-hand side can be bounded from below by the integral $ \int_{\Omega'} \tfrac12 Q ( e ( u ( t ) ) )\dd{x}$. The integral on the right-hand side converges to $ \int_{\Omega'}\tfrac12 Q ( e ( w ) ) \dd{x}$, and since $ t \notin I_{ \infty } $, we have by continuity from above of $ \hm^{ 1 } $ that the measure of the jump set converges to $ \hm^{ 1 } ( J_{ w } \setminus \Gamma ( t ) ) $.  This  proves the desired minimality.
	
	\textit{Jump-set inclusion.}	Next, we argue that the inclusion (\ref{eq:jump_contained_in_gamma}) holds. If this were not true, then we can find a point $ x\in J_{ u ( t ) }\setminus \Gamma( t ) $ that has $1$-dimensional density $ 1$  with respect to the set $ J_{ u ( t ) } $ and density $0$ with respect to $\Gamma( t )$ (see \cite[Thm.~2.83 and (2.41)]{ambrosio_fusco_pallara_functions_of_bv_and_free_discontinuity_problems}). Thus, for some small fixed $ r >0$, we have
	\begin{equation*}
		\hm^{ 1 } \left(
		B_{ r } ( x ) \cap J_{ u ( t ) } 
		\right) \geq  r /2
		\quad
		\text{and}
		\quad
		\hm^{ 1 } \left(
		B_{ r } ( x ) \cap \Gamma( t ) 
		\right)
		<
		r/2.
	\end{equation*}
	Since $ J_{ u ( t_{ p } ) } \subseteq \Gamma ( t ) $, we know that
	\begin{equation*}
		\hm^{ 1 } \left( B_{ r } ( x ) \cap J_{ u( t_{ p } ) } \right)
		<
		r /2.
	\end{equation*}
	Applying the lower semicontinuity of jump sets as $p\to \infty$,  see \Cref{thm:gsbd2_compactness}\ref{item:lscsurface} for  $U = B_{ r } ( x )$,  leads to a contradiction.
\end{proof}

\printbibliography






\end{document}